\documentclass{elsarticle}
\makeatletter
\def\ps@pprintTitle{%
 \let\@oddhead\@empty
 \let\@evenhead\@empty
 \def\@oddfoot{\centerline{\thepage}}%
 \let\@evenfoot\@oddfoot}
\makeatother

\usepackage{a4wide}
\usepackage[utf8]{inputenc}
\usepackage{amsmath}
\usepackage[colorlinks=true]{hyperref}
\usepackage{amsthm}
\usepackage{amsfonts}
\usepackage{graphicx}
\usepackage{hyphenat}
\usepackage{enumitem}
\usepackage{amsfonts}
\newcommand{\overbar}[1]{\mkern 1.5mu\overline{\mkern-1.5mu#1\mkern-1.5mu}\mkern 1.5mu}
\usepackage{array}

\makeatletter

\def\Forb{\mathop{\mathrm{Forb}}\nolimits}
\def\Age{\mathop{\mathrm{Age}}\nolimits}

\def\Rel{\mathop{\mathrm{Rel}}\nolimits}

\def\dom{\mathop{\mathrm{Dom}}\nolimits}
\def\range{\mathop{\mathrm{Range}}\nolimits}
\def\Aut{\mathop{\mathrm{Aut}}\nolimits}

\def\tp{\mathop{\mathrm{tp}}\nolimits}
\def\Flim{\mathop{\mathrm{Flim}}\nolimits}
\def\Nexp{\mathop{\mathrm{Nexp}}\nolimits}
\def\Nemb{\mathop{\mathrm{Nemb}}\nolimits}
\def\sse{\subseteq}
\newcommand{\expand}[2][+]{({#2}*{#2}^{#1})}

\def\str#1{\mathbf {#1}}

\def\arity#1{a(\rel{}{#1})}

\def\rel#1#2{R_{\mathbf{#1}}^{#2}}

\def\F{{\mathcal F}}
\def\K{{\mathcal K}}
\def\Fraisse{Fra\"{\i}ss\' e}

\theoremstyle{definition}
\newtheorem{defn}{Definition}[section]
\newtheorem*{example}{Example}
\newtheorem{remark*}{Remark}
\newtheorem{remark}{Remark}[section]
\theoremstyle{remark}

\theoremstyle{plain}
\newtheorem{thm}{Theorem}[section]
\newtheorem{corollary}[thm]{Corollary}
\newtheorem{prop}[thm]{Proposition}
\newtheorem{observation}[thm]{Observation}
\newtheorem{lem}[thm]{Lemma} 
\newtheorem{conjecture}[thm]{Conjecture}
\newtheorem{question}[thm]{Question}

\def\Ind#1#2{#1\setbox0=\hbox{$#1x$}\kern\wd0\hbox to 0pt{\hss$#1\mid$\hss}
\lower.9\ht0\hbox to 0pt{\hss$#1\smile$\hss}\kern\wd0}
\def\Notind#1#2{#1\setbox0=\hbox{$#1x$}\kern\wd0\hbox to
0pt{\mathchardef\nn="0236\hss$#1\nn$\kern1.4\wd0\hss}\hbox 
to 0pt{\hss$#1\mid$\hss}\lower.9\ht0
\hbox to 0pt{\hss$#1\smile$\hss}\kern\wd0}
\def\ind{\mathop{\mathpalette\Ind{}}}

\newcommand{\struc}[1]{\langle #1\rangle}

\begin{document}
\bibliographystyle{alpha}

\begin{frontmatter}
\title{Ramsey expansions of metrically homogeneous graphs}
\author[KAM]{Andr\'es Aranda}
\ead{andres.aranda@gmail.com}
\author[Ge3,IM]{David Bradley-Williams\fnref{g0}}
\ead{williams@math.cas.cz}
\fntext[g0]{In the final stages of this project was partially supported by the research training group GRK 2240: Algebro-Geometric Methods in Algebra, Arithmetic and Topology, funded by the German Science Foundation (DFG), was further supported by project EXPRO 20-31529X of the Czech Science Foundation (GA\v CR) and by the Czech Academy of Sciences CAS (RVO 67985840).}
\author[KAM]{Jan Hubi\v{c}ka\fnref{g1}}
\ead{hubicka@kam.mff.cuni.cz}
\fntext[g1]{In the final stages of this project supported by project 21-10775S of  the  Czech  Science Foundation (GA\v CR) and by a project that has received funding from the European Research Council (ERC) under the European Union's Horizon 2020 research and innovation programme (grant agreement No 810115, DYNASNET).}
\author[Gr]{Miltiadis Karamanlis}
\ead{kararemilt@gmail.com}
\author[KA]{Michael Kompatscher\fnref{g2}}
\ead{kompatscher@karlin.mff.cuni.cz}
\fntext[g2]{In the final stages of this project supported by a project that has received funding from the European Research Council (ERC) under the European Unions Horizon 2020 research and innovation programme (grant agreement No. 771005, CoCoSym).}
\author[KAM,DD]{Mat\v ej Kone\v cn\'y\fnref{g3}}
\ead{matej.konecny@tu-dresden.de}
\fntext[g3]{In the final stages of this project supported by project 21-10775S of  the  Czech  Science Foundation (GA\v CR) and by the European Research Council (Project POCOCOP, ERC Synergy Grant 101071674). Views and opinions expressed are however those of the authors only and do not necessarily reflect those of the European Union or the European Research Council Executive Agency. Neither the European Union nor the granting authority can be held responsible for them.}

\author[Ca]{Micheal Pawliuk}
\ead{m.pawliuk@utoronto.ca}

\address[KAM]{Department of Applied Mathematics (KAM), Faculty of Mathematics and Physics, Charles University, Prague, Czech Republic}
\address[Ge3]{Mathematisches Institut der Heinrich-Heine-Universit\"at, D\"usseldorf, Germany}
\address[IM]{Institute of Mathematics, Czech Academy of Sciences, Prague, Czech Republic}
\address[Gr]{Department of Mathematics, National and Kapodistrian University of Athens, Greece}
\address[KA]{Department of Algebra, Faculty of Mathematics and Physics, Charles University, Prague, Czech Republic}
\address[DD]{Institute of Algebra, TU Dresden, Dresden, Germany}
\address[Ca]{University of Toronto Mississauga, Toronto, Canada}

\begin{abstract}
We investigate Ramsey expansions, the coherent extension property for partial isometries (EPPA), and the existence of a stationary independence relation for all classes of
metrically homogeneous graphs from Cherlin's catalogue. We show that, with the exception of tree-like
graphs, all metric spaces in the catalogue have precompact Ramsey expansions
(or lifts) with the expansion property. With two exceptions we can also
characterise the existence of a stationary independence relation and coherent
EPPA.

Our results are a contribution to Ne\v set\v ril's
classification programme of Ramsey classes and can be seen as empirical evidence of the
recent convergence in techniques employed to establish the Ramsey property, the expansion property, EPPA and the existence of a stationary independence relation. At the heart of our proof is a canonical way of completing edge-labelled graphs to metric spaces in Cherlin's classes. The existence of such a ``completion algorithm'' then allows us to apply several strong results in the areas that imply EPPA or the Ramsey property.

The main results have numerous consequences for the automorphism groups of the
\Fraisse{} limits of the classes. As corollaries, we prove amenability, unique ergodicity,
existence of universal minimal flows, ample generics, small index property,
21-Bergman property and Serre's property (FA).
\end{abstract}
\end{frontmatter}

\clearpage
\tableofcontents
\clearpage

\section{Introduction}
Given a graph, its {\em associated
metric space} has its vertices as points, and the distance between two points is the length of the shortest path connecting them.
A (countable) structure is {\em homogeneous}  if
every isomorphism between finite substructures is induced by an automorphism.
A {\em metrically homogeneous graph} is a connected graph with the property
that the associated metric space is a homogeneous metric space.

As a new contribution to the well known classification programme of homogeneous structures (of Lachlan and Cherlin~\cite{Lachlan1984,Lachlan1980,Cherlin1998}),
Cherlin~\cite{Cherlin2013} has recently provided a catalogue of metrically homogeneous
graphs. While originally this catalogue was only conjectured to be complete, there is now a purported (yet unpublished) proof.\footnote{This is claimed on Cherlin's website: https://sites.math.rutgers.edu/\~{}cherlin/Paper/inprep.html}

This paper investigates properties of classes of finite substructures of structures from Cherlin's
catalogue.  To state our main results we introduce some terminology first.
A graph is {\em bipartite} if it contains no odd cycles and it is \emph{regular} if there exists $0\leq k\leq \infty$ such that the degree of every vertex is $k$. The {\em diameter} of a graph is the maximal distance
in its associated metric space. A graph of diameter $\delta$ is {\em antipodal} if in its associated metric space for every vertex there exists at most one vertex in distance $\delta$ (and it is {\em antipodally closed} if every vertex has a unique such vertex).
A graph is {\em tree-like} if it is isomorphic to a graph $T_{m,n}$, $2 \leq m,n\leq \infty$, defined as follows:
\begin{defn}
\label{defn:tree-like}
Given $2 \leq m,n\leq \infty$, the graph $T_{m,n}$ is defined to be the (regular) graph in
which the blocks (two-connected components) are cliques of order $n$ and every
vertex is a cut vertex, lying in precisely $m$ blocks.
\end{defn}

Our main results can be summarised as follows (all the remaining notions will be introduced below):
\begin{thm}
\label{thm:ramseyall}
Let $\Gamma$ be the associated metric space of a countably infinite metrically homogeneous graph $\str{G}$ of diameter $\delta$ from Cherlin's catalogue. 
Then one of the following applies:
\begin{enumerate}
	\item If $\str{G}$ is a tree-like graph, then $\Age(\Gamma)$ has no precompact Ramsey expansion. 
	\item If $\str{G}$ is not tree-like, then one of the following holds:
		\begin{enumerate}
			\item If $\str{G}$ is primitive (i.e. neither antipodal nor bipartite), then the class of free orderings of $\Age(\Gamma)$ is Ramsey.
			\item If $\str{G}$ is bipartite and not antipodal then $\Age(\Gamma)$ is Ramsey when extended by a unary predicate denoting the bipartition and by convex linear orderings.
			\item If $\str{G}$ is antipodal, denote by $\mathcal A$ the subclass of $\Age(\Gamma)$ of antipodally closed metric spaces. Then:
					\begin{enumerate}
						\item If $\str{G}$ is not bipartite or has odd diameter, then $\mathcal A$ is Ramsey when extended by a linear ordering convex with respect to the podes.
						\item If $\str{G}$ is bipartite of even diameter, then $\mathcal A$ is Ramsey when extended by a unary predicate denoting the bipartition and a linear ordering convex with respect to the bipartition and the podes.
					\end{enumerate}
					In both these cases, the linear order on one pode in $\mathcal A$ uniquely extends to the other pode so that it is (linearly) isomorphic. For precise definitions and discussion about non-antipodally closed metric spaces, see Section~\ref{sec:antipodal}.
		\end{enumerate}
\end{enumerate}
By ``a predicate denoting the bipartition'' we mean a unary predicate defining two
equivalence classes of vertices, such that all distances in the same equivalence
class are even.

All the Ramsey expansions above have the expansion property.
\end{thm}
We thus completely characterise Ramsey expansions of all currently known infinite metrically homogeneous graphs. In addition to that
we show two related properties, arriving at an almost complete characterisation.

\begin{thm}
\label{thm:EPPAmain}
Let $\Gamma$ be the associated metric space of a countably infinite metrically homogeneous graph $\str{G}$ from Cherlin's catalogue.
\begin{enumerate}
 \item If $\str{G}$ is a tree-like graph, then $\Age(\Gamma)$ does not have EPPA.
 \item If $\str{G}$ is not tree-like then:
 \begin{enumerate}
   \item If $\str{G}$ is antipodal, and:
   \begin{enumerate}
     \item If $\str{G}$ is non-bipartite of even diameter or bipartite of odd diameter, then $\Age(\Gamma)$ has coherent EPPA.
     \item If $\str{G}$ is bipartite of even diameter or non-bipartite of odd diameter, then $\Age(\Gamma)$ extended by a unary predicate denoting podality has coherent EPPA.
   \end{enumerate}
   \item If $\str{G}$ is not antipodal then it has coherent EPPA.
  \end{enumerate}
\end{enumerate}
By ``a predicate denoting podality'' we mean a unary predicate defining two
equivalence classes of vertices, such that no pair of vertices in distance $\delta$ are
in the same equivalence class.
\end{thm}

In the submitted version of this paper, we asked whether the statement of Theorem~\ref{thm:EPPAmain} is best possible, that is, whether one needs to extend the even-diameter bipartite and odd-diameter non-bipartite antipodal classes by a unary predicates denoting podality (it was stated as Problem~1.3). Since then this was resolved by Evans, Hubička, Konečný and Nešetřil~\cite{eppatwographs} who proved that the predicated are not necessary for the diameter 3 non-bipartite case and subsequently by Konečný~\cite{Konecny2019a} who extended their techniques to all classes in question.

\begin{thm}
\label{thm:SIRall}
Let $\Gamma$ be the associated metric space of a countably infinite metrically homogeneous graph $\str{G}$ from Cherlin's catalogue. 
\begin{enumerate}
 \item If $\str{G}$ is isomorphic to $T_{m,n}$ then $\Gamma$ has no stationary independence relation and it has a local stationary independence relation if and only if $m = \infty$ and $n \in \{2,3,\infty\}$. 
 \item If $\str{G}$ is not tree-like then:
 \begin{enumerate}
   \item If $\str{G}$ is antipodal, and:
   \begin{enumerate}
     \item If $\str{G}$ is not bipartite and has even diameter, then there exists a stationary independence relation on $\Gamma$.
     \item If $\str{G}$ is bipartite and has odd diameter, then there exists a local stationary independence relation, but there is no stationary independence relation on $\Gamma$.
     \item Otherwise there is no local stationary independence relation on $\Gamma$.
   \end{enumerate}
   \item If $\str{G}$ is not antipodal:
   \begin{enumerate}
     \item If $\str{G}$ is not bipartite, then there exists a stationary independence relation on $\Gamma$.
     \item If $\str{G}$ is bipartite, then there exists a local stationary independence relation, but there is no stationary independence relation on $\Gamma$.
   \end{enumerate}
  \end{enumerate}
\end{enumerate}
\end{thm}
We remark that for structures with closures it makes sense to consider an
additional axiom as discussed by Evans, Ghadernezhad and Tent~\cite[Definition 2.2]{Evans2016}.
Due to the simple nature of closures considered in this paper this does not make a practical
difference in our results.

\subsection{Edge-labelled graphs and metric completions}
\label{sec:shortestpathsec}
Our arguments are based on an analysis of an algorithm to fill holes in incomplete
metric spaces presented in Section~\ref{sec:basic3}. This algorithm works
for a significant (namely the primitive) part of Cherlin's catalogue and leads to a strong
notion of ``canonical amalgamation''. When constraints on the amalgamation
classes are pushed to the extreme, new phenomena (such as antipodality or
bipartiteness) appear. We show that even in this case our algorithm is useful,
provided that we apply some additional techniques. First, we introduce the
necessary notation to speak about metric spaces with holes.
\begin{defn}
\label{ref:graphs}
An {\em edge-labelled graph} $\str{G}$ is a pair $(G,d)$ where $G$ is the {\em vertex
set} and $d$ is a partial function from $G^2$ to $\mathbb N$ such that $d(u,v)=0$ if and only if $u=v$, and $d(u,v)=d(v,u)$ whenever either number is defined. A pair of vertices $u,v$ on which $d(u,v)$ is defined is called an {\em edge} of $\str{G}$. We also call $d(u,v)$ the {\em length of the edge} $u,v$.
\end{defn}
The standard graph-theoretic notions of homomorphism, embedding, and isomorphism extend naturally to edge-labelled graphs.
Our subgraphs will always be induced (details are given in Section~\ref{sec:perliminaries}).
An edge-labelled graph can also be seen as a relational structure (as discussed
in Section~\ref{sec:reldistance}). We find it convenient
to use notation that resembles the standard notation of metric spaces.

We denote by $\mathcal G^\infty$ the class of all finite edge-labelled graphs and
by $\mathcal G^\delta$ the subclass of $\mathcal G^\infty$ of those graphs containing
no edge of length greater than $\delta$.

An (edge-labelled) graph $\str{G}$ is {\em complete} if every pair of vertices
forms an edge; a complete edge-labelled graph $\str{G}$ is called a {\em metric space} if the {\em triangle inequality} holds,
that is $d(u,w)\leq d(u,v)+d(v,w)$ for every $u,v,w\in G$.
An edge-labelled graph $\str{G}=(G,d)$ is {\em metric} if there exists a metric
space $\str{M}=(G,d')$ such that $d(u,v)=d'(u,v)$ for every edge $u,v$ of $\str{G}$.
Such a metric space $\str{M}$ is also called a {\em (strong) metric completion} of
$\str{G}$.

Given $\str{G}=(G,d)\in \mathcal G^\infty$ the {\em path distance $d^+(u,v)$ of
$u$ and $v$} is the minimum $$\ell=\sum_{1\leq i\leq n-1}d(u_i,u_{i+1})$$ taken over
all possible sequences of vertices for which $u_1=u,u_2,\dots u_n=v$ and
$d(u_i,u_{i+1})$ is defined for every $i\leq n-1$. If there is no such sequence
we put $\ell=\infty$. It is a well-known fact that a connected edge-labelled graph $\str{G}=(G,d)$
is metric if and only if $d(u,v)=d^+(u,v)$ for every edge of $\str{G}$. In this
case $(G,d^+)$ is a metric completion of $\str{G}$ which we refer to as the
{\em shortest path completion}.  This completion algorithm also leads
to an easy characterisation of metric graphs: $\str{G}$ is metric if
and only if it does not contain a {\em non-metric cycle}, by which we mean an edge-labelled graph
corresponding to a graph-theoretic cycle such that one distance in the cycle is greater than sum of
the remaining distances. See e.g.~\cite{Hubicka2016} for details.

Our algorithm is a generalisation of the shortest path
completion algorithm, tailored to preserve some other properties (e.g., forbidding long
cycles whereas the shortest path completion permits forbidding only short cycles).

\subsection{Extension property for partial automorphisms (EPPA)}
\label{subsec:EPPA}
A {\em partial automorphism} of the structure $\str{A}$
is an isomorphism $f\colon \str{B} \to \str{B}'$ where $\str{B}$ and $\str{B}'$ are substructures of $\str{A}$.  
We say that a class of finite structures $\K$  has the {\em extension property for
partial automorphisms}  ({\em EPPA}, sometimes called the {\em Hrushovski extension
property} or in the context of metric spaces the {\em extension property for partial isometries}) if for all $\str{A} \in \K$ there is $\str{B} \in \K$ such that
$\str{A}$ is a substructure of $\str{B}$ and 
every partial automorphism of $\str{A}$ extends to an automorphism of $\str{B}$.
We call $\str{B}$ with such a property an {\em EPPA-witness of $\str{A}$}

In addition to being a non-trivial and beautiful combinatorial property,
classes with EPPA have further interesting properties. For example, Kechris and
Rosendal~\cite{Kechris2007} have shown that the automorphism groups of their
\Fraisse{} limits are amenable.

In 1992 Hrushovski~\cite{hrushovski1992} showed that the class $\mathcal G$ of
all finite graphs has EPPA.  A simple combinatorial argument for Hrushovski's result
was given by Herwig and Lascar~\cite{herwig2000} along with a non-trivial
strengthening for certain, more restricted, classes of structures described by
forbidden homomorphisms. This result was independently used by
Solecki~\cite{solecki2005} and Vershik~\cite{vershik2008} to prove EPPA for the
class of all finite metric spaces with integer, rational or real distances. (Because EPPA is a property of finite structures, there is no difference between integer and rational distances. The techniques also naturally extend to real numbers. For our presentation we will consider integer distances only.)
These results were further strengthened by Rosendal~\cite{rosendal2011,rosendal2011b}.
Recently Conant developed this argument to generalised metric spaces~\cite{Conant2015}, where the distances are elements of a certain classes of distance monoid. Siniora and
Solecki~\cite{Siniora,Siniora2} introduced the stronger notion of \emph{coherent EPPA} (where the
extensions to automorphisms compose whenever the partial automorphisms do, see
Definition~\ref{defn:coherent}), and proved a coherent strengthening of the Herwig--Lascar theorem.\footnote{Since the submission of the paper, there has been a lot of development. For example, Hubička, Konečný, and Nešetřil gave a self-contained combinatorial proof of Solecki and Vershik's result~\cite{Hubicka2018metricEPPA}, and strengthened the Herwig--Lascar and the Siniora--Solecki results~\cite{Hubicka2018EPPA}.}

Hrushovski's result was a key ingredient for a paper by Hodges, Hodkinson, Lascar and Shelah~\cite{hodges1993b} which proved the small index property for the random graph. This line of research has since expanded and the concept of \emph{ample generics} (which follows from the combination of EPPA and the amalgamation property for
automorphisms, where the automorphism group of the amalgam is the same as the automorphism group
of the free amalgam) has been isolated. This is further outlined in Section~\ref{sec:ample}.

\subsection{Ramsey classes}
\label{subsec:ramsey}
The notion of Ramsey classes was isolated in the 1970s and, being a strong
combinatorial property, it has found numerous applications, for example in
topological dynamics~\cite{Kechris2005}.  It was independently proved by Ne\v set\v ril--R\"odl~\cite{Nevsetvril1976} and
Abramson--Harrington~\cite{Abramson1978} that the class of all finite linearly
ordered hypergraphs is a Ramsey class. Several new classes followed. 

For structures $\str{A},\str{B}$ denote the set
of all substructures of $\str{B}$ that are isomorphic to $\str{A}$ by ${\str{B}\choose \str{A}}$. 
\begin{defn}
A class $\mathcal C$ of structures is a \emph{Ramsey class} if for every two objects $\str{A}$ and
$\str{B}$ in $\mathcal C$ and for every positive integer $k$ there exists a
structure $\str{C}$ in $\mathcal C$ such that the following holds: For every
partition of ${\str{C}\choose \str{A}}$ into $k$ classes there exists a
$\widetilde{\str B} \in {\str{C}\choose \str{B}}$ such that
${\widetilde{\str{B}}\choose \str{A}}$ is contained in a single class of the partition. In short we then write
$$
\str{C} \longrightarrow (\str{B})^{\str{A}}_k.
$$
\end{defn}

We now briefly outline the results related to Ramsey classes of metric spaces, see~\cite{Nevsetvril1995,NVT14,Bodirsky2015,Hubicka2016,hubicka2025twenty} for further references.

It was observed by Ne\v set\v ril~\cite{Nevsetvril1989a} that under mild
assumptions every Ramsey class is an amalgamation class. The Ne\v set\v ril
classification programme of Ramsey classes~\cite{Nevsetril2005} asks which 
amalgamation classes provided by the classification programme of homogeneous
structures yield a Ramsey class. Amalgamation classes are often not Ramsey for
simple reasons, such as the lack of a linear order on the vertices of its structures which is known
to be present in every Ramsey class~(\cite{Kechris2005}, see e.g. \cite[Proposition 2.22]{Bodirsky2015}). For this reason we consider enriched classes,
where the language of an amalgamation class is expanded by additional relations. We refer to these classes
as {\em Ramsey expansions} (or {\em lifts}), see Section~\ref{sec:ramseyexpansion}.

In 2005 Ne\v set\v ril \cite{Nevsetvril2007} showed that the class of all
finite metric spaces is a Ramsey class when enriched by free linear ordering
of the vertices (see also~\cite{masulovic2016pre} for an alternative proof). This result was extended to some subclasses $\mathcal A_S$ of
finite metric spaces where all distances belong to a given set $S$ by Nguyen
Van Th{\'e}~\cite{The2010}.  Recently, Hubi\v cka and Ne\v set\v ril further
generalised this result to $\mathcal A_S$ for all feasible choices of
$S$~\cite{Hubicka2016}, earlier identified by Sauer~\cite{Sauer2013},
as well as to the class of metric spaces omitting cycles of odd perimeter.
Ramsey property of one additional case was shown by Soki\'c~\cite{Sokic2017}. Hubi\v cka, Kone\v cn\'y and Ne\v set\v ril study and further generalise Conant's generalised monoid metric spaces and prove both the Ramsey property and EPPA for a much broader family of them~\cite{Hubicka2017sauerconnant, Hubicka2017sauer,Konecny2018b}. 

\subsection{Stationary independence relations}
A stationary independence relation (Definition~\ref{SIR}) is a 
concept recently developed by Tent and
Ziegler~\cite{Tent2013} to show simplicity of the automorphism group of the
Urysohn space. M\"uller \cite{Muller2016} showed how these independence relations give rise to a generalisation of the Kat\v etov construction (of the Urysohn space), employing them to prove universality of automorphism groups of structures admitting stationary independence relations.
Examples of structures with a stationary independence relation include free
amalgamation classes, metric spaces~\cite{Tent2013} and the Hrushovski predimension construction (given by Evans, Ghadernezhad and Tent~\cite{Evans2016}).

Given a structure $\str{M}$ and finite substructures
$\str{A}$ and $\str{B}$, the substructure generated by their union is denoted by
$\struc{\str{A}\str{B}}$.
A stationary independence relation is a ternary relation on finite substructures
of a given homogeneous structure $\str{M}$. The substructures $\str{A}, \str{B}, \str{C}$ are
in the relation, roughly speaking, when $\str{M}$ induces a ``canonical amalgamation'' of
$\struc{\str{A}\str{C}}$ and $\struc{\str{B}\str{C}}$ over $\str{C}$. In this case
we write $\str{A}\ind_{\str{C}}\str{B}$.

Stationary independence relations are axiomatised as follows:
\begin{defn}[(Local) Stationary Independence Relation]
\label{SIR}
Assume $\str{M}$ to be a homogeneous structure.  A ternary relation $\ind$ on
the finite substructures of $\str{M}$ is called a \emph{stationary independence
relation} (\emph{SIR}) if the following conditions are satisfied:
\begin{enumerate}[label=SIR\arabic*]
 \item\label{invariance} \emph{(Invariance)}. The independence of finitely generated substructures in $\str{M}$ only depends on their type. In particular, for 
any
automorphism $f$ of $\str{M}$, we have $\str{A}\ind_{\str{C}}\str{B}$ if and only if
$f(\str{A})\ind_{f(\str{C})}f(\str{B})$.
 \item\label{symmetry}\emph{(Symmetry)}. If $\str{A}\ind_\str{C}\str{B}$, then $\str{B}\ind_\str{C}\str{A}$.
\item\label{monotonicity} \emph{(Monotonicity)}. If $\str{A}\ind_{\str{C}}\struc{\str{B}\str{D}}$, then
$\str{A}\ind_{\str{C}}\str{B}$ and $\str{A}\ind_{\struc{\str{B}\str{C}}}\str{D}$.
\item\emph{(Existence)}. \label{existence} For any $\str{A},\str{B}$ and $\str{C}$ in $\str{M}$, there
is some
$\str{A}'\models \tp(\str{A}/\str{C})$ with $\str{A}'\ind_{\str{C}}\str{B}$. 
\item\emph{(Stationarity)}. \label{stationarity} If $\str{A}$ and $\str{A}'$ have the same
type over $\str{C}$ and are both independent over $\str{C}$ from some set $\str{B}$, then they
also have the same type over $\struc{\str{B}\str{C}}$. 
\end{enumerate}
If the relation $\str{A}\ind_{\str{C}}\str{B}$ is only defined for nonempty $\str{C}$, we
call $\ind$ a \emph{local} stationary independence relation.
\end{defn}
Here $\tp(\str{A}/\str{C})$ denotes the type of $\str{A}$ over $\str{C}$, see Definition~\ref{defn:type}.

\subsection{Obstacles to completion}

The list of subclasses of metric spaces with Ramsey expansions corresponds closely to
the list of classes with EPPA. The similarity of these results is not a coincidence.  All of the
proofs proceed from a given metric space and, by a non-trivial construction, build
an edge-labelled graph with either the desired Ramsey property or EPPA. However, these edge-labelled graphs might not be complete. Using some detailed information about when and how they can be completed to the given class of metric spaces $\mathcal A$ it is then
possible to find an actual witness for the Ramsey property or EPPA in $\mathcal A$.
The actual ``amalgamation engines'' have been isolated (see Theorems~\ref{thm:localfini} and~\ref{thm:herwiglascar}) and
are based on a characterisation of each class by a set of obstacles in the sense of the definition below.
Given a set $\mathcal O$ of edge-labelled graphs, let $\Forb(\mathcal O)$ denote the class
of all finite edge-labelled graphs $\str{G}$ such that there is no $\str{O}\in \mathcal O$ with a homomorphism
$\str{O}\to \str{G}$.
\begin{defn}
Given a class of metric spaces $\mathcal A$, we say that $\mathcal O$ is the {\em
set of obstacles of $\mathcal A$} if $\mathcal A\subseteq \Forb(\mathcal O)$ and
moreover every $\str{G}\in \Forb(\mathcal O)$ has a metric completion into $\mathcal A$.
\end{defn}

\subsection{Outline of the paper}
Cherlin's catalogue, given in Section~\ref{sec:catalogue}, provides a rich spectrum of structures. We follow the
ca\-ta\-logue in an order corresponding to the proof techniques (or main
properties of the amalgamation classes), rather than strictly following the
order of the catalogue as presented by Cherlin~\cite{Cherlin2013}. Several
classes are refined and individual special cases are considered separately.

An essential part of the characterisation is the following description of triangle constraints by means of
five numerical parameters.
A {\em triangle} is a triple of distinct vertices $u,v,w\in M$ and its {\em perimeter} is $d(u,v)+d(v,w)+d(w,u)$.
Classes described by constraints on triangles form an essential part of the catalogue; 
we call them {\em 3-constrained classes}. These classes can be described by means
of five numerical parameters:

\begin{defn}[Triangle constraints]
\label{defn:numerical}
Given integers $\delta$, $K_1$, $K_2$, $C_0$ and $C_1$ we consider the class
$\mathcal A^\delta_{K_1,K_2,C_0,C_1}$ of all finite metric spaces $\str{M}=(M,d)$ with integer
distances such that $d(u,v)\leq \delta$ (we call $\delta$ the {\em diameter} of $\mathcal A^\delta_{K_1,K_2,C_0,C_1}$) for every $u,v\in M$
and for every triangle $u,v,w\in M$ with perimeter $p=d(u,v)+d(u,w)+d(v,w)$, the following are true: ($m=\min\{d(u,v),\allowbreak d(u,w),\allowbreak d(v,w)\}$ is the length of the shortest edge of $u,v,w$)
\begin{itemize}
 \setlength\itemsep{0em}
 \item if $p$ is odd then $2K_1 < p < 2K_2 + 2m$,
 \item if $p$ is odd then $p<C_1$, and
 \item if $p$ is even then $p<C_0$.
\end{itemize}
\end{defn}
Intuitively, the parameter $K_1$ forbids all odd cycles shorter than $2K_1+1$, while $K_2$ ensures that the difference in length between even- and odd-distance paths connecting any
pair of vertices is less than $2K_2+1$. The parameters $C_0$ and $C_1$
forbid induced long even and odd cycles respectively. Not every combination of
numerical parameters makes sense and leads to an amalgamation class. Those that do are
characterised by Cherlin's Admissibility Theorem~\ref{thm:admissible}.

We consider the following classes:

\begin{description}
 \item[Spaces of diameter 2:]
 Classes with diameter 2 are not discussed at all in this paper, because all the relevant results have already been established; these are just the homogeneous graphs from Lachlan and Woodrow's catalogue~\cite{Lachlan1980}.

 The Ramsey properties of homogeneous graphs were first studied by Ne\v set\v ril \cite{Nevsetvril1989a}. See also~\cite{Jasinski2013},
 which states the relevant results in more modern language and also shows the ordering properties in all classes from the Lachlan--Woodrow catalogue.

 The EPPA of the class of all finite graphs was shown by Hrushovski~\cite{hrushovski1992}, EPPA of the class of all finite graphs omitting a given complete graph $\str{K}_n$ by Herwig~\cite{Herwig1995} and Hodkinson and Otto~\cite{hodkinson2003},  and
 the remaining cases are particularly easy. Stationary independence relation for all those classes is an easy exercise.
 \item[Primitive 3-constrained spaces:]
 A metrically homogeneous graph $\str{G}$ is {\em primitive} if there are no non-trivial $\Aut(\str{G})$-invariant equivalence relations on its vertices. 

 In Section~\ref{sec:basic3} we consider the richest regular family of amalgamation classes in the catalogue.
 This case contains the class of all finite metric spaces of finite diameter $\delta$ and, more generally, most classes $\mathcal A^\delta_{K_1,K_2,C_0,C_1}$, $\delta<\infty$, with parameters satisfying Case~\ref{II} or~\ref{III} of Theorem~\ref{thm:admissible}. 

These classes are closed under strong amalgamation. We describe a generalised completion algorithm which allows us to show the coherent EPPA and present a Ramsey expansion.
 \item[Bipartite 3-constrained spaces:]
In Section~\ref{sec:bipartite} we consider the 3-constrained spaces $\mathcal A^\delta_{\infty,0,C_0,2\delta+1}$, $\delta<\infty$ that contain no odd triangles, but triangles with edge length $\delta,\delta,2$ are allowed (and thus $C_0 > 2\delta+2$).
 This corresponds to the non-antipodal Case~\ref{I} and is similar to primitive 3-constrained spaces where the bipartitions form definable equivalence relations and thus introduce imaginary elements. The existence of imaginaries has some consequences on the completion algorithm, Ramsey expansions, and coherent EPPA. 
 \item[Spaces with Henson constraints:] As discussed in Section~\ref{sec:henson}, many 3-con\-strained classes can be further restricted by Henson constraints (spaces with distances $1$ and $\delta$ only). These constraints cannot be represented in the form of forbidden triangles. We show that they have little effect on our constructions because the completion algorithm never introduces new Henson constrains in any class where Henson constrains are admissible (Theorem~\ref{thm:admissible_henson}).
 \item[Antipodal spaces:] Here we consider 3-constrained classes where edges of length $\delta$ form a matching. Every amalgamation class closed under forming antipodal companions (see Definition~\ref{defn:antipodal}) can be extended to an antipodal metric space and these are special cases of our constructions because the matching implies non-trivial algebraic closure  in the~\Fraisse{} limits (in other words the class is not closed for strong amalgamation).  We show how to carry the Ramsey and EPPA results from the underlying class to its antipodal variant in Section~\ref{sec:antipodal} and also discuss antipodal Henson constraints.

 Analysis of antipodal spaces is surprisingly subtle and breaks down to 4 different sub-cases which needs to be considered separately. For two of the cases we can not show optimality of our EPPA argument.
 \item[Classes with infinite diameter:]
So far we have only discussed classes of finite diameter. However, we can derive EPPA and the Ramsey property in the remaining infinite diameter cases from what we already know from the classes with finite diameter, as we will show in Section~\ref{sec:infinite}.
 In Section~\ref{sec:tree-like}, we deal with ages of tree-like graphs $T_{m,n}$. These graphs generalise $\mathbb Z$ seen as a metric space, and we show that their basic properties are similar---there is no precompact Ramsey lift and no EPPA. Both conclusions follow from the fact that the algebraic closure of any vertex is infinite.
\end{description}
\section{Preliminaries}
\label{sec:perliminaries}
We first review the standard model-theoretic notions of relational structures and amalgamation classes (see, for example \cite{Hodges1993}).
Next, we introduce the relevant ``amalgamation engines'' used to build Ramsey objects and EPPA-witnesses throughout this paper.

A language $L$ is a set of relational symbols $\rel{}{}\in L$, each associated with a natural number $\arity{}$ called \emph{arity}.
A \emph{(relational) $L$-structure} $\str{A}$ is a pair $(A,(\rel{A}{};\rel{}{}\in L))$ where $\rel{A}{}\subseteq A^{\arity{}}$ (i.e. $\rel{A}{}$ is a $\arity{}$-ary relation on $A$). The set $A$ is called the \emph{vertex set} or the \emph{domain} of $\str{A}$ and elements of $A$ are \emph{vertices}. The language is usually fixed and understood from the context (and it is in most cases denoted by $L$).  If $A$ is a finite set, we call $\str{A}$ a \emph{finite structure}. We consider only structures with countably many vertices.
The class of all (countable) relational $L$-structures will be denoted by $\Rel(L)$.

A \emph{homomorphism} $f\colon\str{A}\to \str{B}=(B,(\rel{B}{};\rel{}{}\in L))$ is a mapping $f\colon A\to B$ satisfying for every $\rel{}{}\in L$ the implication $(x_1,x_2,\ldots, x_{\arity{}})\in \rel{A}{}\implies (f(x_1),f(x_2),\ldots,f(x_{\arity{}}))\in \rel{B}{}$. (For a subset $A'\subseteq A$ we denote by $f(A')$ the set $\{f(x) : x\in A'\}$ and by $f(\str{A})$ the homomorphic image of a structure.) If $f$ is injective, then $f$ is called a \emph{monomorphism}. A monomorphism is called \emph{embedding} if the above implication is an equivalence, i.e. if for every $\rel{}{}\in L$ we have $(x_1,x_2,\ldots, x_{\arity{}})\in \rel{A}{}\iff (f(x_1),f(x_2),\ldots,f(x_{\arity{}}))\in \rel{B}{}$.  If $f$ is an embedding which is an inclusion then $\str{A}$ is a \emph{substructure} (or \emph{subobject}) of $\str{B}$. By the \emph{age of a structure $\str{A}$}, or $\Age(\str{A})$ we denote the class of all finite substructures of $\str{A}$. For an embedding $f\colon\str{A}\to \str{B}$ we say that $\str{A}$ is \emph{isomorphic} to $f(\str{A})$ and $f(\str{A})$ is also called a \emph{copy} of $\str{A}$ in $\str{B}$. We define $\str{B}\choose \str{A}$ as the set of all copies of $\str{A}$ in $\str{B}$.

\begin{figure}
\centering
\includegraphics{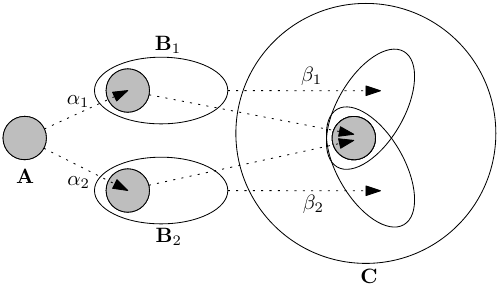}
\caption{An amalgamation of $\str{B}_1$ and $\str{B}_2$ over $\str{A}$.}
\label{amalgamfig}
\end{figure}
Let $\str{A}$, $\str{B}_1$ and $\str{B}_2$ be relational structures and $\alpha_1$ an embedding of $\str{A}$
into $\str{B}_1$, $\alpha_2$ an embedding of $\str{A}$ into $\str{B}_2$, then
every structure $\str{C}$
 with embeddings $\beta_1\colon\str{B}_1 \to \str{C}$ and
$\beta_2\colon\str{B}_2\to\str{C}$ such that $\beta_1\circ\alpha_1 =
\beta_2\circ\alpha_2$ is called an \emph{amalgamation} of $\str{B}_1$ and $\str{B}_2$ over $\str{A}$ with respect to $\alpha_1$ and $\alpha_2$. See Figure~\ref{amalgamfig}.
We will call $\str{C}$ simply an \emph{amalgamation} of $\str{B}_1$ and $\str{B}_2$ over $\str{A}$
(as in most cases $\alpha_1$ and $\alpha_2$ can be chosen to be inclusion embeddings).

We say that an amalgamation is \emph{strong} when $\beta_1(x_1)=\beta_2(x_2)$ if and
only if $x_1\in \alpha_1(A)$ and $x_2\in \alpha_2(A)$.  Less formally, a strong
amalgamation glues together $\str{B}_1$ and $\str{B}_2$ with an overlap no
greater than the copy of $\str{A}$ itself.  A strong amalgamation is \emph{free} if there are no tuples in any relations of $\str{C}$ spanning vertices from both
$\beta_1(B_1\setminus \alpha_1(A))$ and $\beta_2(B_2\setminus \alpha_2(A))$.

An \emph{amalgamation class} is a class $\K$ of finite structures satisfying the following three conditions:
\begin{description}
\item[Hereditary property:] For every $\str{A}\in \K$ and a substructure $\str{B}$ of $\str{A}$ we have $\str{B}\in \K$;
\item[Joint embedding property:] For every $\str{A}, \str{B}\in \K$ there exists $\str{C}\in \K$ such that $\str{C}$ contains both $\str{A}$ and $\str{B}$ as substructures;
\item[Amalgamation property:]
For $\str{A},\str{B}_1,\str{B}_2\in \K$ and $\alpha_1$ an embedding of $\str{A}$ into $\str{B}_1$, $\alpha_2$ an embedding of $\str{A}$ into $\str{B}_2$, there is a $\str{C}\in \K$ which is an amalgamation of $\str{B}_1$ and $\str{B}_2$ over $\str{A}$ with respect to $\alpha_1$ and $\alpha_2$.
\end{description}

\begin{defn}[Type]
\label{defn:type}
 Let $L(\str{X})$ denote the language $L\cup\{c_x:x\in X\}$, where each $c_x$ is a constant symbol interpreted as the vertex $x$. By a {\em type over $\str{X}$}, we mean a maximal satisfiable set of $L(\str{X})$ formulas $p(x)$ with
free variables $x$. The type $\tp(a/\str{X})$ of
a tuple $a$ over $\str{X}$ is the set of all $L(\str{X})$ formulas which are satisfied by
$a$. 
\end{defn}

\subsection{Edge-labelled graphs as relational structures}
\label{sec:reldistance}

For notational convenience we introduced edge-labelled graphs (Definition~\ref{ref:graphs}) and we will
view them, equivalently, also as relational structures in the language $L$ consisting of binary relations
$\rel{}{1},\rel{}{2},\ldots$ which denote the distances.  The notions of homomorphisms, embeddings and substructures
in edge-labelled graphs correspond to the same notions for relational structures.

Given $\str{G}=(G,d)\in \mathcal G^\infty$ and $\str{G}'=(G',d')\in \mathcal G^\infty$ a {\em homomorphism} $G\to G'$ is a function $f\colon G\to G'$
such that $d(x,y)=d'(f(x),f(y))$ whenever $d(x,y)$ is defined.
A homomorphism $f$ is an {\em embedding} (or {\em isometry}) if $f$ is one-to-one and $d(x,y)=d'(f(x),f(y))$ whenever either side of the equality
makes sense. A surjective embedding is an {\em isomorphism} and and {\em automorphism} is an isomorphism $\str{G}\to \str{G}$. A graph $\str{G}$ is an {\em (induced) subgraph} of $\str{H}$ if the identity mapping is an embedding $\str{G}\to \str{H}$.

\subsection{Ramsey expansions}
\label{sec:ramseyexpansion}

Let $L^+$ be a language containing the language $L$. By this we mean $L\subseteq L^+$ and $L^+$ assigns the same arity as $L$ to all $R\in L$.
Under these conditions, every structure $\str{X}=(X,(\rel{X}{}; \rel{}{}\in L^+))\in \Rel(L^+)$ may be viewed as a structure $\str{A}=(X,(\rel{X}{}; \rel{}{}\in L))\in \Rel(L)$ with some additional relations $\rel{X}{}$ for $\rel{}{}\in L^+\setminus L$.
 We call $\str{X}$ a \emph{expansion} (or \emph{lift}) of $\str{A}$ and $\str{A}$ is called the \emph{reduct} (or \emph{shadow}) of $\str{X}$. In this sense the class $\Rel(L^+)$ is the class of all expansions of $\Rel(L)$, or, conversely, $\Rel(L)$ is the class of all shadows of $\Rel(L^+)$.

The question about existence of an Ramsey expansion has been put into more precise setting by means
of the following two definitions:
 
\begin{defn}[Precompact expansion~\cite{NVT14}]
Let $\mathcal K^+$ be a class of lifts to $L^+$ of $L$-structures in $\K$.
We say that $\mathcal K^+$ is a \emph{precompact expansion} (or \emph{lift}) of $\mathcal K$ if for
every structure $\str{A} \in \mathcal K$ there are only finitely many
structures $\str{A}^+ \in \mathcal K^+$ such that $\str{A}^+$ is a lift of
$\str{A}$ (i.e. the shadow of $\str{A}^+$ obtained by forgetting the relations in $L^+\setminus L$ is isomorphic to $\str{A}$). 
\end{defn}

\begin{defn}[Expansion property~\cite{NVT14}]
\label{defn:ordering}
Let a class $\mathcal K^+$ be a lift of $\K$. For $\str{A},\str{B}\in \K$ we say
that $\mathcal K^+$ has the \emph{expansion property} for $\str{A},\str{B}$ if for every lift
$\str{B}^+\in \mathcal K^+$ of $\str{B}$ there is an embedding of every lift $\str{A}^+\in \mathcal K^+$
of $\str{A}$ into $\str{B}^+$.

$\mathcal K^+$ has the \emph{expansion property} with respect to $\K$ if for every $\str{A}\in \K$
there is $\str{B}\in \K$ such that $\mathcal K^+$ has the expansion property for $\str{A},\str{B}$.
\end{defn}
It can be shown (see~\cite{NVT14}) that for every homogeneous class there is up
to bi-definability at most one precompact Ramsey expansion with expansion
property.

\medskip

Our main tool for giving Ramsey property will be Theorem 2.1 of~\cite{Hubicka2016} which we introduce now
after some additional definitions.
 An $L$-structure $\str{A}$ is an \emph{irreducible structure} if every pair of vertices from $A$ is in some relation in $L$ (the relation need not be binary: for example, a complete $k$-hypergraph is irreducible).
\begin{defn}[Homomorphism-embedding~\cite{Hubicka2016}]
 A homomorphism
$f\colon\str{A}\to\str{B}$ is a \emph{homo\-morphism-embedding}  if for every irreducible substructure $\str{C}$ of $\str{A}$, the restriction of $f$ to $C$ is an embedding into $\str{B}$.
\end{defn}
While for (undirected) graphs the homomorphism and homomorphism-em\-bed\-ding coincide, for general relational structures they may differ.
\begin{defn}[Completion~\cite{Hubicka2016}]
\label{defn:completion}
Let $\str{C}$ be a structure. An irreducible structure $\str{C}'$ is a \emph{(strong) completion}
of $\str{C}$ if there exists a one-to-one homomorphism-embedding $\str{C}\to\str{C}'$.

Of particular interest is the question of whether a completion of a given structure exists, such that the completed structure lies in a given class $\mathcal K$. If it does, we speak of its \emph{$\mathcal K$-completion}.
\end{defn}
Note that in \cite{Hubicka2016} the definition of completion is weaker than the definition of strong completion. In this paper all completions will be implicitly strong.
\begin{defn}[Locally finite subclass~\cite{Hubicka2016}]
\label{def:localfinite}
Let $\mathcal R$ be a class of finite irreducible structures and $\mathcal K$ a subclass of $\mathcal R$. We say
that the class $\mathcal K$ is a \emph{locally finite subclass of $\mathcal R$} if for every $\str{C}_0 \in \mathcal R$ there is a finite integer $n = n(\str {C}_0)$ such that 
every structure $\str C$ has a strong $\K$-completion (i.e. there exists $\str{C}' \in \K$ that is a strong completion of $\str{C}$), provided that the following conditions are satisfied:
\begin{enumerate}
\item there is a homomorphism-embedding from $\str{C}$ to $\str{C}_0$ (in other words, $\str{C}_0$ is a, not necessarily strong, $\mathcal R$-completion of $\str{C}$), and,
\item every substructure of $\str{C}$ with at most $n$ vertices has a strong $\K$-comple\-tion.
\end{enumerate}
\end{defn}

\begin{thm}[Hubi\v cka--Ne\v set\v ril~\cite{Hubicka2016}]
\label{thm:localfini}
Let $\mathcal R$ be a Ramsey class of irreducible finite structures and let $\K$ be a hereditary
locally finite subclass of $\mathcal R$ with strong amalgamation.
Then $\K$ is Ramsey.

Explicitly:
For every pair of struc\-tures $\str{A}, \str{B}$ in  $\K$ there exists
a structure $\str{C} \in \K$  such that
$$
\str{C} \longrightarrow (\str{B})^{\str{A}}_2.
$$
\end{thm}

\subsection{Coherent EPPA}
The following is a strengthening of the Herwig--Lascar Theorem~\cite[Theorem 2]{herwig2000} for coherent EPPA which will be our main tool to prove
EPPA in this paper.

\begin{defn}[Coherent maps~\cite{solecki2009,Siniora}]
Let $X$ be a set and $\mathcal P$ be a family of partial bijections 
 between subsets
of $X$. A triple $(f, g, h)$ from $\mathcal P$ is called a {\em coherent triple} if $$\dom(f) = \dom(h), \range(f ) = \dom(g), \range(g) = \range(h)$$ and $$h = g \circ f.$$

Let $X$ and $Y$ be sets, and $\mathcal P$ and $\mathcal Q$ be families of partial bijections between subsets
of $X$ and between subsets of $Y$, respectively. A function $\varphi\colon \mathcal P \to \mathcal Q$ is said to be a
{\em coherent map} if for each coherent triple $(f, g, h)$ from $\mathcal P$, its image $\varphi(f), \varphi(g), \varphi(h)$ in $\mathcal Q$ is coherent.
\end{defn}
\begin{defn}[Coherent EPPA~\cite{solecki2009,Siniora}]
\label{defn:coherent}
A class $\K$ of finite $L$-structures is said to have {\em coherent EPPA} if $\K$ has EPPA and moreover the extension of partial automorphisms
is coherent. That is, for every $\str{A} \in \K$, there exists $\str{B} \in \K$ such that $A\subseteq B$ and every
partial automorphism $f$ of $\str{A}$ extends to some $\hat{f} \in \Aut(\str{B})$ with  the property that the map $\varphi$ from the partial automorphisms of $\str{A}$ to the automorphism of $\str{B}$ given by $\varphi(f) = \hat{f}$ is coherent.
\end{defn}
\begin{thm}[Solecki--Siniora~\cite{solecki2009,Siniora}]
\label{thm:herwiglascar}
Let $\mathcal O$ be a finite family of structures, $\str{A}\in \Forb(\mathcal O)$, and $P$ be a set of partial isomorphisms of $A$.  If there exists a structure $\str{M}$ containing $\str{A}$ such that each element of $P$ extends to an automorphism of $\str{M}$ and moreover there is no $\str{O}\in \mathcal O$ with a homomorphism $\str{O}\to\str{M}$, then there exists a finite structure $\str{B}\in \Forb(\mathcal O)$ and $\phi\colon P \to \Aut(\str{B})$ such that
\begin{enumerate}
\item $\phi(p)$ is an extension of $p$, and
\item $\phi$ is coherent.
\end{enumerate}
\end{thm}
We will call a $\str{B}$ with the properties as stated in Theorem~\ref{thm:herwiglascar} a {\em coherent EPPA-witness} of $\str{A}$.

Observe that while formulated differently, in our setting of strong amalgamation classes the conditions of Theorem~\ref{thm:herwiglascar} very
similar to ones of Theorem~\cite{Hubicka2016}. The infinite structure extending all partial isomorphisms is the \Fraisse{} limit
and thus in both cases we only need to show a bound on the size of obstacles for the completion algorithm.
To show EPPA we additionally need to have a completion algorithm which preserve all symmetries.\footnote{This similarity has been formalized by a recent strengthening of the Solecki--Siniora theorem by Hubička, Konečný, and Nešetřil~\cite{Hubicka2018EPPA}.}

\section{Cherlin's catalogue of metrically homogeneous graphs}
\label{sec:catalogue}
Now we present Cherlin's catalogue of metrically homogeneous
graphs~\cite{Cherlin2013} and the relevant definitions. We slightly change the presentation in contrast to~\cite{Cherlin2013},
but the changes are minor and easily seen to be equivalent.

\begin{conjecture}[Cherlin's Metric Homogeneity Classification Conjecture~\cite{Cherlin2013}]
\label{con:cherlin}
The countable metrically homogeneous graphs are the following.
\begin{enumerate}
\item In diameter $\delta\leq 2$: the homogeneous graphs, classified by Lachlan and Woodrow~\cite{Lachlan1980}.
\item In diameter $\delta\geq 3$:
\begin{enumerate}
\item The finite ones, classified by Cameron~\cite{Cameron1980}.
\item Macpherson’s regular tree-like graphs $T_{m,n}$ with $2 \leq m,n \leq \infty$,
\item The \Fraisse{} limits of amalgamation classes of the form $\mathcal A_3\cap \mathcal A_H$ with 
$\mathcal A_3$ 3-constrained and $\mathcal A_H$ of Henson type or antipodal Henson type.
\end{enumerate}
\end{enumerate}
\end{conjecture}
Let us remark that there is now a purported (yet unpublished) proof of this conjecture.\footnote{This is claimed on Cherlin's website: https://sites.math.rutgers.edu/\~{}cherlin/Paper/inprep.html}

\subsection{3-constrained spaces $\mathcal A_3$}
Recall the numerical parameters $\delta$, $K_1$, $K_2$, $C_0$ and $C_1$ introduced in Definition~\ref{defn:numerical} to describe the 3-constrained spaces.
Different numerical parameters can be used to describe the same classes of
structures. To avoid this redundancy we will assume the following constraints
which describe meaningful sets of parameters.

\begin{defn}[Acceptable numerical parameters]
\label{defn:acceptable}
A sequence of parameters $(\delta,K_1,K_2,C_0,C_1)$ is {\em acceptable} if it satisfies the following conditions:
\begin{itemize}
  \setlength\itemsep{0em}
  \item $3\leq \delta\leq \infty$;
  \item $1\leq K_1\leq K_2\leq 2\delta$ or $K_1=\infty$ and $K_2=0$;
  \item $2\delta+1\leq C_0,C_1\leq 3\delta+2$, and if $\delta < \infty$ then $C_0$ is even and $C_1$ is odd;
  \item If $K_1=\infty$ (the bipartite case) then $C_1=2\delta+1$.
\end{itemize}
\end{defn}
\begin{thm}[Cherlin's Admissibility Theorem~\cite{Cherlin2013}]
\label{thm:admissible}
Let $(\delta,K_1,K_2,C_0,C_1)$ be an acceptable sequence of parameters (in particular, $\delta\geq 3$). Then
the associated class $\mathcal A^\delta_{K_1,K_2,C_0,C_1}$ is an amalgamation class if and only if
one of the following three groups of conditions is satisfied, where we write $C$ for $\min(C_0,C_1)$
and $C'$ for $\max(C_0,C_1)$:
\begin{enumerate}[label=(\Roman*)]
\setlength\itemsep{0em}
\item\label{I} $K_1=\infty$ (the bipartite case; so $K_2=0$ and $C_1=2\delta+1$).
\item\label{II} $K_1<\infty, C\leq 2\delta+K_1$, and
\begin{itemize}
 \setlength\itemsep{0em}
 \item $C=2K_1+2K_2+1$;
 \item $K_1+K_2\geq \delta$;
 \item $K_1+2K_2\leq 2\delta-1$, and:
\end{itemize}
\begin{enumerate}[label=(II\Alph*)]
\setlength\itemsep{0em}
\item\label{IIa} $C'=C+1$, or
\item\label{IIb} $C'>C+1, K_1=K_2$, and $3K_2=2\delta-1$.
\end{enumerate}
\item\label{III} $K_1<\infty$, $C>2\delta+K_1$, and:
\begin{itemize}
 \setlength\itemsep{0em}
 \item $K_1+2K_2\geq 2\delta-1$ and $3K_2\geq 2\delta$;
 \item If $K_1+2K_2=2\delta-1$ then $C\geq 2\delta+K_1+2$;
 \item If $C'>C+1$ then $C\geq 2\delta+K_2$.
\end{itemize}
\end{enumerate}
\end{thm}
A sequence of parameters $(\delta,K_1,K_2,C_0,C_1)$  is called {\em admissible} if and only if it satisfies one of the three sets of conditions in Theorem~\ref{thm:admissible}.

\begin{table}[t]
\centering 
\begin{tabular}{|rrrr|r|r|l|l|}
\hline
 $K_1$   & $K_2$&$C_0$& $C_1$& $M$ & Case & $\mathcal S$ & Structure \\ \hline
 $\infty$& 0 &  8 &  7 &  --   & \ref{I} & $\emptyset$ & Bipartite antipodal\\ 
 $\infty$& 0 &  10 & 7 & --    & \ref{I} & $\emptyset$ & Bipartite\\ 
 1       & 2 &  8 &  7 & -- & \ref{IIa} & $\emptyset$ & Antipodal\\ 
 1       & 2 & 10 &  9 & 2 & \ref{III} & $K_n$ & No $\delta\delta\delta$, $1\delta\delta$ triangles\\ 
 1       & 2 & 10 & 11 & 2 & \ref{III} & $K_n$ and/or $I_m$ & No $1\delta\delta$ triangles\\ 
 1       & 3 &  8 &  9 & 2 & \ref{III} & $\emptyset$ & No 5-anticycle \\ 
 1       & 3 & 10 & 9 & 2 & \ref{III} & Any w/o $I_3$ & No $\delta\delta\delta$ triangles\\ 
 1       & 3 & 10 & 11 & $2, 3$ & \ref{III} & Any & All metric spaces \\ 
 2       & 2 & 10 & 9 & 2 & \ref{III} & $\emptyset$ & No $\delta\delta\delta$, $1\delta\delta$, $111$\\ 
 2       & 2 & 10 & 11 & 2 & \ref{III} & $I_m$ & No $1\delta\delta$, $111$ triangles\\ 
 2       & 3 & 10 & 9 & 2 & \ref{III} & Any w/o $K_3$,$I_3$ & No $\delta\delta\delta$, $111$ triangles\\ 
 2       & 3 & 10 & 11 & $2, 3$ & \ref{III} & Any w/o $K_3$ & No $111$ triangles \\ 
 3       & 3 & 10 & 11 & 3 & \ref{III} & $\emptyset$ & No 5-cycle \\ \hline
\end{tabular}
\caption{All admissible parameters for $\delta=3$ with the set $\mathcal S$ of Henson constraints limited to meaningful choices~\cite[Table 2]{Amato2016}
and parameter $M$ satisfying Definitions~\ref{defn:magiccompletion} and~\ref{defn:bimagiccompletion}. $K_n$ denotes the $n$-clique (i.e. the metric space with $n$ vertices and all distances 1) and $I_n$ is the $n$-anticlique (all distances $\delta$). The second column lists the possible choices for magic distances (see Definition \ref{defn:bimagiccompletion})}
\label{tab:delta3}
\end{table}
\begin{example}
All admissible parameters with $\delta=3$ are listed in Table~\ref{tab:delta3}.

\end{example}
\subsection{Henson constraints}
\label{sec:Henson}
Suppose $\delta\geq 3$. A {\em $(1,\delta)$-space} is a metric space in which all distances are $1$ or $\delta$;
thus the relation $d(x,y)\leq 1$ is an equivalence relation, and the classes lie at mutual distance $\delta$ (in $(1,\infty)$-spaces, all distances are equal to 1).
A $(1,\delta)$-space will also be called a {\em $\delta$-Henson constraint}, or just a \emph{Henson constraint} if
$\delta$ is clear from the context. 

A {\em clique} is a metric space $\str{K}=(K,d)$ where $d(u,v)=1$ for every $u\neq v$.
$\str{K}^*=(K,d^*)$ is an {\em antipodal companion of the clique $\str{K}$} if there exists $S\subset K$ such 
that for distinct vertices $u,v$, $d^*(u,v)=\delta-1$ if $u\in S,v\notin S$ or vice versa and $d^*(u,v)=1$ otherwise.

\begin{defn}[Acceptable parameters with Henson constraints]
The sequence of parameters $(\delta,K_1,K_2,\allowbreak C_0,C_1,\mathcal S)$ is {\em acceptable} if
\begin{enumerate}
 \item $(\delta,K_1,K_2,C_0,C_1)$ is an acceptable sequence of numerical parameters.
 \item $\mathcal S$ is a set of $\delta$-Henson constraints if $\delta = \infty$ or $C_1>2\delta + 1$ or $C_0 > 2\delta + 2$, and a set of cliques and their antipodal companions if $\delta < \infty$, $C_1 = 2\delta + 1$ and $C_0 = 2\delta + 2$ (cf. Definition \ref{defn:antipodal}).
 \item $\mathcal S$ is {\em irredundant} in the sense that no constraint in $\mathcal S$ contains triangles forbidden in $\mathcal A^\delta_{K_1,K_2,C_0,C_1}$, and every constraint in $\mathcal S$ consists of at least 4 vertices. 
\end{enumerate}
\end{defn}
\begin{thm}[Admissible parameters with Henson constraints]
\label{thm:admissible_henson}
Let  $(\delta,K_1,K_2,C_0,C_1,\mathcal S)$ be an acceptable sequence of parameters, where $\mathcal S$
is a set of $\delta$-Henson constraints. Let $C=\min(C_0,C_1)$ and $C'=\max(C_0,C_1)$. Let $\mathcal A^\delta_\mathcal S$
be the class of all finite metric spaces of diameter $\delta$ omitting embeddings from members of $\mathcal S$.
Then the class $\mathcal A^\delta_{K_1,K_2,C_0,C_1} \cap \mathcal A^\delta_\mathcal S$ is an amalgamation class if and only if
the following holds:
\begin{enumerate}
 \item The sequence $(\delta,K_1,K_2,C_0,C_1)$ is admissible.
 \item If $\delta < \infty$, $C = 2\delta + 1$, $K_1 < \infty$, and $\mathcal S$ is nonempty, then $\delta\geq 4$ and $\mathcal S$
consists of a clique and all its antipodal companions.
 \item If $\delta = \infty$, or $C>2\delta+K_1$ (Case~\ref{III}), then
\begin{itemize}
   \item If $K_1=\delta$ then $\mathcal S$ is empty;
   \item If $\delta < \infty$ and $C=2\delta+2$ then $\mathcal S$ is empty.
\end{itemize}
\end{enumerate}
\end{thm}
A sequence of parameters $(\delta,K_1,K_2,C_0,C_1,\mathcal S)$  is called {\em admissible} if and only if it satisfies the conditions in Theorem~\ref{thm:admissible_henson}.

\section{Primitive 3-constrained spaces}
\label{sec:basic3}
In this section we will work with fixed admissible parameters $\delta<\infty$, $K_1$,
$K_2$, $C_0$ and $C_1$, such that the associated homogeneous metric space $\Gamma^\delta_{K_1,K_2,C_0,C_1}$ is primitive, i.e. it has no non-trivial definable equivalence relation on it. As one can verify, those are exactly the parameters where the class $\mathcal A^\delta_{K_1,K_2,C_0,C_1}$ has strong amalgamation and contains at least one triangle with odd perimeter.  
Recall that we denote $C=\min(C_0,C_1),$ and $C'=\max(C_0,C_1).$ Then we can characterise the parameters as follows:
\begin{defn}
The admissible parameters $\delta<\infty$, $K_1$, $K_2$, $C_0$ and $C_1$ are {\em primitive} when
they satisfy case \ref{II} or \ref{III} of                                        
Theorem~\ref{thm:admissible} and moreover do not form an antipodal space. 

In other words, the triangle $\delta\delta a$ for some $a\geq 1$ is permitted and $$C_0,C_1 \geq 2\delta + 2.$$
\end{defn}
For many of the lemmas in this section, the inequalities $C_0,C_1\geq 2\delta + 2$ will be an essential assumption.

Recall that $\mathcal A^\delta_{K_1,K_2,C_0,C_1}$ was defined in Definition \ref{defn:numerical} as the class of all metric spaces that satisfy some constraints on its triangles, i.e. its 3-element subspaces. In the following it will often be more convenient to think of $\mathcal A^\delta_{K_1,K_2,C_0,C_1}$ as the class of metric spaces that do not embed any triangle violating those constraints --- we are going to refer to such triangles as \emph{forbidden triangles}. We also will slightly abuse notation and use the term triangle for both triples of vertices $u,v,w$ and for triples of edges $a,b,c$ with $a = d(u,v)$, $b = d(v,w)$, $c = d(u,w)$. By Definition \ref{defn:numerical} a triangle $abc$ is forbidden if it satisfies one of the following conditions: 
\begin{description}
\item[Non-metric:] $a,b,c$ is forbidden if $a+b<c$,
\item[$K_1$-bound:] $a+b+c < 2K_1+1$ and $a+b+c$ is odd,
\item[$K_2$-bound:] $b+c\geq 2K_2+a$ and $a+b+c$ is odd and $a\leq b,c$,
\item[$C_1$-bound:] $a+b+c\geq C_1$ and $a+b+c$ is odd,
\item[$C_0$-bound:] $a+b+c\geq C_0$ and $a+b+c$ is even,
\item[$C$-bound:] If $|C_0-C_1| = 1$, the $C_1$-bound and $C_0$-bound can be expressed together as $a+b+c\geq C$, where $C=\min(C_0,C_1)$.
\end{description}
Triangles that are not forbidden will be called \emph{allowed}.

\subsection{Generalised completion algorithm for 3-constrained classes}
\label{sec:algorithm}

Let $\mathcal D=\{1,2,\ldots \delta\}^2$ be a collection of (ordered) pairs. It is more natural to consider unordered pairs, but notationally easier to consider ordered pairs. We will refer to elements of $\mathcal D$ as {\em forks}. 

Consider a $\delta$-bounded variant of the shortest path completion, where we assume that the input graphs contain no distances greater than $\delta$ and in the output all edges longer than $\delta$ are replaced by an edge of that length. There is an alternative formulation of this completion: For a fork $\vec{f}=(a,b)$, define $d^+(\vec{f})=\min(a+b,\delta)$. In the $i$-th step look at all incomplete forks $\vec{f}$ (i.e. triples of vertices $u,v,w$ such that exactly two edges are present) such that $d^+(\vec{f}) = i$ and define the length of the missing edge to be $i$.

This algorithm proceeds by first adding edges of length 2, then edges of length 3 and so on up to edges of length $\delta$ and has the property that out of all metric completions of a given graph, every edge of the completion yielded by this algorithm is as close to $\delta$ as possible.

It makes sense to ask what happens if, instead of trying to make each edge as close to $\delta$ as possible, one would try to make each edge as close to some parameter $M$ as possible. For $M$ in a certain range, such an algorithm exists. For each fork $\vec{f}=(a,b)$ one can define $d^+(\vec{f}) = a+b$ and $d^-(\vec{f}) = |a-b|$, i.e. the largest and the smallest possible distance that can metrically complete the fork $\vec{f}$. The generalised algorithm will complete $\vec{f}$ by $d^+(\vec{f})$ if $d^+(\vec{f})<M$, by $d^-(\vec{f})$ if $d^-(\vec{f})>M$ and by $M$ otherwise. It turns out that there is a good permutation $\pi$ of $\{1,\ldots,\delta\}$, such that if one adds the distances in the order prescribed by the permutation, this generalised algorithm will produce a correct completion whenever one exists. It is easy to check that the choice $M=\delta$ and $\pi=\text{id}_\delta$ corresponds to the shortest path completion algorithm.

\begin{defn}[Completion algorithm]\label{defn:ftmcompletion}
Given $c\geq 1$, $\mathcal F\subseteq \mathcal D$, and
a graph $\str{G}=(G,d)\in \mathcal G^\delta$, we say that $\str{G}'=(G,d')$ is the 
{\em $(\mathcal F,c)$-completion} of $\str{G}$ if $d'(u,v)=d(u,v)$ whenever $u,v$ is an edge of $\str{G}$
and $d'(u,v)=c$ if $u,v$ is not an edge of $\str{G}$ and there exist $(a,b)\in \mathcal F$, $w\in G$ such that $\{d(u,w),d(v,w)\}=\{a,b\}$. There are no other edges in $\str{G}'$.

Given $1\leq M\leq \delta$, a one-to-one function $t\colon\{1,2,\ldots,\delta\}\setminus \{M\}\to \mathbb N$ 
and a function $\mathbb F$
from $\{1,2,\ldots,\delta\}\setminus \{M\}$ to the power set of $\mathcal D$, we define the {\em $(\mathbb F,t,M)$-completion} of  $\str{G}$
as the limit of a sequence of edge-labelled graphs  $\str{G}_1, \str{G}_2,\ldots$ such that $\str{G}_1=\str{G}$ and $\str{G}_{k+1}=\str{G}_k$ if $t^{-1}(k)$ is undefined
and $\str{G}_{k+1}$ is the $(\mathbb F(t^{-1}(k)),t^{-1}(k))$-completion of $\str{G}_{k}$ otherwise, with every pair of vertices not forming an edge in this limit set to distance $M$.
\end{defn}

We will call the vertex $w$ from Definition \ref{defn:ftmcompletion} the {\em witness of the edge $u,v$}. The function $t$ is called the {\em time function} of the completion because edges of length $a$ are inserted to $\str{G}_{t(a)}$ the $t(a)$-th step of the completion. If for a $(\mathbb F, t, M)$-completion and distances $a,c$ there is a distance $b$ such that $(a,b)\in \mathbb F(c)$ (i.e. the algorithm might complete a fork $(a,b)$ with distance $c$), we say that {\em $c$ depends on $a$}.

\begin{defn}[Magic distances]
\label{defn:magicdistance}
Let $M\in\{1,2,\ldots, \delta\}$ be a distance. We say that $M$ is {\em magic} (with respect to $\mathcal A^\delta_{K_1,K_2,C_0,C_1}$) if $$\max\left(K_1, \left\lceil\frac{\delta}{2}\right\rceil\right) \leq M \leq \min\left(K_2,\left\lfloor\frac{C-\delta-1}{2}\right\rfloor\right).$$
\end{defn}

Note that for primitive admissible parameters $(\delta, K_1,K_2,C_0,C_1)$ such an $M$ always exists.

\begin{observation}\label{obs:magicismagic}
The set of magic distances (with respect to $\mathcal A^\delta_{K_1,K_2,C_0,C_1}$) is $$S=\left\{1\leq a \leq \delta : aab\text{ is allowed for all }1\leq b\leq \delta\right\}.$$
\end{observation}

\begin{proof}
If a distance $a$ is in $S$, then $a\geq K_1$ (otherwise the triangle $aa1$ has perimeter $2a+1$, which is odd and smaller than $2K_1+1$, hence forbidden by the $K_1$ bound), $a\geq \left\lceil\frac{\delta}{2}\right\rceil$ (otherwise the triangle $aa\delta$ is non-metric), $a\leq \left\lfloor\frac{C-\delta-1}{2}\right\rfloor$ (otherwise the triangle $aab$ has perimeter $C$ for $b = C -2a \leq \delta$), and $a\leq K_2$ (otherwise the triangle $aa1$ has odd perimeter and $2a\geq 2K_2+1$, hence is forbidden by the $K_2$ bound). The other implication follows from the definition of $\mathcal A^\delta_{K_1,K_2,C_0,C_1}$.
\end{proof}

\begin{remark}
A Sage-implementation of the following completion algorithm is available at~\cite{PawliukSage2}. It also contains all the examples discussed in this paper.

The same algorithm is used in our earlier paper~\cite{Aranda2017c}, although there we restrict to a special set of parameters. The analysis here is significantly more detailed and it is necessary to consider the cases where $K_2<\delta$ or $C'>C+1$. It is somewhat surprising that considering these cases does not make the algorithm significantly more complicated.
\end{remark}

Let $M$ be a magic distance and $x\in\{1,\ldots,\delta\}\setminus\{M\}$. Define $$\mathcal F^+_x = \left\{(a,b)\in \mathcal D : a+b=x\right\},$$ $$\mathcal F^-_x = \left\{(a,b)\in \mathcal D : |a-b|=x\right\},$$ $$\mathcal F^C_x = \left\{(a,b)\in \mathcal D : C-1-a-b=x\right\}.$$ We further denote
$$\mathbb F_M(x) =
\begin{cases} 
      \mathcal F^+_x\cup \mathcal F^C_x & \text{if }x < M \\
      \mathcal F^-_x & \text{if }x > M.
\end{cases}
$$
For a magic distance $M$, we also define the function $t_M\colon \{1,\ldots,\delta\}\setminus \{M\} \rightarrow \mathbb N$ as
$$t_M(x) =
\begin{cases} 
      2x-1 & \text{if } x < M \\
      2(\delta-x) & \text{if }x > M.
\end{cases}
$$
\begin{figure}
\centering
\includegraphics{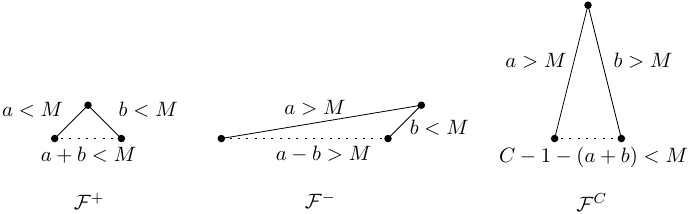}
\caption{Forks used by $\mathbb F_M$.}
\label{fig:Fforks}
\end{figure}
Forks and how they are completed according to $\mathbb F_M$ are schematically depicted in Figure~\ref{fig:Fforks}.

\begin{defn}[Completion with magic parameter $M$]
\label{defn:magiccompletion}
Let $M$ be a magic distance satisfying the following extra conditions:
\begin{enumerate}
\item If the parameters satisfy Case~\ref{III} with $K_1+2K_2 = 2\delta - 1$, then $M>K_1$;
\item if the parameters satisfy Case~\ref{III} and further $C'>C+1$ and $C=2\delta+K_2$, then $M<K_2$.
\end{enumerate}
We then call the $(\mathbb F_M,t_M,M)$-completion (of $\str{G}$) the {\em completion (of $\str{G}$) with magic parameter $M$}. 
\end{defn}

\begin{remark}
The completion with magic parameter $M$ can be equivalently stated as a shortest path completion, but using a different ordered monoid (see Section~\ref{sec:conclussion}, paragraph 4).
\end{remark}

Our main goal of the following section is the proof of Theorem \ref{thm:magiccompletion} that shows that the completion of $\str{G}$ with magic parameter $M$ lies in $\mathcal A^\delta_{K_1,K_2,C_0,C_1}$ if and only if $\str{G}$ has some completion in $\mathcal A^\delta_{K_1,K_2,C_0,C_1}$.

The two extra conditions in Definition \ref{defn:magiccompletion} are a way to deal with certain extremal choices of admissible primitive parameters $(\delta, K_1,K_2,C_0,C_1)$. 
We are going to check that also with those extra conditions, there will always exist a suitable magic distance:

\begin{lem}
For primitive parameters $(\delta, K_1,K_2,C_0,C_1)$ there is always an $M$ satisfying Definitions~\ref{defn:magicdistance} and \ref{defn:magiccompletion}.
\end{lem}
\begin{proof}

For Case~\ref{III} with $K_1+2K_2 = 2\delta-1$, we proceed as follows: From admissibility, we have
\[\begin{split}K_1+2K_2&=2\delta-1\\ 3K_2&\geq2\delta\end{split}\] so we conclude that $K_1<K_2$. From this information and $K_1+2K_2=2\delta-1$, we derive $K_1<\frac{2}{3}\delta$. We know that $\delta-1 \geq \frac{2}{3}\delta$ for $\delta \geq 3$ and $K_1\leq \delta-2$, so $\left\lfloor\frac{C-\delta-1}{2}\right\rfloor \geq \left\lfloor\frac{\delta+K_1+1}{2}\right\rfloor\geq \frac{\delta+K_1}{2}$. Hence, $K_1 < \left\lfloor\frac{C-\delta-1}{2}\right\rfloor$ and there is always a magic number greater than $K_1$.\bigskip

In Case~\ref{III} with $C'>C+1$ and $C=2\delta+K_2$, we know from admissibility that $C>2\delta + K_1$, so $K_2 > K_1$. Now we need $\left\lceil\frac{\delta}{2}\right\rceil < K_2$. For $\delta \geq 3$, the inequality $\left\lceil\frac{\delta}{2}\right\rceil \leq \frac{2}{3}\delta$ holds with equality only for $\delta = 3$. Admissibility tells us $3K_2\geq 2\delta$. Now, if $\delta > 3$ or $K_2\neq \frac{2}{3}\delta$, it follows that $\left\lceil\frac{\delta}{2}\right\rceil < K_2$. 

The only remaining possibility is $\delta=3$ and $K_2=2$, which implies $C=8$ and $K_1=1$, which gives us $2K_2+K_1 = 5 = 2\delta-1$. The admissibility condition $C\geq 2\delta +K_1+2$ then yields $C\geq 9$, a contradiction. Hence there always is a magic number smaller than $K_2$.

If both these situations occur simultaneously, then we further require $M$ with $K_1<M<K_2$. But that follows as $C=2\delta+K_2$ and whenever $K_1+2K_2=2\delta-1$, from admissibility we have $C\geq 2\delta+K_1+2$, hence $K_2\geq K_1+2$.
\end{proof}

Observe that the algorithm only makes use of $C$, $\delta$ and $M$. The
interplay of individual parameters of algorithm is schematically depicted in
Figure~\ref{fig:algorithm}.
\begin{figure}
\centering
\includegraphics{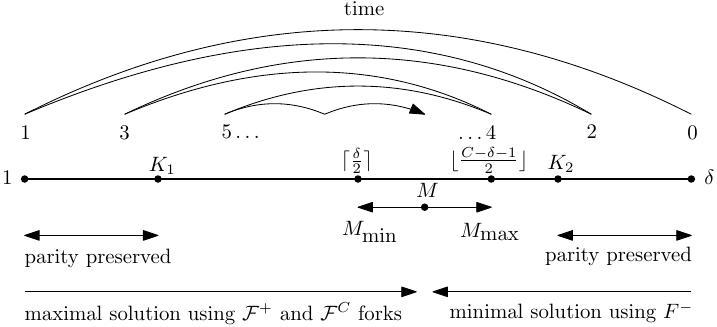}
\caption{A sketch of the main parameters of the completion algorithm, the Optimality Lemma~\ref{lem:bestcompletion} and the Parity Lemma~\ref{lem:sameparity}.}
\label{fig:algorithm}
\end{figure}

\begin{example}[Case~\ref{IIb}]
In our proofs, Case~\ref{IIb} will often form a special case.
The smallest (in terms of diameter) set of acceptable parameters that is in Case \ref{IIb} is: $$\delta=5, C=C_1 = 13, C^\prime=C_0 = 16, K_1 = K_2 = \frac{2\delta-1}{3} = 3.$$ Here $M=3$, and it is the only choice for a magic number.

Forbidden triangles are those that are non-metric (113, 114, 115, 124, 125, 135, 225), or rejected for the $K_1$-bound $(111,122)$, the $K_2$-bound (144, 155, 245), or the $C_1$-bound (355, 445, 555). There are no triangles forbidden by $C_0$.
\begin{figure}
\centering
\includegraphics{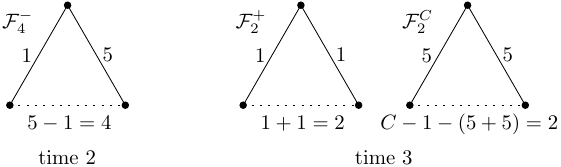}
\caption{Forks considered by the algorithm to complete to $\mathcal A^5_{3,3,16,13}$ with $M=3$.}
\label{fig:forks}
\end{figure}
\begin{table}[t]
\centering
\begin{tabular}{|c|c|c|c|c|c|}
\hline
&$j=1$&$j=2$&$j=3$&$j=4$&$j=5$\\ \hline
$i=1$&\textbf{2} &  $1,\textbf{3}$  &  $2,\textbf{3},4$    & $\textbf{3},5$   		& \textbf{4} 		\\ 
$i=2$&					&  $2,\textbf{3},4$&  $1,2,\textbf{3},4,5$& $2,\textbf{3},4$ 		& $\textbf{3},5$ 	\\ 
$i=3$&					&  							&  $1,2,\textbf{3},4,5$& $1,2,\textbf{3},4,5$	& $2,\textbf{3},4$	\\ 
$i=4$&					&  							&  									&  $2,\textbf{3},4$		& $1,\textbf{3},5$	\\ 
$i=5$&					&  							&  									&  									& $\textbf{2},4$		\\ \hline
\end{tabular}
\caption{Possible ways to complete $(i,j)$ forks, the bold number is the completion with magic parameter $M = 3$.}
\label{tab:forks}
\end{table}
Table~\ref{tab:forks} lists all possible completions of forks, with the completion preferred by our algorithm in bold type. Completions for forks in this class depicted in Figure~\ref{fig:forks}.
Notably, the magic number $M=3$ is chosen for all forks except $(1,1)$, which is completed by $d^+((1,1)) = 2$, $(1,5)$, which is completed by $d^-((1,5)) = 4$, and $(5,5)$, which is a $C$-bound case. 
Those cases are the only forks where $M=3$ cannot be chosen, so instead the algorithm chooses the nearest possible completion. What makes Case~\ref{IIb} special is the situation where one can choose $M-1$ or $M+1$ but not $M$ when completing a fork (for $\delta=5$ it is the fork $(5,5)$, as both the triangles $5,5,2$ and $5,5,4$ are allowed, while $5,5,3$ is forbidden by the $C_1$ bound; this behaviour is going to force us to deal with some corner cases later). 

The algorithm will thus effectively run in three steps.  First (at time 2) it will complete all forks $(1,5)$ with distance 4, next (at time 3) it will complete all forks $(1,1)$ and $(5,5)$ with distance 2 and finally it will turn all non-edges into edges of distance 3. Examples of runs of this algorithm are given later, see Figures~\ref{fig:1555} and~\ref{fig:11555}.

See also~\cite{Aranda2017c} for an additional example of a run of the algorithm for the space $\mathcal A^6_{2,1,16,15}$, or see the Sage implementation in \cite{PawliukSage2}.
\end{example}

\subsection{What do forbidden triangles look like?}
\label{sec:forbtriangles}
The majority of the proofs in the following sections assume that the completion algorithm with magic parameter $M$ introduces some forbidden triangle and then we argue that the triangle must be forbidden in any completion, hence the input structure has no completion into $\mathcal A^\delta_{K_1,K_2,C_0,C_1}$. In such an argument it will be helpful to know how the different types of forbidden triangles relate to the magic parameter $M$. Therefore, in the following paragraphs we will study how triangles forbidden by different bounds are related to the magic parameter $M$. We will use $a,b,c$ for the lengths of the sides of the triangle and without loss of generality assume $a\leq b\leq c$. All conclusions are summarised in Figure~\ref{fig:Ftriangles}.
\begin{figure}[t]
\centering
\includegraphics{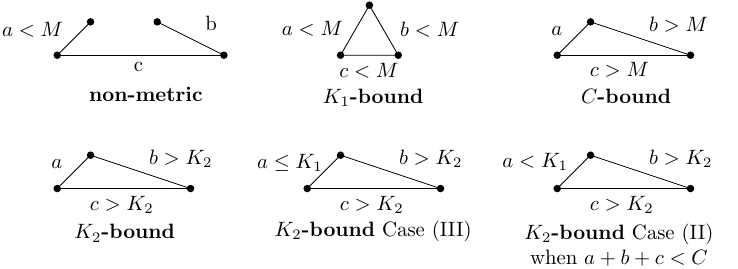}
\caption{Types of forbidden triangles.}
\label{fig:Ftriangles}
\end{figure}

\begin{description}
\item[non-metric:]
If $a+b<c$, then $a < M$, because otherwise $a+b\geq 2M \geq \delta$.

\item[$K_1$-bound:]
If $a+b+c < 2K_1+1$, $a+b+c$ is odd and $abc$ is metric, then $a,b,c < K_1\leq M$, because if $c\geq K_1$, then from the metric condition $a+b\geq c\geq K_1$ and hence $a+b+c\geq 2K_1$, for odd $a+b+c$ this means $a+b+c\geq 2K_1+1$.

\item[$C$-bound:]
If $a+b+c\geq C$ then $b,c > M$. Suppose for a contradiction that $a,b\leq M$. We then have $a+b\geq C-c\geq C-\delta$, but on the other hand $a+b\leq 2M\leq 2\left\lfloor \frac{C-\delta-1}{2} \right\rfloor\leq C-\delta-1$, which together yield $C-\delta-1\geq C-\delta$, a contradiction. Note that in some cases $C'\neq C+1$, but this observation still holds, as it only uses $a+b+c\geq C$.

\item[$K_2$-bound:]
If $abc$ is a metric triangle with odd perimeter, then $abc$ breaks the $K_2$ condition if and only if $b+c\geq 2K_2+a+1$ (the 1 on the right side comes from $a+b+c$ being odd and all distances being integers). Then $b,c>K_2$, because if $b\leq K_2$, from metricity we have $c\leq a+b$, hence $a+2K_2\geq (a+b)+b \geq c+b\geq 2K_2+a+1$, a contradiction.

Moreover, in Case~\ref{III} of Theorem~\ref{thm:admissible} we have $a \leq K_1$ because if $a > K_1$, we have $b+c\geq 2K_2+a+1 > 2K_2+K_1+1$ and from admissibility conditions for Case~\ref{III} we have $2K_2+K_1\geq 2\delta-1$, which gives $b+c>2\delta$, a contradiction. (Note that if $2K_2+K_1>2\delta-1$, we have $a<K_1$.)

Finally if $a+b+c < C$ (which is stronger than not being forbidden by the $C$ bound, as it also includes the $C'>C+1$ cases) and we are in Case~\ref{II} (where $C=2K_1+2K_2+1$), we get $a < K_1$, because if $a\geq K_1$, we would get $a+b+c\geq 2K_2+2a+1\geq 2K_2+2K_1+1 = C$.

Note that later we shall refer to all the corner cases mentioned in these paragraphs.
\end{description}

\subsection{Basic properties of the completion algorithm}
In this section we develop several technical observations about the algorithm which will be used
in Section~\ref{sec:magiccompletion} to show the main result about the correctness of the algorithm.

Recall the definition of $t_M$ and $\mathbb F_M$:
$$t_M(x) =
\begin{cases} 
      2x-1 & \text{if } x < M \\
      2(\delta-x) & \text{if }x > M.
\end{cases}
$$
$$\mathbb F_M(x) =
\begin{cases} 
      \mathcal F^+_x\cup \mathcal F^C_x & \text{if }x < M \\
      \mathcal F^-_x & \text{if }x > M.
\end{cases}
$$
Intuitively, the function $\mathbb F_M$ selects the forks that will be completed to triangles with an edge of type $t^{-1}_M(x)$ at time $x$. At time 0 it looks for forks that can be completed with distance $\delta$, then with distance 1, jumping back and forth on the distance set and approaching $M$ (cf. Figure \ref{fig:algorithm}). Observe that all forks that cannot be completed with $M$ are in some $\mathbb F_M(x)$.

Now we shall precisely state and prove that $t_M$ gives a suitable injection for the algorithm, as claimed before Definition~\ref{defn:ftmcompletion}.

\begin{lem}[Time Consistency Lemma]\label{lem:expandtime}
Let $a,b$ be distances different from $M$. If $a$ depends on $b$, then $t_M(a) > t_M(b)$.
\end{lem}

\begin{proof}
We consider three types of forks used by the algorithm:
\begin{description}
\item[$\mathcal F^+$:]
If $a<M$ and $\mathcal F^+_a\neq\emptyset$, then $b<a<M$, hence $t_M(b) < t_M(a)$.

\item[$\mathcal F^C$:]
If $a < M$ and $\mathcal F^C_a\neq\emptyset$, then we must have $b,c > M$ ($(b,c)\in \mathcal F^C_a$ with $a=C-1-b-c$). Otherwise, if for instance $b\leq M$, then $C-\delta-1\leq C-1-c = a+b < 2M \leq 2\left\lfloor \frac{C-\delta-1}{2} \right\rfloor$, a contradiction. As $C\geq 2\delta+2$ (we are dealing with the primitive case), we obtain the inequality $b = (C-1)-c-a \geq (2\delta + 1) - \delta - a = \delta+1-a$. Hence $t_M(b) \leq 2(a-1) < 2a-1 = t_M(a)$.

\item[$\mathcal F^-$:]
Finally, we consider the case where $a > M$ and $\mathcal F^-_a\neq\emptyset$. Then either $a = b-c$, which implies $b>a>M$ and thus $t_M(b)<t_M(a)$, or $a = c-b$, which means $b = c-a\leq \delta-a$. Because of $a>M\geq \left\lceil\frac{\delta}{2}\right\rceil$, we have $b<M$. So $t_M(b) \leq 2(\delta-a) - 1 < 2(\delta-a) = t_M(a)$.
\end{description}
\end{proof}

\begin{lem}[$\mathbb F_M$ Completeness Lemma]\label{lem:misgood}
Let $\str{G}\in \mathcal G^\delta$ and $\overbar{\str{G}}$ be its completion with magic parameter $M$. If there is a forbidden triangle (w.r.t. $\mathcal A^\delta_{K_1,K_2,C_0,C_1}$) or a triangle with perimeter at least $C$ in $\overbar{\str{G}}$ with an edge of length $M$, then this edge is also present in $\str{G}$.
\end{lem}

Observe that for $C'\neq C+1$, this lemma is talking not only about forbidden triangles, but about all triangles with perimeter at least $C$.

\begin{proof}
By Observation \ref{obs:magicismagic} no triangle of type $aMM$ is forbidden, so suppose that there is a forbidden triangle $abM$ in $\overbar{\str{G}}$ such that the edge of length $M$ is not in $\str{G}$. For convenience define $t_M(M) = \infty$, which corresponds to the fact that edges of length $M$ are added in the last step.

\begin{description}
\item[non-metric:] If $abM$ is non-metric then either $a+b<M$ or $|a-b|>M$. By Lemma \ref{lem:expandtime} we have in both cases that $t_M(a+b)$ (respectively, $t_M(|a-b|)$) is greater than both $t_M(a)$ and $t_M(b)$. Therefore the completion algorithm would chose $a+b$ (resp, $|a-b|$) as the length of the edge instead of $M$.

\item[$K_1$-bound:] Now that we know that $abM$ is metric, we also know that it is not forbidden by the $K_1$ bound, because $M\geq K_1$.

\item[$C$-bound:] If $a+b+M\geq C$ (which includes all the triangles forbidden by $C_0$ or $C_1$ bounds), then $t_M(C-1-a-b)>t_M(a),t_M(b)$ by Lemma \ref{lem:expandtime}, so the algorithm would set $C-1-a-b$ instead of $M$ as the length of the third edge.

\item[$K_2$-bound:] Finally we deal with the $K_2$ bound. Suppose that $abM$ is metric, its perimeter is less than $C$, and it is forbidden by the $K_2$ bound. From Section \ref{sec:forbtriangles} we have that the two long edges have to be longer that $K_2$, and the shortest edge is at most $K_1$ with equality only in Case~\ref{III} with $K_1+2K_2=2\delta-1$.

As $M\leq K_2,$ we know that $M$ is the shortest edge. But also $M\geq K_1$, hence this situation can happen only when $K_1$ is the length of the shortest edge, which is only in Case~\ref{III} with $K_1+2K_2=2\delta-1$. But from Definition~\ref{defn:magiccompletion} we have in this case $M>K_1$. Hence this situation never occurs.
\end{description}
\end{proof}

It may seem strange that the algorithm does not differentiate between $C_0$ and $C_1$. The following observation justifies this
by showing that in the case where $C'>C+1$, these bounds have a relatively limited effect on the run of the algorithm.
\begin{observation}\label{obs:FCforks}
If $C'>C+1$, then either $\mathcal F^C_x$ is empty for all $x < M$ or the parameters satisfy $\ref{IIb}$. In the latter case, only $\mathcal F^C_{M-1} = \{ (\delta,\delta) \}$ is non-empty. Furthermore, in this case $t_M(M-1)$ is the maximum of the time-function. This implies that $(\delta,\delta)$-forks are completed to $M-1$ in the penultimate step of the completion algorithm.
\end{observation}
\begin{proof}
Consider a fork $(a,b) \in \mathcal F^C_x$ and the cases where $C'>C+1$ is allowed.

In Case~\ref{III} with $C'>C+1$ we have (by admissibility) $C\geq 2\delta+K_2$, so $x = C-1-a-b\geq K_2-1$ with equality only for $C=2\delta+K_2$. From the extra condition for a magic parameter in Definition~\ref{defn:magiccompletion}, we get that $M < K_2$.

In Case~\ref{IIb} we have $M=K_2=K_1=\frac{2\delta-1}{3}$, hence $C=2K_1+2K_2+1 = 2\delta+K_2$, thus again we have $C-1-a-b\geq K_2-1$. This means that the only fork in $\mathcal F^C_{M-1}$ is going to be $(\delta,\delta)$, which will be completed by $K_2-1=M-1$.

In order to see that $t_M(M-1)$ is maximal, it is enough to check $t_M(M+1) < t_M(M-1)$. We have $3M=3K_2=2\delta-1$, so by definition $t_M(M-1) = 2M-3$ and $t_M(M+1) = 2\delta-2M-2$, so we want $2M-3>2\delta-2M-2$, or $4M>2\delta+1$ which is true for $\delta\geq 5$ and this always holds in Case \ref{IIb}. So $t_M(M-1) > t_M(M+1)$ and therefore $t_M(M-1)>t_M(a)$ for any $a$ different from $M$ and $M-1$. 
\end{proof}

\begin{lem}[Optimality Lemma]\label{lem:bestcompletion}
Let $\str{G}=(G,d)\in \mathcal G^\delta$ such that it has a completion in $\mathcal A^\delta_{K_1,K_2,C_0,C_1}$. Denote by $\overbar{\str{G}}=(G,\bar d)$ the completion of $\str{G}$ with magic parameter $M$ and let $\str{G}'=(G,d')\in\mathcal A^\delta_{K_1,K_2,C_0,C_1}$ be an arbitrary completion of $\str{G}$. 
Then for every pair of vertices $u,v\in G$ one of the following holds:
\begin{enumerate}
 \item $d'(u,v) \geq \bar d(u,v) \geq M$,
 \item $d'(u,v) \leq \bar d(u,v) \leq M$,
 \item the parameters $(\delta,K_1,K_2,C_0,C_1)$ satisfy Case~\ref{IIb}, $\bar d(u,v) = M-1$, $d'(u,v) > M$ and $d'(u,v)$ has the same parity as $\bar d(u,v)$.
\end{enumerate}
\end{lem}
Note that for $\bar d(u,v)=M$ the statement trivially holds.

\begin{proof}
Suppose that the statement is not true and take any witness $\str{G}'=(G, d')$ (i.e. a completion of ${\str{G}}$ into $\mathcal A^\delta_{K_1,K_2,C_0,C_1}$ such that there is a pair of vertices violating the statement). Recall that the completion with magic parameter $M$ is defined as a limit of a sequence $\str{G}_1, \str{G}_2, \ldots$ of edge-labelled graphs such that $\str{G}_1={\str{G}}$ and each two subsequent graphs differ at most by adding edges of a single distance.

Take the smallest $i$ such that in the graph $\str{G}_i = (G,d_i)$ there are vertices $u,v\in G$ with $d_i(u,v) > M$ and $d_i(u,v) > d'(u,v)$ or $d_i(u,v) < M$ and $d_i(u,v) < d'(u,v)$. Let $w\in G$ be the witness of $d_i(u,v)$. In Case~\ref{IIb}, by Observation~\ref{obs:FCforks} edges of length $M-1$ are added in the last step of our completion algorithm. Therefore we know that the distances $d_{i-1}(u,w)$ and $d_{i-1}(v,w)$ satisfy the optimality conditions in point 1 or 2.

We shall distinguish three cases, based on whether $d_{i}(u,v)$ was introduced by $\mathcal F^-$, $\mathcal F^+$ or $\mathcal F^C$:

\paragraph{$\mathcal F^-$ case} We have $M < d_i(u,v) = |d_{i-1}(u,w)-d_{i-1}(v,w)|$. Without loss of generality let us assume $d_{i-1}(u,w) > d_{i-1}(v,w)$, which means that $d_{i-1}(u,w) > M$ and $d_{i-1}(v,w) < M$ (as $M\geq \left\lceil\frac{\delta}{2}\right\rceil$). From the minimality of $i$, it follows that $d'(u,w) \geq d_{i-1}(u,w)$ and $d'(v,w)\leq d_{i-1}(v,w)$. Since $\str{G}'$ is metric we have $d_i(u,v) = d_{i-1}(u,w)-d_{i-1}(v,w) \leq d'(u,w)-d'(v,w) \leq d'(u,v)$, which is a contradiction.

\paragraph{$\mathcal F^+$ case} We have $M > d_i(u,v) = d_{i-1}(u,w)+d_{i-1}(v,w)$, hence $d_{i-1}(u,w),\allowbreak d_{i-1}(v,w)<M$. By the minimality of $i$ we have $d'(u,w)\leq d_{i-1}(u,w)$ and $d'(v,w)\leq d_{i-1}(v,w)$. Since $\str{G}'$ is metric, we get $d'(u,v)\leq d_i(u,v)$, which contradicts our assumptions.

\paragraph{$\mathcal F^C$ case} We have $M > d_i(u,v) = C-1-d_{i-1}(u,w)-d_{i-1}(v,w)$.

First suppose that $C'=C+1$. Recall that, by the admissibility of $C$, we have $C-1\geq 2\delta+1$ and $M\leq \left\lfloor \frac{C-\delta-1}{2} \right\rfloor$. Thus we get $d_{i-1}(u,w),d_{i-1}(v,w)>M$ (otherwise, if, say, $d_{i-1}(u,w)\leq M$, we obtain the contradiction $C-\delta-1\geq 2M > d_{i-1}(u,w)+d_i(u,v) = C-1-d_{i-1}(v,w) \geq C-\delta-1$). So again $d'(u,w)\geq d_{i-1}(u,w)$ and $d'(v,w)\geq d_{i-1}(v,w)$, which means that the triangle $u,v,w$ in $\str{G}'$ is forbidden by the $C$ bound, which is absurd as $\str{G}'$ is a completion of $\str{G}$ in $\mathcal A^\delta_{K_1,K_2,C_0,C_1}$.

\medskip

It remains to discuss the case where $C' > C+1$. By Observation~\ref{obs:FCforks}, we only need to consider Case~\ref{IIb}, $d_{i}(u,v) = K_2-1 = M-1$ and $d_{i-1}(u,w) = d_{i-1}(v,w) = \delta$. By our assumption we have $d'(u,v) > d_i(u,v)$. Hence if $d'(u,v) \geq M$ it has to have the same parity as $d_i(u,v)$ (otherwise the triangle $u,v,w$ would be forbidden in $\str{G'}$ by the $C$ bound).
\end{proof}

Next we show that the algorithm initially runs in a way that preserves the parity of completions to $\mathcal A^\delta_{K_1,K_2,C_0,C_1}$.

\begin{lem}[Parity Lemma]\label{lem:sameparity}
Let $\str{G}$, $\overbar{\str G}$ and $\str{G'}$ be as in Lemma \ref{lem:bestcompletion}. Then for every pair of vertices $u,v\in G$ such that either $\bar d(u,v) \leq \min(K_1,M-1)$ or $\bar d(u,v) \geq \max(K_2,M+1)$, at least one of the following holds:
\begin{enumerate}
\item The parity of $\bar d(u,v)$ is the same as the parity of $d'(u,v)$;
\item the parameters come from Case~\ref{III}, $C=2\delta+K_1+1$, $C\neq  2K_1+2K_2+1$, $M>K_1>1$ and $\bar d(u,v)=K_1$.
\end{enumerate}
\end{lem}

Note that we are only interested in distances not equal to $M$. 

\begin{proof}
Suppose that the statement is not true, and let $\str{G}'=(G, d')$ be a counterexample. Recall that the completion with magic parameter $M$ is defined as a limit of a sequence $\str{G}_1, \str{G}_2, \ldots$ of edge-labelled graphs  such that $\str{G}_1={\str{G}}$ and each two subsequent edge-labelled graphs differ at most by adding edges of a single distance.

Take the smallest $i$ such that in $\str{G}_i = (G,d_i)$ there are vertices $u,v\in G$ with $d_i(u,v)$ and $d'(u,v)$ not satisfying the lemma. Denote by $w$ a witness of the distance $d_i(u,v)$. As in the proof Lemma~\ref{lem:bestcompletion}, we can argue that $d_{i-1}(u,w)$ respectively $d_{i-1}(v,w)$ satisfy the optimality conditions 1 or 2 in Lemma~\ref{lem:bestcompletion}.

First we will show that the exceptional case 2 from the statement only happens at the very end of the induction, hence when using the induction hypothesis (or minimality of $i$), we can work only with the first part of the statement.

Suppose that the parameters satisfy Case~\ref{III} and further $C=2\delta+K_1+1$, $C\neq 2K_1+2K_2+1$ and $M>K_1>1$. We have $t_M(K_1) > t_M(a)$ for any distance $a<K_1$ and also, by admissibility, $t_M(K_1) > t_M(b)$ for any distance $b\geq K_2$ and $b>M$: since $t_M(K_1) = 2K_1-1$ and $t_M(b) \leq 2\delta-2K_2$, we need to verify that $2K_1-1 > 2\delta-2K_2$ and thus $2K_1+2K_2 > 2\delta+1$. By admissibility it follows $2K_2+K_1\geq 2\delta$ (when $2K_2+K_1=2\delta-1$, admissibility implies $C\geq 2\delta+K_1+2$), which give the desired bound.

\medskip

Next observe that if $K_1=1$ then from Lemma \ref{lem:bestcompletion} we have that whenever $\bar d(u,v)=1$ for some vertices $u,v$, then in any completion the edge has also length 1, hence also fixed parity.

As in the proof of Lemma \ref{lem:bestcompletion}, we will now distinguish three cases based on whether $d_i(u,v)$ was introduced due to $\mathcal F^+$, $\mathcal F^-$ or $\mathcal F^C$:

\paragraph{$\mathcal F^+$ case} In this case $d_i(u,v)<M$ and $d_i(u,v)=d_{i-1}(u,w)+d_{i-1}(w,v)$. Because of our assumption $d_i(u,v)\leq K_1$, the perimeter of the triangle $uvw$ in $\str{G}_i$ is even and at most $2K_1$. By Lemma~\ref{lem:bestcompletion} either the third possibility happened, hence $\bar{d}(u,v)$ has the same parity as $d'(u,v)$, or we have $d'(u,v)\leq d_{i}(u,v)$, $d'(u,w)\leq d_{i-1}(u,w)$ and $d'(v,w)\leq d_{i-1}(v,w)$, hence $d'(u,v)+d'(u,w)+d'(w,v)$ is odd and smaller than $2K_1+1$. Thus the triangle $uvw$ is forbidden by the $K_1$ bound in $\str{G}'$, a contradiction.

\paragraph{$\mathcal F^-$ case} Here $d_i(u,v) > M$ and (without loss of generality) $d_i(u,v)=d_{i-1}(u,w)-d_{i-1}(w,v)$. Then the triangle $uvw$ has even perimeter with respect to $d_i$. By our assumption we have $d_i(u,v)\geq K_2$ and thus $d_{i-1}(u,w)>d_i(u,v)\geq K_2$ and $d_{i-1}(v,w)<M$.

This implies $d_i(u,v)+d_{i-1}(u,w) = 2d_i(u,v)+d_{i-1}(v,w)\geq 2K_2+d_{i-1}(v,w)$. From Lemma \ref{lem:bestcompletion} we get that $d'(u,v)\geq d_i(u,v)$, $d'(u,w)\geq d_{i-1}(u,w)$ and $d'(v,w)\leq d_{i-1}(v,w)$, hence also $d'(u,v)+d'(u,w) \geq 2K_2+d'(v,w)$ holds. Thus the triangle $uvw$ is forbidden by the $K_2$ bound in $\str{G}'$, a contradiction.

\paragraph{$\mathcal F^C$ case} Here $d_i(u,v)<M$ and $d_i(u,v)=C-1-d_{i-1}(u,w)-d_{i-1}(w,v)$. From our assumption it follows that $d_i(u,v)\leq K_1$.

In Case~\ref{III} we have $C\geq 2\delta+K_1+1$, hence $d_i(u,v) = K_1$ if and only if $C=2\delta+K_1+1$, $M>K_1$ and $d(u,w)= d(v,w) = \delta$; this case is treated in point 2.

It remains to consider Case~\ref{II}. Hence we can assume that $d'(u,v) < d(u,v)$ and these edges have different parity. Note that then the triangle $uvw$ has even perimeter $C-1$. By Lemma~\ref{lem:bestcompletion} we have $d'(u,w)+d'(v,w)\geq d_{i-1}(u,w)+d_{i-1}(v,w) = C-1-d(u,v) = 2K_1+2K_2-d_i(u,v)$. But as $d'(u,v) \leq d_i(u,v)\leq K_1$ we have $d'(u,w)+d'(v,w)\geq 2K_2+d'(u,v)$, so the triangle $uvw$ is forbidden by the $K_2$ bound in $\str{G}'$, a contradiction.
\end{proof}

\begin{lem}[Automorphism Preservation Lemma]
\label{lem:aut}
Let $\str{G}\in \mathcal G^\delta$ and let $\overbar{\str{G}}$ be its completion with magic parameter $M$. Then every automorphism of $\str{G}$ is also an automorphism of $\overbar{\str{G}}$.
\end{lem}
\begin{proof}
Given $\str{G}$ and an automorphism $f\colon\str{G}\to \str{G}$, it can be verified by induction that for every $k>0$, $f$ is also an automorphism graph $\str{G}_k$ as in Definition~\ref{defn:ftmcompletion}.
For every edge $x,y$ of $\str{G}_k$ which is not an edge of $\str{G}_{k-1}$, it is true that $f(x),f(y)$
is also an edge of $\str{G}_k$ which is not an edge of $\str{G}_{k-1}$, and moreover the edges $x,y$ and $f(x),f(y)$ are of the same length. This follows
directly from the definition of $\str{G}_k$.
\end{proof}

\subsection{Correctness of the completion algorithm}
\label{sec:magiccompletion}
In this section we prove:
\begin{thm} \label{thm:magiccompletion}
Let $\delta$, $K_1$, $K_2$, $C_0$ and $C_1$ be primitive admissible parameters.
Suppose that $\str{G}=(G,d)\in \mathcal G^\delta$ has a completion into $\mathcal A^\delta_{K_1,K_2,C_0,C_1}$, and let $\overbar{\str{G}}=(G,\bar{d})$ be its completion with magic parameter $M$. Then $\overbar{\str{G}}\in\mathcal A^\delta_{K_1,K_2,C_0,C_1}$.

$\overbar{\str{G}}$ is optimal in the following sense:
Let $\str{G}'=(G,d')\in\mathcal A^\delta_{K_1,K_2,C_0,C_1}$ be an arbitrary completion of $\str{G}$ in  $\mathcal A^\delta_{K_1,K_2,C_0,C_1}$, then 
for every pair of vertices $u,v\in G$ one of the following holds:
\begin{enumerate}
 \item $d'(u,v) \geq \bar{d}(u,v) \geq M$,
 \item $d'(u,v) \leq \bar{d}(u,v) \leq M$,
 \item the parameters $\delta$, $K_1$, $K_2$, $C_0$ and $C_1$ satisfy Case \ref{IIb}, $d'(u,v) \neq M$ and $\bar{d}(u,v) = M-1$.
\end{enumerate}

Finally, every automorphism of $\str{G}$ is also an automorphism of $\overbar{\str{G}}$.
\end{thm}
In the next five lemmas we will use Lemmas \ref{lem:bestcompletion} and \ref{lem:sameparity} to show that $\str{G}\in \mathcal G^\delta$ has a completion into $\mathcal A^\delta_{K_1,K_2,C_0,C_1}$, if and only if the algorithm with a magic parameter $M$ yields such a completion. We will deal with each type of forbidden triangle separately, and in doing that, we will implicitly use the results of Section~\ref{sec:forbtriangles}.

\begin{lem}[$C$-bound Lemma]\label{lem:Cbound}
Suppose $C'=C+1$, and let ${\str{G}}=(G,{d})\in \mathcal G^\delta$ be such that there is a completion of ${\str{G}}$ into $\mathcal A^\delta_{K_1,K_2,C_0,C_1}$; let $\overbar{\str{G}}=(G,\bar d)$ be its completion with magic parameter $M$. Then there is no triangle forbidden by the $C$ bound in $\overbar{\str{G}}$.
\end{lem}

\begin{proof}
Suppose for a contradiction that there is a triangle with vertices $u,v,w$ in $\overbar{\str{G}}$ such that $\bar d(u,v)+\bar d(v,w)+\bar d(u,w)\geq C$. For brevity let $a=\bar d(u,v)$, $b=\bar d(v,w)$ and $c=\bar d(u,w)$. Assume without loss of generality that $a\leq b\leq c$. Let $a',b',c'$ be the corresponding edge lengths in an arbitrary completion of ${\str{G}}$ into $\mathcal A^\delta_{K_1,K_2,C_0,C_1}$. Then two cases can appear. 

Either $a,b,c > M$, and then by Lemma \ref{lem:bestcompletion} we have $a'\geq a$, $b'\geq b$ and $c'\geq c$, so we get the contradiction $a' + b' + c' \geq C$; or $a\leq M$, $c\geq b>M$ and $a+b+c\geq C$. In this case Lemma \ref{lem:bestcompletion} implies $b'\geq b$ and $c'\geq c$ and $a'\leq a$. If the edge $(u,v)$ was already in $\str{G}$, then clearly $a' + b' + c' \geq a+b+c\geq C$, which is a contradiction. If $(u,v)$ was not already an edge in $\str{G}$, then it was added by the completion algorithm with magic parameter $M$ in step $t_M(a)$. Let $\bar{a}=C-1-b-c$. Then clearly $\bar{a}<a$, which means that $t_M(\bar{a}) < t_M(a)$, and as $\bar{a}$ depends on $b,c$, we have $t_M(b),t_M(c)<t_M(\bar{a})$. But then the completion with magic parameter $M$ actually sets the length of the edge $u,v$ to be $\bar{a}$ in step $t_M(\bar{a})$, which is a contradiction.
\end{proof}

\begin{lem}[Metric Lemma]\label{lem:metric}
Let $\str{G}$ and $\overbar{\str{G}}$ be as in Lemma \ref{lem:Cbound}. Then there are no non-metric triangles in $\str{G}$.
\end{lem}

\begin{proof}
Suppose for a contradiction that there is a triangle with vertices $u,v,w$ in $\str{G}$ such that $d(u,v)+d(v,w)<d(u,w)$. Denote $a=d(u,v)$, $b=d(v,w)$ and $c=d(u,w)$ and assume without loss of generality that $a\leq b < c$. Let $a',b',c'$ be the corresponding edge lengths in an arbitrary completion of $\bar{\str{G}}$ into $\mathcal A^\delta_{K_1,K_2,C_0,C_1}$. We shall distinguish three cases based on Section~\ref{sec:forbtriangles}:

\begin{enumerate}

\item First suppose $a,b,c < M$. Then $t_M(a)\leq t_M(b) < t_M(a+b) < t_M(c)$, which means that $c$ must be already in $\str{G}$. Note that in Case~\ref{IIb} if $b=K_1-1=M-1$, then $c\geq M$, hence we can use Lemma~\ref{lem:bestcompletion} for $a$ and $b$, which gives us that $a' + b' \leq a + b < c = c'$, which is a contradiction.

\item Another possibility is $a<M$ and $b,c\geq M$ (actually $c>M$, since $abc$ is non-metric).

Suppose $a'\leq a$ and $c'\geq c$ (the first possibility of Lemma~\ref{lem:bestcompletion}). If $b$ was already in $\str{G}$, then $\str{G}$ has no completion which is a contradiction. Otherwise clearly $c-a > b \geq M$, so $t_M(c-a) < t_M(b)$. But as $c-a$ depends on $c$ and $a$, we get $t_M(c-a) > t_M(c),t_M(a)$, which means that the completion algorithm with magic parameter $M$ would complete the edge $v,w$ with the length $c-a$ and not with $b$.

If the previous paragraph does not apply we have Case~\ref{IIb} and $a=K_1-1=K_2-1$. But then as $M=K_2$, we have $b\geq K_2$, which means $a+b\geq 2K_2-1 = \frac{4\delta-2}{3}-1\geq \delta$ for $\delta\geq 5$, which holds in \ref{IIb}, but that means that $abc$ is actually metric, a contradiction.

\item The last possibility is $a,b<M$ and $c\geq M$. Then either (by Lemma~\ref{lem:bestcompletion} and Lemma~\ref{lem:misgood} if $c=M$) we have $a'\leq a$, $b'\leq b$ and $c'\geq c$, hence the triangle $a',b',c'$ is again non-metric, or we have Case~\ref{IIb}, $b=K_1-1$, $a\leq K_1-1$.   The rest of proof of this lemma consists of a verification of this special case.
\end{enumerate}

From admissibility of \ref{IIb} we have $M=K_1=K_2=\frac{2\delta-1}{3}$ and $\delta\geq 5$. Note that $c-a \geq b+1 = K_1 = M$ from non-metricity of $abc$, hence $c>M$.

If both $a$ and $b$ were already in $\bar{\str{G}}$, then $abc$ is non-metric in any completion by Lemma~\ref{lem:bestcompletion}. The same thing is true if $b$ was already in $\bar{\str{G}}$ and $a'\leq a$ in any completion (i.e. either $a<K_1-1$ or $a$ was not introduced by $\mathcal F^C$ due to a $(\delta,\delta)$ fork).

Note that for $\delta\geq 8$ it cannot happen that $a=b=K_1-1$, as then $a+b=2K_1-2=2\frac{2\delta-1}{3} - 2 \geq \delta$, hence $a+b<c$ is absurd. So the only case when $a=b=K_1-1$ is $\delta=5$ (because from \ref{IIb} it follows that $\delta=3m+2$ for some $m\geq 1$). In that case we have triangle $5,2,2$ and each of the twos either was in $\bar{\str{G}}$ or is supported by a fork $(1,1)$ or by a fork $(5,5)$. And it can be shown that none of these structures has a strong completion into $\mathcal A^\delta_{K_1,K_2,C_0,C_1}$.

Hence $b$ was not in the input graph and $a<K_1-1$.

Observe that $c-a = M$. From non-metricity of $abc$ we have $c-a \geq b+1 = K_1=M$. And if $c-a\geq M+1$, then $t_M(c-a)\leq t_M(M+1) = 2\delta-2M-2$. And this is strictly less than $t_M(M-1) = 2M-3$ since $M=\frac{2\delta-1}{3}$ and $\delta\geq 5$. Further as $M=\frac{2\delta-1}{3}$ is odd, we see that $a$ and $c$ have different parities.

From Lemma~\ref{lem:bestcompletion} we have that in any completion $c'\geq c$ and $a'\leq a$. So the only way that the triangle $u,v,w$ can be metric is to have $b'> b$. Note that $c'-a'\geq c-a = M = K_2$, hence $c'\geq a'+K_2$. And from Lemma~\ref{lem:sameparity} we have that the parities of $a,b,c$ are preserved.

Note that as $M$ is odd, $b'$ is even. And since the parities of $c'$ and $a'$ are different, we have that $a'+b'+c'$ is odd. Also note that $c'+b'\geq a'+K_2+K_2+1 \geq 2K_2+a'$. Hence $u,v,w$ is forbidden by the $K_2$ bound in $\str{G}'$, which is a contradiction.
\end{proof}

\begin{lem}[$K_1$-bound Lemma]
\label{lem:K1bound}
Let $\str{G}$, $\overbar{\str{G}}$ be as in Lemma \ref{lem:Cbound}. Then there are no triangles forbidden by the $K_1$-bound in $\overbar{\str{G}}$.
\end{lem}

\begin{proof}
Suppose for a contradiction that there is a metric (from Lemma~\ref{lem:metric} we already know that all triangles in $\str{G}$ are metric) triangle with vertices $u,v,w$ in $\str{G}$ such that $\bar d(u,v)+\bar d(v,w)+\bar d(u,w)$ is odd and less than $2K_1+1$. Denote $a=\bar d(u,v)$, $b=\bar d(v,w)$ and $c=\bar d(u,w)$. From Section~\ref{sec:forbtriangles} we get $a,b,c < K_1\leq M$.

First suppose that Lemma~\ref{lem:bestcompletion} gives us that for any completion $a',b',c'$ that $a'\leq a$, $b'\leq b$ and $c'\leq c$. Also $a$ has the same parity as $a'$, $b$ as $b'$ and $c$ as $c'$ by Lemma~\ref{lem:sameparity}, hence $a'+b'+c'\leq a+b+c$ and those two expressions have the same parity, hence $a',b',c'$ is also forbidden by the $K_1$ bound, a contradiction.

Otherwise we have Case~\ref{IIb} and $c=K_1-1$. But then from metricity of $abc$ either $a+b=c$ (but then $a+b+c$ is even, a contradiction), or $a+b=c+1$ (if $a+b\geq c+2$, then the perimeter of the triangle is too large to be forbidden by the $K_1$ bound). But again in any completion $a'\leq a$ and $b'\leq b$, so either $c'\leq c$ or $c'=c+1$ (from metricity). From Lemma~\ref{lem:sameparity} we know that the parity of $c$ is preserved, hence $c'=c+1$ is absurd, so $a'\leq a$, $b'\leq b$ and $c'\leq c$, and we can apply the same argument as in the preceding paragraph.
\end{proof}

\begin{lem}[$C_0,C_1$-bound Lemma]
\label{lem:C01bound}
Let $C'>C+1$ and let $\str{G}$, $\overbar{\str{G}}$ be as in Lemma \ref{lem:Cbound}. Then there are no triangles forbidden by either of the $C_0$ and $C_1$ bounds in $\overbar{\str{G}}$.
\end{lem}

\begin{proof}
Suppose for a contradiction that there is a triangle with vertices $u,v,w$ in $\overbar{\str{G}}$, such that $a+b+c\geq C$ and has parity such that it is forbidden by one of the $C$ bounds, where $a=\bar d(u,v)$, $b=\bar d(v,w)$ and $c=\bar d(u,w)$.

In Case~\ref{IIb}, we have $K_1=K_2$, $C=2K_1+2K_2+1 = 4K_2+1$ and $3K_2=2\delta-1$, hence $C=2\delta + K_2$. For parameters from Case~\ref{III} we have $C\geq 2\delta+K_2$, which means that we always have $C\geq 2\delta+K_2$. This implies that $b,c>K_2\geq M$ and $a\geq K_2$. If $a$ was already present in $\str{G}$, then by Lemmas \ref{lem:misgood}, \ref{lem:bestcompletion} and \ref{lem:sameparity} we have that any completion $a',b',c'$ has $a'=a$, $b'\geq b$ and $c'\geq c$ and the parities are preserved, hence $a',b',c'$ is forbidden by the $C$ bound as well, a contradiction to $\str{G}$ having a completion. If $a$ is not in $\str{G}$, we have $a\neq M$ (by Lemma~\ref{lem:misgood}) and actually $a>M$ as $a\geq K_2\geq M$. Thus we can again use Lemmas \ref{lem:bestcompletion} and \ref{lem:sameparity} to get a contradiction.
\end{proof}

\begin{lem}[$K_2$-bound Lemma]
\label{lem:K2bound}
Let $\str{G}$, $\overbar{\str{G}}$ be as in Lemma \ref{lem:Cbound}.  Then there are no triangles forbidden by the $K_2$-bound in $\overbar{\str{G}}$.
\end{lem}

\begin{proof}
We know that all triangles are metric and not forbidden by the $C$ bounds. Suppose for a contradiction that there is a triangle with vertices $u,v,w$ in $\overbar{\str{G}}$ such that $b+c\geq 2K_2+a+1$, where $a=\bar d(u,v)$, $b=\bar d(v,w)$ and $c=\bar d(u,w)$. We know that $b,c > K_2$ and $a\leq K_1$ by Section~\ref{sec:forbtriangles}, where equality can occur only in Case~\ref{III} when $2K_2+K_1=2\delta-1$ and furthermore $M > a$ (because of Definition~\ref{defn:magiccompletion}).

Note that from the conditions for Case~\ref{III}, we know that if $2K_2+K_1=2\delta-1$, then $C\geq 2\delta+K_1+2$, which means that for edges $a,b,c$ Lemma~\ref{lem:sameparity} guarantees that the parity is preserved.

Unless $a=K_1-1$ and Case~\ref{IIb}, Lemmas~\ref{lem:bestcompletion} and~\ref{lem:sameparity} yield that $a'+b'+c'$ has the same parity as $a+b+c$ for any completion $a',b',c'$ and $b'\geq b$, $c'\geq c$ and $a'\leq a$, hence triangle $u,v,w$ is forbidden by the $K_2$ bound in any completion of ${\str{G}}$, which is a contradiction.

The last case remaining is $a=K_1-1=K_2-1$, Case~\ref{IIb}. But then $b+c\geq 2K_2+a+1 = 3K_2=2\delta-1$, so either $b+c=2\delta-1$, or $b+c=2\delta$. But from being forbidden by the $K_2$-bound we know that $a+b+c$ is odd, hence $b+c$ has different parity than $a$. And we know that $a=K_2-1 = \frac{2\delta-1}{3}-1$, which is even, hence $b+c=2\delta-1$. We also know that parities are preserved, so if $a'\geq K_2+1$, then $a'+b'+c'\geq 2\delta + K_2$ and it is thus forbidden by one of the $C$ bounds.
\end{proof}
\begin{proof}[Proof of Theorem~\ref{thm:magiccompletion}]
From the Lemmas~\ref{lem:Cbound}, \ref{lem:metric}, \ref{lem:K1bound}, \ref{lem:C01bound}, \ref{lem:K2bound} we conclude that the algorithm will correctly complete every graph $\str{G}$
which has completion into $\mathcal A^\delta_{K_1,K_2,C_0,C_1}$. The optimality statement follows by Lemma~\ref{lem:bestcompletion}. Automorphisms are preserved according to Lemma~\ref{lem:aut}.
\end{proof}

\subsection{Stationary independence relation}
In this section we show a corollary to Theorem~\ref{thm:magiccompletion} proving that the completion with magic parameter $M$ has the right properties needed
to define a stationary independence relation.

As pointed out in \cite{Tent2013}, certain ``canonical'' ways of amalgamation give rise to stationary independence relations, while in \cite{Muller2016} it was observed that a stationary independence relation gives rise to a well-defined notion of amalgamation with desirable properties established in Lemma 4.4 of \cite{Muller2016}. We are going to refine and formalise this relationship to show that every stationary independence relation on a homogeneous structure induces a canonical symmetric amalgamation operator on its age, a characterisation we use repeatedly in determining whether or not various metric graphs admit (local) stationary independence relations. We start with some useful observations on stationary independence relations.

\begin{lem} \label{lem:SIRproperties}
Let $\ind$ be a (local) SIR. Then the following properties hold:
\begin{enumerate}
\item $\str{A} \ind_{\str{C}} \str{B} \leftrightarrow \langle \str{AC} \rangle \ind_{\str{C}} \str{B} \leftrightarrow \str{A} \ind_{\str{C}} \langle \str{BC} \rangle$,
\item\emph{(Transitivity)}. $\str{A} \ind_{\str{C}} \str{B}  \land \str{A} \ind_{\langle \str{BC} \rangle} \str{D} \to \str{A} \ind_{\str{C}} \str{D}$,
\item $\str{A} \ind_{\str{C}} \langle \str{BD} \rangle \leftrightarrow \str{A} \ind_{\str{C}} \str{B} \land \str{A} \ind_{\langle \str{BC} \rangle} \str{D}$.
\end{enumerate}
\end{lem}

\begin{proof}
Point (1) and (2) are shown in \cite{Tent2013} and \cite{Muller2016} respectively. In order to show (3), i.e. that the inverse direction of the implication in the Monotonicity axiom holds, note that by (1) we have $\str{A} \ind_{\langle\str{BC} \rangle} \str{D} \leftrightarrow \str{A} \ind_{\langle \str{BC} \rangle} \langle \str{BD} \rangle$. Then, by (2) $\str{A} \ind_{\str{C}} \str{B} \land \str{A} \ind_{\langle \str{BC} \rangle} \langle \str{BD} \rangle \to \str{A} \ind_{\str{C}} \langle \str{BD} \rangle$. 
\end{proof}

\begin{defn}
Let $\mathcal C$ be an amalgamation class. We say that $\oplus$ is an \emph{amalgamation operator}, if it assigns to every triple of structures $\str{A}, \str B_1, \str B_2 \in \mathcal C$ with embeddings $e_1\colon \str{A} \to \str B_1$ and $e_2\colon \str{A} \to \str B_2$ a unique amalgam, i.e. a structure $\str{D} \in \mathcal C$ and embeddings $f_1 \colon \str B_1 \to \str{D}$, $f_2 \colon \str B_2 \to \str{D}$, such that $f_1 \circ e_1 = f_2 \circ e_2$. In short, we write $\str{D} = \str B_1 \oplus_\str{A} \str B_2$. All the amalgamation operators in this paper are \emph{symmetric}\footnote{We would like to thank  A. Kwiatkowska, T. Rzepecki and R. Sullivan for pointing out the inadvertent omission to mention symmetry for amalgamation operators in an earlier version of this paper.} 
in that $\str B_1 \oplus_\str{A} \str B_2 = \str B_2 \oplus_\str{A} \str B_1$. We call $\oplus$ a \emph{local} symmetric amalgamation operator if it is only defined for non-empty $\str{A}$, and $\oplus$ is \emph{canonical} on $\mathcal C$ if additionally the following hold:
\begin{enumerate}
\item $\str B_1 \oplus_\str{A} \str B_2$ has minimal domain, i.e. it is generated by the union of $f_1(\str B_1)$ and  $f_2(\str B_2)$
\item Monotonicity: If $\str B_1 \oplus_\str{A} \str B_2 = \langle f_1(\str B_1) \cup f_2(\str B_2) \rangle$, then $\str B_1 \oplus_\str{A} \str{B_2'} = \langle f_1(\str B_1) \cup f_2(\str{B_2'}) \rangle$ for all substructures $e_2(\str{A}) \subseteq \str{B_2'} \subseteq \str B_2$
\item Associativity: $(\str{A} \oplus_{\str C_1} \str{B}) \oplus_{\str C_2} \str{D} = \str{A} \oplus_{\str C_1} (\str{B} \oplus_{\str C_2} \str{D})$.
\end{enumerate}
\end{defn}

Then the following holds:

\begin{thm} \label{thm:canonicalamalg}
A homogeneous structure $\str{M}$ admits a (local) stationary independence relation if and only if $\Age(\str{M})$ has a (local) canonical symmetric amalgamation operator. Moreover, there is a one-to-one correspondence between  (local) stationary independence relations $\ind$ and (local) canonical symmetric amalgamation operators $\oplus$ by: $\str{A} \ind_{\str{C}} \str{B}$ if and only if $\langle \str{ABC} \rangle$ is isomorphic to $\langle \str{AC} \rangle \oplus_{\str{C}} \langle \str{BC} \rangle$.
\end{thm}

\begin{proof}
We first show that every canonical symmetric amalgamation operator gives rise to a stationary independence relation. Examples of this fact were already given in \cite{Tent2013}. Let $\oplus$ be a canonical symmetric amalgamation operator on $\Age(\str{M})$. Then we define a stationary independence relation by setting $ \str{A} \ind_{\str{C}} \str{B}$ if and only if $\langle \str{ABC} \rangle$ is isomorphic to $\langle \str{AC} \rangle \oplus_{\str{C}} \langle \str{BC} \rangle$ under an isomorphism commuting with the embeddings. The axioms \ref{invariance}, \ref{existence} and \ref{stationarity} follow straightforwardly from the fact that $\oplus$ is an amalgamation operator together with the homogeneity and universality of $\str{M}$, while \ref{symmetry} follows directly from the symmetry of $\oplus$.

For \ref{monotonicity}, observe first that by the minimality of $\oplus$ we have that $\langle \str{XY} \rangle = \str{X} \oplus_{\str{X}} \langle \str{XY} \rangle$ for all $\str{X}, \langle \str{XY} \rangle \in \Age(\str{M})$. 
Let $\str{A} \ind_{\str{C}} \langle \str{BD} \rangle$; by our observation this is equivalent to 
$$\langle \str{ABDC} \rangle = \langle \str{AC} \rangle \oplus_{\str{C}} \langle \str{BDC} \rangle = \langle \str{AC} \rangle \oplus_{\str{C}} (\langle \str{BC} \rangle \oplus_{\langle \str{BC} \rangle} \langle \str{BCD} \rangle).$$
Since $\oplus$ is associative, this is equivalent to $\langle \str{ABDC} \rangle = (\langle \str{AC} \rangle \oplus_{\str{C}} \langle \str{BC} \rangle) \oplus_{\langle \str{BC} \rangle} \langle \str{BCD} \rangle$. Since $\oplus$ is monotone, this implies  $(\langle \str{AC} \rangle \oplus_{\str{C}} \langle \str{BC} \rangle) = \langle \str{ABC} \rangle$, hence $\str{A} \ind_{\str{C}} \str{B} $ and $\str{A} \ind_{\langle \str{BC} \rangle} \str{D}$. This concludes the proof that $\ind$ is a SIR.

\medskip
For the opposite direction, let $\ind$ be a stationary independence relation on the finitely generated substructures of $\str{M}$. Let $\str{A}$, $\str{B}$ and $\str{C}$ be in the age of $\str{M}$ and let $e_1 \colon \str{C} \to \str{A}$ and $e_2 \colon \str{C} \to \str{B}$ be embeddings. By the homogeneity of $\str{M}$ there are embeddings $f_1\colon \str{A} \to \str{M}$ and $f_2\colon \str{B} \to \str{M}$ such that $f_1 e_1 = f_2 e_2$. By \ref{existence} and homogeneity, there is an $\str{A}'$ such that $\str{A}' \ind_{f_1 e_1(\str{C})} f_2(\str{B})$ and such that there is an automorphism $\alpha$ of $\str{M}$ fixing $f_1 e_1(\str{C})$ that maps $f_1(\str{A})$ to $\str{A}'$. We then define $\str{A} \oplus_{\str{C}} \str{B}$ as the amalgam $ \langle \str{A}' \cup f_2(\str{B}) \rangle$ with respect to the embeddings $\alpha f_1$ and $f_2$, noting that by \ref{stationarity}, homogeneity and \ref{invariance}, the isomorphism type of $ \langle \str{A}' \cup f_2(\str{B}) \rangle$ does not depend on the choices of $\alpha, f_1$ or $f_2$ so that $\str{A} \oplus_{\str{C}} \str{B}$ is well-defined. Similarly employing \ref{symmetry} along with \ref{invariance}, \ref{stationarity} and homogeneity, we have $\str{A} \oplus_{\str{C}} \str{B} = \str{B} \oplus_{\str{C}} \str{A}$. Now by definition $\oplus$ is monotone and $\str{A} \oplus_{\str{C}} \str{B}$ has minimal domain, so $\oplus$ is a canonical symmetric amalgamation operator.

For showing associativity, assume $(\str{A} \oplus_{\str C_1} \str{B}) \oplus_{\str C_2} \str{D}$ and $\str{A} \oplus_{\str C_1} (\str{B} \oplus_{\str C_2} \str{D})$ are defined and
let $\str{A}'$ and $\str{B}'$ be such that $\str{B} \oplus_{\str C_2} \str{D} = \langle \str{B'D} \rangle$ and $\str{A} \oplus_{\str C_1} (\str{B} \oplus_{\str C_2} \str{D} )= \langle \str{A' B' D} \rangle$. 
Thus we have $\str{B}' \ind_{\str C_2} \str{D}$ and $\str{A}' \ind_{\str C_1} \langle \str{B'D} \rangle$. By Lemma \ref{lem:SIRproperties} (3) this is equivalent to 
$$\str{A}' \ind_{\str C_1} \str{B}'  \land \str{A}' \ind_{\str{B}'} \str{D} \land \str{B}' \ind_{\str C_2} \str{D}.$$
Again by Lemma \ref{lem:SIRproperties} (3), \ref{symmetry} and $\str C_2 \subseteq \str{B}'$ this is equivalent to
$$\str{A}' \ind_{\str C_1} \str{B}' \land \str{D} \ind_{\str C_2} \langle \str A'\str B' \rangle.$$
By our definition of $\oplus$ we then have
$$\langle \str A' \str B' \str D \rangle = \str{A} \oplus_{\str C_1} (\str{B} \oplus_{\str C_2} \str{D}) = (\str{A} \oplus_{\str C_1} \str{B}) \oplus_{\str C_2} \str{D},$$
proving that $\oplus$ is associative.
\end{proof}

\begin{corollary}
\label{cor:SIR}
Let $\delta$, $K_1$, $K_2$, $C_0$ and $C_1$ be primitive admissible parameters.
For every magic parameter $M$ we can define a {\em stationary independence relation with magic parameter $M$} on $\Gamma^\delta_{K_1,K_2,C_0,C_1}$ 
as follows:
 $\str{A}\ind^M_{\str{C}}\str{B}$ if and only if $\struc{\str{A}\str{B}\str{C}}$ is isomorphic to the completion with magic parameter $M$ of the free amalgamation of $\struc{\str{A}\str{C}}$ and $\struc{\str{B}\str{C}}$ over $\str{C}$.
\end{corollary}

\begin{proof}
For structures $\str{A}, \str{B}, \str{C}$ in $\mathcal{A}^\delta_{K_1,K_2,C_0,C_1}$, such that $\str{C}$ embeds into $\str{A}$ and $\str{B}$ we define $\str{A} \oplus_{\str{C}} \str{B}$ to be the completion with magic parameter $M$ of the free amalgam of $\str{A}$ and $\str{B}$ over $\str{C}$. If we can show that $\oplus$ is a canonical symmetric amalgamation operator, then we are done by Theorem \ref{thm:canonicalamalg}. By Theorem \ref{thm:magiccompletion}, $\str{A} \oplus_{\str{C}} \str{B}$ is an element of $\mathcal{A}^\delta_{K_1,K_2,C_0,C_1}$, hence $\oplus$ is a symmetric amalgamation operator. Monotonicity and associativity of $\oplus$ follow straightforwardly from the optimality property of the completion with magic parameter $M$ (Theorem \ref{thm:magiccompletion} (1),(2)).
\end{proof}

\subsection{Ramsey property and EPPA}
\label{sec:ramseyEPPA1}
Theorem \ref{thm:magiccompletion} implies the following lemma which is crucial when applying Theorems~\ref{thm:herwiglascar} and~\ref{thm:localfini}.
\begin{lem}[Finite Obstacles Lemma]
\label{lem:obstacles}
Let $\delta$, $K_1$, $K_2$, $C_0$ and $C_1$ be primitive admissible parameters.
Then the class $\mathcal A^\delta_{K_1,K_2,C_0,C_1}$ has a finite set of obstacles which are all cycles of diameter at most $2^\delta \cdot 3$.
\end{lem}
\begin{example}
Consider $\mathcal A^5_{3,3,16,13}$ discussed in Section~\ref{sec:algorithm}.
The set of obstacles of this class contains all the forbidden triangles listed
earlier, but in addition to that it also contains some cycles with 4 or more vertices. A complete
list of those can be obtained by running the algorithm backwards from the forbidden
triangles.

All such cycles with 4 vertices can be constructed from the triangles by substituting distances by the forks depicted at Figure~\ref{fig:forks}. This means substituting $2$ for $11$ or $55$, and $4$ for $15$ or $51$. With equivalent cycles removed this give the following list: 
$$
\begin{array}{rrcl}
\hbox{non-metric:}&124&\implies& 1\mathbf{11}4, 1\mathbf{55}4, 12\mathbf{15}, 12\mathbf{51}\\
&125&\implies& 1\mathbf{11}5, 1\mathbf{55}5^{**}\\
&114&\implies& 11\mathbf{15}^*\\
&225&\implies& \mathbf{11}25^{*}, \mathbf{55}25\\
\hbox{$K_1$-bound:}&122&\implies& 1\mathbf{11}2, 1\mathbf{55}2\\
\hbox{$K_2$-bound:}&144&\implies& 1\mathbf{15}4, 1\mathbf{51}4, 14\mathbf{15}\\
&245&\implies& \mathbf{11}45, \mathbf{55}45, 2\mathbf{15}5^*, 25\mathbf{15}\\
\hbox{$C$-bound:}&445&\implies& \mathbf{15}45, \mathbf{51}45, 4\mathbf{15}5\\
\end{array}
$$
Observe that running the algorithm may produce multiple forbidden triangles which leads
to duplicated cycles in the list.
Such duplicates are denoted by $*$. For example, 125 was expanded to $1\mathbf{11}5$. The algorithm will first notice the fork $(1,5)$ and produce $114$. This is also a forbidden triangle, but a different one.  In the case of $1\mathbf{55}5$ (another expansion of 125) the algorithm will again use the fork $(1,5)$ first and produce the triangles $455$ and $145$, which are valid triangles, see Figure~\ref{fig:1555}. Not all expansions here are necessarily forbidden, because not all of them correspond to a valid run of the algorithm.  However with the exception of cases denoted by $**$ all the above 4-cycles are forbidden.
\begin{figure}[t]
\centering
\includegraphics{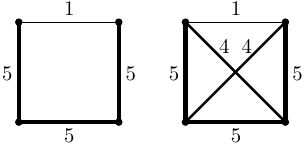}
\caption{Completing the cycle 1555.}
\label{fig:1555}
\end{figure}

Repeating the procedure one obtains the following cycles with five edges that cannot be completed into this class of metric graphs:
 $$11111,\allowbreak 11115,\allowbreak  11155,\allowbreak  11515,\allowbreak  15155,\allowbreak  11555,\allowbreak  15555,\allowbreak  55555.$$ 
An example of a failed run of algorithm trying to complete one of the forbidden cycles is depicted in Figure~\ref{fig:11555}.
\begin{figure}
\centering
\includegraphics{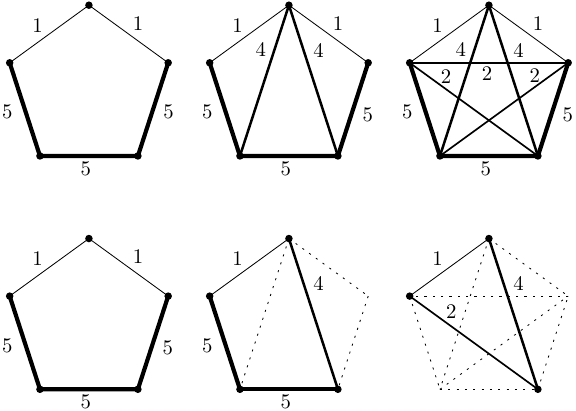}
\caption{Failed run attempting to complete the cycle 11555. In the bottom row is the backward run from the non-metric triangle 124 to the original obstacle used in the proof of Lemma~\ref{lem:obstacles}.}
\label{fig:11555}
\end{figure}
Because there are no distances of 2 or 4 in the cycles with five edges that cannot be completed into this class, it follows that all cycles with at least six edges can be completed.
\end{example}
 
\begin{proof}[Proof of Lemma \ref{lem:obstacles}]
Let $\str{G}=(G,d)\in \mathcal G^\delta$ be an edge-labelled graph no completion in
$\mathcal A^\delta_{K_1,K_2,C_0,C_1}$.  We seek a subgraph of $\str{G}$ of bounded size
which has also no completion into $\mathcal A^\delta_{K_1,K_2,C_0,C_1}$.

Consider the sequence of graphs $\str{G}_0, \str{G}_1,\ldots,\str{G}_{2M+1}$ as given by Definition~\ref{defn:magiccompletion}
when completing $\str{G}$ with magic parameter $M$.  Set $\str{G}_{2M+2}$ to be the actual completion.

Because $\str{G}_{2M+2}\notin \mathcal A^\delta_{K_1,K_2,C_0,C_1}$ we know it contains a forbidden triangle $\str{O}$.
This triangle always exists, because $\mathcal A^\delta_{K_1,K_2,C_0,C_1}$ is 3-constrained.
By backward induction on $k=2M+1,2M,\ldots, 0$ we obtain cycles $\str{O}_k$
of $\str{G}_k$ such that $\str{O}_k$ has no completion in $\mathcal A^\delta_{K_1,K_2,C_0,C_1}$
and there exists a homomorphism $f\colon\str{O}_k\to \str{G}_k$.

Put $\str{O}_{2M+1}=\str{O}$. By Lemma~\ref{lem:misgood} we know that this triangle is also in $\str{G}_{2M+1}$.
At step $k$ consider every edge $u,v$ of $\str{O}_{k+1}$ which is not an edge of $\str{G}_k$ considering
its witness $w$ (i.e. vertex $w$ such that the edges $u,w$ and $v,w$ implied the addition of the edge $u,v$) and extending
$\str{O}_k$ by a new vertex $w'$ and edges $d(u,w')=d(u,w)$ and $d(v,w')=d(v,w)$.
One can verify that the completion algorithm will fail to complete $\str{O}_k$ the same way
as it failed to complete $\str{O}_{k+1}$ and moreover there is a homomorphism $\str{O}_{k+1}\to \str{G}_k$.

At the end of this procedure we obtain $\str{O}_0$, a subgraph of $\str{G}$, that
has no completion into $\mathcal A^\delta_{K_1,K_2,C_0,C_1}$.
The bound on the size of the cycle follows from the fact that only $\delta$ steps
of the algorithm are actually changing the graph and each time every edge may
introduce at most one additional vertex.

Let $\mathcal O$ consist of all edge-labelled cycles with at most $2^\delta 3$
vertices that are not completable in $\mathcal A^\delta_{K_1,K_2,C_0,C_1}$. Clearly $\mathcal O$ is finite. To check that $\mathcal O$ is a set of obstacles it remains
to verify that there is no $\str{O}\in \mathcal O$ with a homomorphism
to some $\str{M}\in \mathcal A^\delta_{K_1,K_2,C_0,C_1}$. Denote by $\mathcal{O}'$ the set of all homomorphic images
of structures in $\mathcal{O}$ that are not completable in $\mathcal A^\delta_{K_1,K_2,C_0,C_1}$.
Assume, to the contrary, the existence of such an
$\str{O}=(O,d)\in \mathcal{O}'$ and $\str{M}=(M,d')$ and a homomorphism $f\colon\str{O}\to\str{M}$ and among all those choose one minimising 
the difference of $\lvert \str{O}\rvert$ and $\lvert \str{M}\rvert$. It follows that $\lvert \str{O}\rvert-\lvert \str{M}\rvert=1$. 
Denote by $x,y$ the pair of vertices identified by $f$. Let $\str{O}'=(O,d'')$  be
a metric graph such that $d''(z,z')=d(f(z),f(z'))$ for every pair $\{z,z'\}\neq \{x,x'\}$.
It follows that because $\mathcal A^\delta_{K_1,K_2,C_0,C_1}$ has the strong amalgamation property, and also
$\str{O}'=(O,d'')$ has a completion in $\mathcal A^\delta_{K_1,K_2,C_0,C_1}$.
\end{proof}
To demonstrate the use of these results, we show the following theorems which we later strengthen in Section~\ref{sec:henson} for
classes with Henson constraints.
\begin{thm}
\label{thm:regularramsey}
For every choice of admissible primitive parameters $\delta$, $K_1$, $K_2$, $C_0$ and $C_1$, the class $\overrightarrow{\mathcal A}^\delta_{K_1,K_2,C_0,C_1}$ of free orderings of $\mathcal A^\delta_{K_1,K_2,C_0,C_1}$ is Ramsey and
has the expansion property.
\end{thm}
\begin{proof}
Ramsey property follows by a combination of Theorem~\ref{thm:localfini} and Lemma~\ref{lem:obstacles}.

To show the expansion property we use now standard argument that edge-Ramsey implies ordering property~\cite{Nevsetvril1995,Jasinski2013}:
Given metric space $\str{A}\in \mathcal A^{\delta}_{K_1,K_2,C_0,C_1}$ construct ordered metric space $\overrightarrow{\str{B}}_0\in \overrightarrow{\mathcal A}^\delta_{K_1,K_2,C_0,C_1}$ as a disjoint union of all
possible orderings of $\str{A}$. Now consider every pair of
vertices $a < b$, $d(a,b)\neq M$ and add third vertex $c$ in distance $M$ from both $a$ and $b$ with order extended in a way so
$a < c < b$ holds. Because $M$ is magic, by Observation~\ref{obs:magicismagic}, all new triangles are allowed and thus it is possible to complete this structure to ordered metric space $\overrightarrow{\str{B}}_1\in \overrightarrow{\mathcal A}^\delta_{K_1,K_2,C_0,C_1}$.
Now denote by $\overrightarrow{\str{E}}$ an ordered metric space consisting of two vertices in distance $M$ and construct
$$\overrightarrow{\str{B}}\longrightarrow(\overrightarrow{\str{B}_1})^{\overrightarrow{\str{E}}}_2.$$

We claim that $\str{B}$ (the unordered reduct of $\overrightarrow{\str{B}}$) has the property that every ordering of $\str{B}$
contains every ordering of $\str{A}$.  Denote by $\leq$ the order of $\overrightarrow{\str{B}}$ and chose arbitrary linear order order $\leq'$ of vertices of $\str{B}$. $\leq'$ implies two-coloring of copies of $\overrightarrow{\str{E}}$
in $\overrightarrow{\str{B}}$: color copy red if both orders agree and blue otherwise.  Because $\overrightarrow{\str{B}}$ is Ramsey, we obtain a monochromatic
copy of $\overrightarrow{\str{B}}_1$ which contains a copy of $\overrightarrow{\str{B}}_0$ with the property that $\leq'$ restricted to this copy either agrees either with $\leq$ or with $\geq$.
In both cases we obtain a copy of every ordering of $\str{A}$ within this copy of $\overrightarrow{\str{B}}_0$.

\end{proof}
\begin{thm}
\label{thm:regulareppa}
For every choice of admissible primitive parameters $\delta$, $K_1$, $K_2$, $C_0$ and $C_1$, the class $\mathcal A^\delta_{K_1,K_2,C_0,C_1}$  has coherent EPPA.
\end{thm}
\begin{proof}
This follows by a combination of Theorem~\ref{thm:herwiglascar}, Lemma~\ref{lem:obstacles} and Lemma~\ref{lem:aut}.
\end{proof}
\section{Primitive spaces with Henson constraints}
\label{sec:henson}
In this section we extend the results of Section~\ref{sec:basic3} to classes with Henson constraints.
An important fact about spaces with Henson constraints is the following~\cite[Section 1.3]{Amato2016}:

\begin{remark}\label{rem:Finite_Henson}
Let $\mathcal{O}$ be a set of $(1,\delta)$-spaces. Then there exists a finite set $\mathcal{S}\subset\mathcal{O}$, such that $\Forb(\mathcal{S})=\Forb(\mathcal{O})$.
\end{remark}

A reflexive and transitive relation on a set $X$ is said to be a quasi order. It is a well quasi order if every non-empty subset has at least one, but only finitely many, minimal elements. Now, up to isometry, $(1,\delta)$ spaces are characterised by the sizes of their cliques i.e., maximal subsets of vertices at mutual distance 1. If we express every such space as a finite non-decreasing sequence of non negative integers we see that a $(1,\delta)$-space is embeddable in another if and only if the latter contains a subsequence which majorises the former term by term. It is a well known fact that the set of finite sequences of a well quasi ordered set is a well quasi ordered set with respect to the above relation, see ~\cite{higman1952ordering}.

\begin{lem}
\label{lem:hensoncompletions}
Let $(\delta,K_1,K_2,\allowbreak C_0,C_1,\mathcal S)$ be a primitive admissible sequence of parameters and let $\mathcal S$ be a non-empty class of Henson constraints.
There is an $M$ such that the completion with magic parameter $M$ does not introduce
distances $1$ or $\delta$.
\end{lem}
\begin{proof}
For $\delta>3$ the completion algorithm with magic parameter $M \geq \lceil \frac{\delta}{2} \rceil$ introduces distances $\delta$ or $1$ either when $M=\delta$ or when $1$ is the only possible completion of some fork.

For $K_1< \delta$ we can choose $M$ to be smaller than $\delta$, hence the first case only appears if $K_1 = \delta$. However then by Theorem~\ref{thm:admissible_henson}, $\mathcal S=\emptyset$.
The second case only appears if $C=2\delta+2$ and $C' = C+1$, since then $(\delta,\delta)$-forks can only be completed by distance $1$. Then again, by Theorem~\ref{thm:admissible_henson}, $\mathcal S=\emptyset$.
\end{proof}
\begin{thm}
\label{thm:SIRhenson}
Let $(\delta,K_1,K_2,\allowbreak C_0,C_1,\mathcal S)$ be a primitive admissible sequence of parameters and let $\mathcal S$ be a non-empty class of Henson constraints.
For every magic parameter $M$ we can define a {\em stationary independence relation with magic parameter $M$} on $\Gamma^\delta_{K_1,K_2,C_0,C_1,\mathcal S}$ 
as follows:
 $\str{A}\ind_{\str{C}}\str{B}$ if and only if $\struc{\str{A}\str{B}\str{C}}$ is isomorphic to the completion with magic parameter $M$ of the free amalgamation of $\struc{\str{A}\str{C}}$ and $\struc{\str{B}\str{C}}$ over $\str{C}$.
\end{thm}
\begin{proof}
This result can be shown as Corollary~\ref{cor:SIR}, the correctness of the completion algorithm can be verified by a combination of Theorem~\ref{thm:magiccompletion} and Lemma~\ref{lem:hensoncompletions}.
\end{proof}
\begin{thm}
\label{thm:ramseyHenson}
Let $(\delta,K_1,K_2,\allowbreak C_0,C_1,\mathcal S)$ be a primitive admissible sequence
of parameters such that $\mathcal S$ is a non-empty class of Henson constraints. Then the class of $\overrightarrow{\K}$ of free
linear orderings of $\K=\mathcal A^\delta_{K_1,K_2,\allowbreak C_0,C_1}\cap
\mathcal A^\delta_\mathcal S$ is Ramsey and has the expansion property. 
\end{thm}
\begin{proof}
Recall the definition of locally finite subclasses in Definition~\ref{def:localfinite}.
We show that $\overrightarrow{\mathcal K}$ is a locally finite subclass of $\overrightarrow{\mathcal A}^\delta_{K_1,K_2,\allowbreak C_0,C_1}$. Given $\str{C}_0\in \overrightarrow{\mathcal A}^\delta_{K_1,K_2,\allowbreak C_0,C_1}$ we put $n=|C_0|$.
Consider any $\str{C}$ with a homomorphism to $\str{C}_0$ such that every substructure with at most $n$ vertices can be completed to $\str{C}'\in \overrightarrow{\K}$. Because Henson constrains are complete edge-labelled graphs, we know that every Henson constraint in $\str{C}$ is mapped to an isomorphic copy of itself by any homomorphism $\str{C} \to \str{C}_0$. Therefore there there is no $\str{H}\in \mathcal S$, $|H|\geq n$ such that $\str{H}\to \str{C}$. Because every subgraph with at most $n$ vertices can be completed to $\overrightarrow{\K}$ we also know that there is no $\str{H}\in \mathcal S$, $|H|\leq n$ such that $\str{H}\to \str{C}$.

By Lemma~\ref{lem:hensoncompletions} we know that the magic completion $\str{C}'$ of $\str{C}$ (which exists by Theorem~\ref{thm:magiccompletion}) will not introduce any edges of distance $1$ and $\delta$ and thus also $\str{C}'$ contains no forbidden Henson constraints.
The Ramsey property follows by Theorem~\ref{thm:localfini}.
\end{proof}
\begin{thm}
\label{thm:EPPAHenson}
Let $(\delta,K_1,K_2,\allowbreak C_0,C_1,\mathcal S)$ be a primitive admissible sequence of parameters such that $S$ is a non-empty class of Henson constraints. Then $\mathcal A^\delta_{K_1,K_2,\allowbreak C_0,C_1}\cap     \mathcal A^\delta_\mathcal S$ has coherent EPPA.
\end{thm}
\begin{proof}
By Lemma~\ref{lem:obstacles} there is a finite set of obstacles $\mathcal O$ for  $\mathcal A^\delta_{K_1,K_2,\allowbreak C_0,C_1}$.
By Remark~\ref{rem:Finite_Henson} $\mathcal S$ is finite. Given $A\in \K$, apply Theorem~\ref{thm:herwiglascar} to obtain an EPPA-witness $\str{B}\in \Forb(\mathcal O\cup \mathcal S)$.
Denote by $\str{C}$ its completion with magic parameter $M$. By Lemma~\ref{lem:hensoncompletions} we know that $\str{C}\in \Forb(\mathcal S)$
and by Lemma~\ref{thm:magiccompletion} we know that the automorphism group is unaffected. It follows that $\str{C}$ is the desired completion.
\end{proof}

\section{Bipartite 3-constrained spaces}
\label{sec:bipartite}

In this section we discuss the bipartite classes of finite diameter in Cherlin's catalogue (Case \ref{I} in Theorem \ref{thm:admissible}). These are classes of metric spaces $\mathcal A^\delta_{K_1,K_2,C_0,C_1}$ with parameters $$\delta < \infty, K_1=\infty, K_2=0, C_1=2\delta+1.$$ Furthermore we assume that $$C_0 > 2\delta +3.$$ The antipodal case where $C_0 = 2\delta +2$ will be treated in Section \ref{sec:antipodal}. The parameter $C_0$ has to be even, so $2\delta+3$ is not an acceptable value for $C_0$.  We also discuss the Henson constraints for bipartite graphs.

By the condition $K_1=\infty$, the metric spaces in $\mathcal A^\delta_{\infty,0,C_0,2\delta+1}$ contain no triangles of odd perimeter. As a direct consequence the relation consisting of all pairs $(x,y)$ such that $d(x,y)$ is even, is an equivalence relation with two equivalence classes; this fact also motivates the name ``bipartite 3-constrained spaces''.

\subsection{Generalised completion algorithm for bipartite 3-con\-strained classes}
Our aim in this section is to again describe a procedure that completes a given edge-labelled graph $\str{G}$ to metric spaces in $\mathcal A^\delta_{\infty,0,C_0,2\delta+1}$ whenever possible. The completion algorithm constructed in Section \ref{sec:algorithm} fails in general in the bipartite setting, since it might introduce triangles with odd perimeter (for instance when adding the magic distance in the final step). Hence Theorem \ref{thm:magiccompletion} cannot be applied here.

In the following we show how we can slightly adapt the algorithm to ensure that no new odd cycles are generated. The basic idea for our completion algorithm is again to optimize the length of the newly introduced edges towards a magic parameter $M$, respectively its successor $M+1$. The length of the remaining edges are then set to $M$ or $M+1$ depending on the parity prescribed by the bipartition.

\begin{defn}[Bipartite magic distances]
\label{defn:bimagiccompletion}
Let $M\in\{1,2,\ldots, \delta\}$ be a distance. We say that $M$ is {\em magic (with respect to $\mathcal A^\delta_{\infty,0,C_0,2\delta+1}$)} if $$\left\lfloor\frac{\delta}{2}\right\rfloor \leq M < M+1 \leq \left\lfloor\frac{C_0-\delta-1}{2}\right\rfloor.$$
\end{defn}
Compare this with Definition~\ref{defn:magicdistance}. By the assumption $C_0 > 2\delta +3$ a magic distance always exists. 
\begin{observation}\label{obs:bimagicismagic}
$M$ is magic if and only if it has the following property:
for all even $1< b\leq \delta$ the triangles $MMb$ and $(M+1)(M+1)b$ are in $\mathcal A^\delta_{\infty,0,C_0,2\delta+1}$; and for all odd $1\leq b\leq \delta$ the triangle $M(M+1)b$ is in $\mathcal A^\delta_{\infty,0,C_0,2\delta+1}$.
\end{observation}

\begin{proof}
If $M$ has this property then $M\geq \left\lfloor\frac{\delta}{2}\right\rfloor$ (otherwise for even $\delta$ the triangle $MM\delta$ would be non-metric; similarly for odd $\delta$ the triangle $M(M+1)\delta$ would be non-metric). Also, $M\leq \left\lfloor\frac{C_0-\delta-1}{2}\right\rfloor$ (otherwise if $\delta$ is even the triangle $(M+1)(M+1)\delta$ has perimeter $C_0$, hence is forbidden by the $C_0$ bound; and if $\delta$ is odd the triangle $M(M+1)\delta$ has perimeter $C_0$). The other implication follows from the definition of $\mathcal A^\delta_{\infty,0,C_0,2\delta+1}$.
\end{proof}

With the following simple lemma we reduce our discussion to completions of \textit{connected} edge-labelled graphs, i.e. edge-labelled graphs such that there exists a path connecting each pair of vertices.

\begin{lem} \label{lem:biconnected}
$\str{G}\in \mathcal G^\delta$ has a completion to $\mathcal A^\delta_{\infty,0,C_0,2\delta+1}$ if and only if all of its connected components have a completion to $\mathcal A^\delta_{\infty,0,C_0,2\delta+1}$.
\end{lem}

See Remark~\ref{rem:biunconnected} for more discussion about the disconnected case.

\begin{proof}
It suffices to show that every $\str{G}=(G,d)$ that is the disjoint union of two graphs $A,B$ from $\mathcal A^\delta_{\infty,0,C_0,2\delta+1}$ has a completion $(G,d') \in \mathcal A^\delta_{\infty,0,C_0,2\delta+1}$. 

Fix some $x\in A$ and $y\in B$. Then, for every non-edge $(x',y')$ with $x'\in A$ and $y'\in B$ let $d'(x',y') = M$ if $d(x,x') + d(y,y')$ is even and $d'(x',y') = M+1$ otherwise. It is not hard to verify that all the newly introduced triangles are of the form $MMb$, $(M+1)(M+1)b$ where $b$ is even, or $M(M+1)b$ where $b$ is odd. Hence $(G,d') \in \mathcal A^\delta_{\infty,0,C_0,2\delta+1}$.
\end{proof}

For connected graphs we now give the following definition of a completion algorithm:

\begin{defn}[Bipartite completion algorithm]\label{defn:biftmcompletion}
Given $1\leq M\leq M+1\leq \delta$, a one-to-one function $t\colon\{1,2,\ldots,\delta\}\setminus \{M,M+1\}\to \mathbb N$ 
 and a function $\mathbb F$ from $\{1,2,\ldots,\delta\}\setminus \{M,M+1\}$ to the power set of $\mathcal D$, then we define the {\em $(\mathbb F,t,M,M+1)$-completion} of a connected edge-labelled graph $\str{G} = (G,d)$ as the limit of the sequence $\str{G}_1, \str{G}_2,\ldots$ that is constructed as in Definition \ref{defn:ftmcompletion}; the length of all remaining non-edges $(u,v)$ in this limit is then set to $M$ if $d^+(u,v)$ has the same parity as $M$ and to $M+1$ otherwise. In addition we will stick to the other notational conventions introduced in Section \ref{sec:algorithm} (time function, $a$ depends on $b$, \ldots).
\end{defn}

Let $M$ be a magic distance and let $1\leq x\leq \delta$ with $x\neq M, M+1$. Then we define the fork sets $\mathcal F^{+}_x$ and $\mathcal F^{-}_x$ as in the last section and $\mathcal F^{C_0}_x = \left\{(a,b)\in \mathcal D : C_0-2-a-b=x\right\}$, i.e. $(a,b) \in \mathcal F^{C_0}_x$ if and only if $a+b+x$ is equal to $C_0 -2$ and hence also even.
\begin{figure}
\centering
\includegraphics{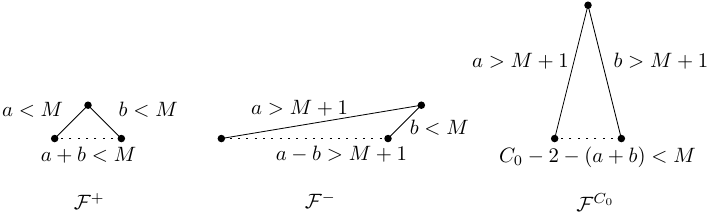}
\caption{Forks used by $\mathbb F$ by the bipartite algorithm.}
\label{fig:BFforks}
\end{figure}
Forks used by the bipartite algorithm are schematically depicted in Figure~\ref{fig:BFforks}.

We further define
$$\mathbb F_M(x) =
\begin{cases} 
      \mathcal F^+_x\cup \mathcal F^{C_0}_x & x < M \\
      \mathcal F^-_x & x > M +1.
\end{cases}
$$
For a magic distance $M$, we define the function $t_M\colon \{1,\ldots,\delta\}\setminus \{M, M+1\} \rightarrow \mathbb N$ as
$$t_M(x) =
\begin{cases} 
      2x-1 & x < M \\
      2(\delta-x) & x > M+1.
\end{cases}
$$

We then call the $(\mathbb F_M,t_M,M,M+1)$-completion of $\str{G}$ the {\em bipartite completion of $\str{G}$ with magic parameter $M$}.

\begin{lem}[Bipartite Time Consistency Lemma]\label{lem:biexpandtime}
Let $a,b$ be distances different from $M$ and $M+1$. If $a$ depends on $b$, then $t_M(a) > t_M(b)$.
\end{lem}

\begin{proof}
Analogously to Lemma~\ref{lem:expandtime} we consider three types of forks.
\begin{description}
\item[$\mathcal F^+$:]
This follows in complete analogy to Lemma~\ref{lem:expandtime}: If $a<M$ and $\mathcal F^+_a\neq\emptyset$, then $b<a<M$, hence $t_M(b) < t_M(a)$.

\item[$\mathcal F^{C_0}$:]
If $a < M$ and $\mathcal F^{C_0}_a\neq\emptyset$, then we must have $b,c > M+1$. Otherwise, if for instance $b\leq M+1$, then $C_0-\delta-2\leq C_0-2-c = a+b < 2M+1 \leq 2\left\lfloor \frac{C_0-\delta-1}{2} \right\rfloor - 2 + 1$, a contradiction. As $C_0 \geq 2\delta+4$, we obtain the inequality $b = (C_0-2)-c-a \geq (2\delta + 2) - \delta - a = \delta+2-a$. Hence $t_M(b) \leq 2(a-2) < 2a-1 = t_M(a)$.

\item[$\mathcal F^-$:]
Finally, we consider the case where $a > M+1$ and $\mathcal F^-_a\neq\emptyset$. Then either $a = b-c$, which implies $b>a>M+1$ and thus $t_M(b)<t_M(a)$, or $a = c-b$, which means $b = c-a\leq \delta-a$. Because of $a>M+1\geq \left\lceil\frac{\delta}{2}\right\rceil$, we have $b<M$. So $t_M(b) \leq 2(\delta-a) - 1 < 2(\delta-a) = t_M(a)$.
\end{description}
\end{proof}

\begin{lem}[Bipartite $\mathbb F_M$ Completeness Lemma]\label{lem:bimisgood}
Let ${\str{G}}\in \mathcal G^\delta$ and $\overbar{\str{G}}$ be its bipartite completion with magic parameter $M$. If there is a triangle of even perimeter in $\overbar{\str{G}}$ that is either non-metric or forbidden due to the $C_0$-bound, and that contains an edge of length $M$ or $M+1$, then that edge was already in ${\str{G}}$.
\end{lem}
\begin{proof}
By Observation \ref{obs:bimagicismagic} no triangles of type $aMM$, $aM(M+1)$ or $a(M+1)(M+1)$ of even perimeter are forbidden. 
Suppose then that there is a forbidden triangle $abN$ of even perimeter in $\overbar{\str{G}}$ such that $M\leq N\leq M+1$ and the edge of length $N$ is not in ${\str{G}}$.  For convenience define $t_M(M) = t_M(M+1) = \infty$, which corresponds to the fact that edges of lengths $M$ and $M+1$ are added in the last step.
\begin{description}
\item[non-metric:] If $abN$ is non-metric then either $a+b<N$ or $|a-b|>N$. By Lemma \ref{lem:biexpandtime} we have in both cases that $t_M(a+b)$ (respectively, $t_M(|a-b|)$) is greater than both $t_M(a)$ and $t_M(b)$. Therefore the completion algorithm would chose $a+b$ (resp. $|a-b|$) as the length of the edge instead of $N$. Observe that because $abN$ has even perimeter it follows that this value differs from $N$ by at least 2, so it is not equal to $M$ or $M+1$.
\item[$C_0$-bound:] If $a+b+N\geq C_0$ (which includes all the triangles forbidden by $C_0$-bound), then $t_M(C_0-2-a-b)>t_M(a),t_M(b)$ by Lemma \ref{lem:biexpandtime}, so the algorithm would set $C_0-2-a-b$ instead of $N$ as the length of the third edge.
\end{description}
\end{proof}

\begin{lem}[Bipartite Optimality and Parity Lemma]\label{lem:bibestcompletion}
Let ${\str{G}}=(G,d)\in \mathcal G^\delta$ be connected such that there is a completion of $\str{G}$ into $\mathcal A^\delta_{\infty,0,C_0,2\delta+1}$. Denote by $\overbar{\str{G}}=(G,\bar d)$ its bipartite completion of $\str{G}$ with magic parameter $M$ and let $\str{G}'=(G,d')\in\mathcal A^\delta_{\infty,0,C_0,2\delta+1}$ be an arbitrary completion of $\str{G}$.
Then for every pair of vertices $u,v\in G$ one of the following holds:
\begin{enumerate}
 \item $d'(u,v) \geq \bar d(u,v) \geq M + 1$,
 \item $d'(u,v) \leq \bar d(u,v) \leq M$,
 \item $M\leq \bar d(u,v)\leq M+1$.
\end{enumerate}
Furthermore the parities of $d'(u,v)$ and $\bar d(u,v)$ are equal.
\end{lem}

\begin{proof}
The first part of the Lemma can be proven analogously to Lemma \ref{lem:bestcompletion}:

Suppose that the statement is not true and take any witness $\str{G}'=(G, d')$ (i.e. a completion of ${\str{G}}$ into $\mathcal A^\delta_{\infty,0,C_0,2\delta+1}$ such that there is a pair of vertices violating the statement). Recall that the completion with magic parameter $M$ is defined as a limit of a sequence $\str{G}_1, \str{G}_2, \ldots$ of edge-labelled graphs such that $\str{G}_1={\str{G}}$ and each two subsequent graphs differ at most by adding edges of a single distance.

Take the smallest $i$ such that in the graph $\str{G}_i = (G,d_i)$ there are vertices $u,v\in G$ with $d_i(u,v) > M+1$ and $d_i(u,v) > d'(u,v)$ or $d_i(u,v) < M$ and $d_i(u,v) < d'(u,v)$. Let $w\in G$ be the witness of $d_i(u,v)$.

Now again we shall distinguish three cases, based on whether $d_{i}(u,v)$ was introduced by $\mathcal F^-$, $\mathcal F^+$ or $\mathcal F^{C_0}$.
The cases $\mathcal F^-$ and $\mathcal F^+$ follow exactly the same way as in the proof of Lemma~\ref{lem:bestcompletion}. We verify the remaining case $\mathcal F^{C_0}$.

Recall that, by the admissibility of $C_0$, we have $C_0-1\geq 2\delta+3$ and $M+1\leq \left\lfloor \frac{C_0-\delta-1}{2} \right\rfloor$.
 Thus we get $d_{i-1}(u,w),d_{i-1}(v,w)>M+1$ (otherwise, if, say, $d_{i-1}(u,w)\leq M+1$, we obtain contradiction $C_0-\delta-1\geq 2M + 2 > d_{i-1}(u,w)+d_i(u,v) = C_0-2-d_{i-1}(v,w) \geq C_0-\delta-2$ where $2M+2\neq C_0-\delta-2$ are both even but within an interval of size 1). So again by the Bipartite Optimality and the Parity Lemma \ref{lem:bibestcompletion} $d'(u,w)\geq d_{i-1}(u,w)$ and $d'(v,w)\geq d_{i-1}(v,w)$, which means that the triangle $u,v,w$ in $\str{G}'$ is forbidden by the $C_0$ bound, which is absurd as $\str{G}'$ is a completion of $\str{G}$ in $\mathcal A^\delta_{\infty,0,C_0,2\delta+1}$.

\medskip

For the second part (about parities), observe first that the parity of an edge $d'(u,v)$ in $\str{G'}$ has to be equal to $d'(u,w)+d'(w,v)$ for every other vertex $w$, since there are no triangles of odd perimeter in $\mathcal A^\delta_{\infty,0,C_0,2\delta+1}$. Since $\str{G}$ is connected, the parity of $d'(u,v)$ is equal to the parity of the path distance of $(u,v)$ in $\str{G}$.

By the definition of the bipartite completion algorithm as a limit of graphs $\str{G}_1, \str{G}_2, \ldots$, if the statement is not true, then there has to be a smallest $i$ such that in the graph $\str{G}_i = (G,d_i)$ there are vertices $u,v\in G$ where the parity of $d_i(u,v)$ differs from the parity of $d'(u,v)$. Let $w$ be a witness for the edge $(u,v)$. Then three cases can appear (corresponding to $\mathcal F^-$, $\mathcal F^+$, $\mathcal F^{C_0}$).

\begin{description}
\item[$\mathcal F^-$:]
In the first case we have $d_i(u,v) = |d_{i-1}(u,w)-d_{i-1}(v,w)|$, which has the same parity as $d_{i-1}(u,w) + d_{i-1}(v,w)$. By minimality of $i$ this value has the same parity as $d'(u,w)+d'(v,w)$, hence this case cannot appear. 

\item[$\mathcal F^+$:]
In the second case we have $d_i(u,v) = d_{i-1}(u,w)+d_{i-1}(v,w)$. Analogously to the first case then $d_{i}(u,v)$ has to have the same parity as $d'(u,w)+d'(v,w)$, which is a contradiction.

\item[$\mathcal F^{C_0}$:]
In the third case $d_i(u,v) = C_0-2-d_{i-1}(u,w)-d_{i-1}(v,w)$. Since $C_0 -2$ is even, this distance has again the same parity $d_{i-1}(u,w)+d_{i-1}(v,w)$ and hence $d'(u,w)+d'(v,w)$, which is a contradiction.
\end{description}

In the last step, the distances $M$ and $M+1$ are added according to the parity of the path distance, hence also in this step the parity of edges is preserved.
\end{proof}

Note that Lemma \ref{lem:biexpandtime} implies that any magic completion of a graph $\str{G}$ contains a triangle with odd perimeter if and only if every completion of $\str{G}$ contains such a triangle. In order to verify the correctness of our completion algorithm it is therefore only left to verify the analogous statement for the two other types of forbidden triangles: non-metric triangles, and even triangles that are forbidden due to the $C_0$-bound. But for those triangles the result can be shown just by following the corresponding proofs in the primitive case.

\begin{figure}[t]
\centering
\includegraphics{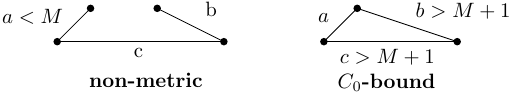}
\caption{Types of forbidden triangles in bipartite spaces.}
\label{fig:BFtriangles}
\end{figure}
\begin{lem}[Bipartite Metric Lemma]\label{lem:biometric}
Let ${\str{G}}=(G,{d})\in \mathcal G^\delta$ be a connected edge-labelled graph such that there is a completion of ${\str{G}}$ into $\mathcal  A^\delta_{\infty,0,C_0,2\delta+1}$; let $\overbar{\str{G}}=(G,\bar d)$ be its bipartite completion with magic parameter $M$. Then there are no non-metric triangles in $\overbar{\str{G}}$.
\end{lem}

\begin{proof}
We proceed in analogy to the proof of Lemma~\ref{lem:metric}.

Suppose for a contradiction that there is a triangle with vertices $u,v,w \in \overbar{\str{G}}$ such that $\bar d(u,v)+\bar d(v,w)< \bar d(u,w)$. Denote $a=\bar d(u,v)$, $b=\bar d(v,w)$ and $c=\bar d(u,w)$ and assume without loss of generality that $a\leq b < c$. By Lemma~\ref{lem:bibestcompletion} we know that $a+b+c$ is even. Let $a',b',c'$ be the corresponding edge lengths in an arbitrary completion of ${\str{G}}$ into $\mathcal  A^\delta_{\infty,0,C_0,2\delta+1}$. We shall distinguish three cases:

\begin{enumerate}

\item First suppose $a,b,c < M$. Then $t_M(a)\leq t_M(b) < t_M(a+b) < t_M(c)$, which means that $c$ must be already in $\str{G}$. By Lemma~\ref{lem:bibestcompletion}  $a' + b' \leq a + b < c = c'$, which is a contradiction.

\item Another possibility is $a<M$ and $b,c\geq M$ (actually $c> M+1$, since $abc$ is non-metric). 
By Lemma~\ref{lem:bibestcompletion} we know that $a'\leq a$ and $c'\geq c$.  If $b$ was already in $\str{G}$, then $\str{G}$ has no completion -- a contradiction. Otherwise clearly $c-a > b \geq M$, so $t_M(c-a) < t_M(b)$ (define $t_M(M) = t_M(M+1)=\infty$). But as $c-a$ depends on $c$ and $a$, we get $t_M(c-a) > t_M(c),t_M(a)$, which means that the bipartite completion algorithm with magic parameter $M$ would complete the edge $v,w$ with the length $c-a$ and not with $b$.

\item The last possibility is $a,b<M$ and $c\geq M$. Then (by Lemma~\ref{lem:bibestcompletion} and Lemma~\ref{lem:bimisgood} if $M\leq c\leq M+1$) we have $a'\leq a$, $b'\leq b$ and $c'\geq c$, hence the triangle $a',b',c'$ is again non-metric.
\end{enumerate}
\end{proof}

\begin{lem}[Bipartite $C_0$-bound Lemma]\label{lem:biCbound}
Let $\overbar{\str{G}}, \str{G}$ be as in Lemma \ref{lem:biometric}.
Then there are no triangles forbidden by the $C_0$-bound in $\overbar{\str{G}}$.
\end{lem}

\begin{proof}
Again we proceed analogously to Lemma \ref{lem:Cbound}.
Suppose for contradiction that there is a triangle with vertices $u,v,w$ in $\overbar{\str{G}}$ such that $\bar d(u,v)+\bar d(v,w)+\bar d(u,w)\geq C_0$. For brevity let $a=\bar d(u,v)$, $b=\bar d(v,w)$ and $c=\bar d(u,w)$. Assume without loss of generality $a\leq b\leq c$. Let $a',b',c'$ be the corresponding edge lengths in an arbitrary completion of $\bar{\str{G}}$ into $\mathcal  A^\delta_{\infty,0,C_0,2\delta+1}$. Then two cases can appear. 

Either $a,b,c > M+1$, and then by Lemma \ref{lem:bibestcompletion} we have $a'\geq a$, $b'\geq b$ and $c'\geq c$, so we get the contradiction $a' + b' + c' \geq C_0$; or $a\leq M+1$, $c\geq b>M+1$ and $a+b+c\geq C_0$. In this case Lemmas \ref{lem:bimisgood} and \ref{lem:bibestcompletion} imply $b'\geq b$ and $c'\geq c$ and $a'\leq a$. If the edge $(u,v)$ was already in $\str{G}$, then clearly $a' + b' + c' \geq a+b+c\geq C_0$, which is a contradiction. If $(u,v)$ was not already an edge in $\str{G}$, then it was added by the bipartite completion algorithm with magic parameter $M$ in step $t_M(a)$. Let $\bar{a}=C_0-2-b-c$. Then clearly $\bar{a}<a$, which means that $t_M(\bar{a}) < t_M(a)$, and as $\bar{a}$ depends on $b,c$, we have $t_M(b),t_M(c)<t_M(\bar{a})$. But then the bipartite completion with magic parameter $M$ actually sets the length of the edge $u,v$ to be $\bar{a}$ in step $t_M(\bar{a})$, which is a contradiction.
\end{proof}

\begin{lem}[Bipartite automorphism Preservation Lemma]
\label{lem:biaut}
Let $\str{G}\in \mathcal G^\delta$ be connected and $\overbar{\str{G}}$ be its bipartite completion with magic parameter $M$. Then every automorphism of $\str{G}$ is also an automorphism of $\overbar{\str{G}}$.
\end{lem}
\begin{proof}
Cf. proof of Lemma \ref{lem:aut}. Observe that since $\str{G}$ is assumed to be connected, the final step of the algorithm that includes edges of length $M$ and $M+1$ is canonical. Hence automorphisms are preserved.
\end{proof}

\begin{thm} \label{thm:bimagiccompletion}
Let $3\leq\delta < \infty$, $C_0>2\delta+3$ and $\mathcal S$ be an admissible set of Henson constraints for $\mathcal A^\delta_{\infty,0,C_0,2\delta+1}$.
Let $\str{G}=(G,d)$ be a connected edge-labelled graph such that there is a completion of $\str{G}$ into $\mathcal A^\delta_{\infty,0,C_0,2\delta+1}\cap \mathcal A^\delta_\mathcal S$ and let $\overbar{\str{G}}=(G,\bar d)$ be its bipartite completion with magic parameter $M$. Then $\overbar{\str{G}}\in\mathcal A^\delta_{\infty,0,C_0,2\delta+1}\cap \mathcal A^\delta_\mathcal S$.

$\overbar{\str{G}}$ is optimal in the following sense:
Let $\str{G}'=(G,d')\in\mathcal  A^\delta_{\infty,0,C_0,2\delta+1}$ be an arbitrary completion of $\str{G}$ in  $\mathcal  A^\delta_{\infty,0,C_0,2\delta+1,\mathcal S}$, then 
for every pair of vertices $u,v\in G$ one of the following holds:
\begin{enumerate}
 \item $d'(u,v) \geq \bar{d}(u,v) \geq M+1$,
 \item $d'(u,v) \leq \bar{d}(u,v) \leq M$,
 \item $M\leq \bar{d}(u,v)\leq M+1$.
\end{enumerate}
Furthermore the parity of every distance in $\str{G}'$ is
the same as the parity of the corresponding distance in $\overbar{\str{G}}$
and every automorphism of $\str{G}$ is also an automorphism of
$\overbar{\str{G}}$.
\end{thm}
\begin{proof}
For $\mathcal S=\emptyset$ the statement follows from Lemmas \ref{lem:biometric}, \ref{lem:biCbound}, \ref{lem:bibestcompletion} and \ref{lem:biaut}.

Observe that the only non-empty admissible set $\mathcal S$ consists of a single anti-clique (that is a metric space with all distances $\delta$) and in this case $\delta$ is even.
Then we can use the fact that $\delta\geq 4$ and thus $M$ can be always chosen to be at most $\delta-2$. In this case the bipartite completion with magic parameter $M$ will never introduce an edge of distance $\delta$.
\end{proof}

\begin{remark} \label{rem:biunconnected}
Note that if we drop the condition of $\str{G}$ being connected in Theorem \ref{thm:bimagiccompletion}, we can still compute a completion $\overbar{\str{G}}$ of $\str{G}$ by first completing all connected components according to Theorem \ref{thm:bimagiccompletion} and then adding edges $M$ and $M+1$ as described in the proof of Lemma \ref{lem:biconnected}. This completion $\overbar{\str{G}}$ still satisfies the optimality conditions 1,2 and 3; however we will lose other important features: 
\begin{enumerate}
\item The constructed completion is not uniquely determined by $\str{G}$, but also depends on how we connect the different components of $\str{G}$. (The proof of Lemma~\ref{lem:biconnected} used a non canonical choice of $x\in A$ and $y \in B$, two connected components.)
\item The automorphism group of $\overbar{\str{G}}$ can be a proper subgroup of $\Aut(\str{G})$,
\item Edges in $\str{G'}$ and $\overbar{\str{G}}$ may have different parities.
\end{enumerate}
The above observations have an impact on the results of the following section: Point 1 corresponds to the fact that we only have local, but not global stationary independence relation (see Corollary \ref{cor:biSIR}).
 \end{remark}

\subsection{Local stationary independence relation}
\begin{corollary} 
\label{cor:biSIR}
Let $\delta\geq 3$, $C_0>2\delta+3$ and $\mathcal S$ be an admissible set of Henson constraints for $\mathcal A^\delta_{\infty,0,C_0,2\delta+1}$. Then there is no stationary independence relation on $\Gamma^\delta_{\infty,0,C_0,2\delta+1,\mathcal S}$. However, for every magic parameter $M$ there is a local stationary independence relation on $\Gamma^\delta_{\infty,0,C_0,2\delta+1,\mathcal S}$ as follows:
 $\str{A}\ind_{\str{C}}\str{B}$ if and only if $\struc{\str{A}\str{B}\str{C}}$ is isomorphic to the completion of the completion with magic parameter $M$ of the free amalgamation of $\struc{\str{A}\str{C}}$ and $\struc{\str{B}\str{C}}$ over $\str{C}$.
\end{corollary}
\begin{proof}

First, suppose for a contradiction that there is a stationary independence relation $\ind$ in $\Gamma^\delta_{\infty,0,C_0,2\delta+1}$. By Theorem \ref{thm:canonicalamalg} this is equivalent to the existence of a canonical symmetric amalgamation operator $\oplus$ on $\mathcal{A}^\delta_{\infty,0,C_0,2\delta+1}$. Let $\str{A}, \str{B} \in \mathcal{A}^\delta_{\infty,0,C_0,2\delta+1}$ such that $\str{A}$ consist of two vertices $u,v$ with $d(u,v)=1$ and let $\str{B}$ contains only one vertex $w$. In the canonical amalgam over the empty set $\str{A} \oplus_{\empty} \str{B}$, monotonicity implies that both $(\{u,w\},d)$ and $(\{u,v\},d)$ are canonical amalgams of two points over the empty set. By the uniqueness of $\oplus$ we have that $d(u,w) = d(v,w)$. But then the triangle $(u,v,w)$ has odd perimeter, which contradicts that $\str{A} \oplus_{\empty} \str{B} \in \mathcal{A}^\delta_{\infty,0,C_0,2\delta+1}$.

The local stationary independence relation follows by same argument as in proof of Corollary~\ref{cor:SIR}.
By locality of the stationary independence relation we note that all graphs that need to be completed in the proof are already connected, and thus Theorem~\ref{thm:bimagiccompletion} applies.
\end{proof}

\subsection{Ramsey property and EPPA}
We follow the general direction of Section~\ref{sec:ramseyEPPA1}. The extra
difficulty is that the bipartiteness cannot be expressed by means of a finite set of
obstacles, because such a set must contain odd cycles of unbounded length. However we can
obtain the following variant of Lemma~\ref{lem:obstacles}.

\begin{lem}[Bipartite Finite Obstacles Lemma]
\label{lem:biobstacles}
Let $\delta\geq 3$, $C_0>2\delta+3$. 
Then there is finite set $\mathcal O$ of edge-labelled cycles such that every edge-labelled graph $\str{G}\in \Forb(\mathcal O)$ without odd cycles is in $\mathcal A^\delta_{\infty,0,C_0,2\delta+1}$
\end{lem}
\begin{proof}
Cf. proof of Lemma \ref{lem:obstacles}. To verify that the final step of algorithm will succeed we use the fact that there are no odd cycles in $\str{G}$.
\end{proof}
\begin{thm}
\label{thm:biramseyregular}
Let $3 \leq \delta < \infty$, $C_0>2\delta+3$ and $\mathcal S$ be an admissible set of Henson constraints for $\mathcal A^\delta_{\infty,0,C_0,2\delta+1}$. 
The class $\overrightarrow{\mathcal A}^\delta_{\infty,0,C_0,2\delta+1}\cap \overrightarrow{\mathcal A}_{\delta,\mathcal S}$ of convex orderings of $\mathcal  A^\delta_{\infty,0,C_0,2\delta+1}\cap \mathcal A^\delta_\mathcal S$ with an additional unary predicate determining the bipartition is Ramsey and has the expansion property.
\end{thm}
Here a convex ordering is any ordering such that vertices in first bipartition (denoted by the unary predicate) form an initial segment.
\begin{proof}
Let $\mathcal O$ be given by Lemma~\ref{lem:biobstacles}.  Denote by $\mathcal O'$ the family of all possible expansions of $\mathcal O$ by a unary predicate determining the bipartition
with additional structure on two vertices which forbids vertices from the same bipartition from being connected by an edge of odd length and vertices from different bipartitions from being connected by an edge of even length.  Observe that structures in $\Forb(\mathcal O')$ have no odd cycles and thus Theorem~\ref{thm:localfini} and Lemma~\ref{lem:biobstacles} apply.

If $\mathcal S$ is non-empty, then observe that $M$ can be chosen to be at most $\delta-2$ and one can use the same argument as in proof of Theorem~\ref{thm:ramseyHenson}.

The expansion property again follows by the standard argument.
\end{proof}

\begin{thm}
\label{thm:biregulareppa}
Let $\delta\geq 3$, $C_0>2\delta+3$ and $\mathcal S$ be an admissible set of Henson constraints for $\mathcal A^\delta_{\infty,0,C_0,2\delta+1}$. 
Then the class $\mathcal A^\delta_{\infty,0,C_0,2\delta+1}\cap \mathcal A^\delta_\mathcal S$ has coherent EPPA.
\end{thm}
\begin{proof}
Let $\mathcal O$ be given by Lemma~\ref{lem:biobstacles}.  By an application of
Theorem~\ref{thm:herwiglascar} obtain $\str{B}\in \Forb(\mathcal O)$ which is
a coherent EPPA-witness of $\str{A}$. Without loss of generality we can assume that $\str{B}$
is connected (otherwise the connected component of $\str{B}$ containing $\str{A}$ is a coherent EPPA-witness, too).  If $\str{B}$ has no odd cycles, apply Theorem~\ref{thm:bimagiccompletion}
to obtain the desired EPPA-witness of $\str{A}$.

If $\str{B}$ contains odd cycles, construct $\str{C}=(C,d')$ as follows.
 The vertex set $C$ is $B\times\{0,1\}$, and 
$$d'((u,i),(v,j)) =
\begin{cases}
 0 & \hbox{if } (u,i)=(v,j)\\
 d(u,v) & \hbox{if $d(u,v)$ is odd and $i\neq j$}\\
 d(u,v) & \hbox{if $d(u,v)$ is even and $i=j$}.\\
\end{cases}
$$
Observe that $\str{C}$ is connected and contains no odd cycles. Let $p\colon A\to \{0,1\}$ be a function
determining the bipartition of $\str{A}$. We verify that $\str{C}$ is an
extension of $\phi(\str{A})$ for the following embedding $\phi(v)=(v,p(v))$.
Denote by $\pi(v,j)=v$ the projection, which is a homomorphism of $\str{C}\to \str{B}$
and also embedding $\phi(\str{A})\to \str{A}$.

Every partial isometry $\psi$ of $\phi(\str{A})$ induces a partial isometry $\psi\circ\pi$ of $\str{A}$
which extends to automorphism $\bar\psi$ of $\str{B}$. This automorphism induces an automorphism $\psi'$ 
of $\str{C}$ by mapping $(u,i)\mapsto (\bar\psi(u),i).$
It however might not be an extension of $\psi$ because $\psi'$ may change the values of the function $p$.
In this case combine it with an automorphism mapping $(v,0) \mapsto (v,1)$ and $(v,1)\mapsto (v,0)$. It is easy to see
that this construction preserves coherence of the extensions in $\str{B}$.

If $\mathcal S$ is non-empty, again observe that $M$ can be chosen to be at most $\delta-2$ and use same argument as in proof of Theorem~\ref{thm:EPPAHenson}.
\end{proof}

\section{Antipodal spaces}
\label{sec:antipodal}
In this section we discuss the antipodal classes in Cherlin's catalogue. We say that an amalgamation class $\K$ of metric spaces of diameter $\delta$ is {\em antipodal} if the edges of distance $\delta$ form a matching in the \Fraisse{} limit of $\K$. In particular then there are no triangles with more than one edge of length $\delta$ in $\K$. Note that this also implies that $\K$ has no strong amalgamation. For admissible parameters $(\delta,K_1,K_2,C_0,C_1, \mathcal S)$ the class $\mathcal A^{\delta}_{K_1,K_2,C_0,C_1, \mathcal S}$ is antipodal if and only if $C = 2\delta +1$. More precisely only the two cases can appear:
\begin{defn}
The admissible parameters $3\leq \delta<\infty$, $K_1$, $K_2$, $C_0$ and $C_1$ are {\em antipodal} when
\begin{enumerate}[label=(\Roman*)]
\setlength\itemsep{0em}
\item[\ref{I}] $K_1=\infty$, $C_0=2\delta+2$ (the bipartite case; so $K_2=0$ and $C_1=2\delta+1$), or
\item[\ref{IIa}] $1\leq K_1\leq \frac{\delta}{2}, K_2 = \delta - K_1, C_0=2\delta+2, C_1 = 2\delta+1$.
\end{enumerate}
\end{defn}
The parameters in this case are pushed to the extreme situation where, either there is no magic parameter, or $M=\lfloor\frac{\delta}{2}\rfloor$ is the only parameter satisfying Definition~\ref{defn:magiccompletion}, respectively Definition~\ref{defn:bimagiccompletion}.
\begin{defn}
\label{defn:antipodal}
Let $\delta \geq 3$. For $\str{A}=(A,d)\in \mathcal G^{\delta-1}$, an {\em antipodal companion} of $\str{A}$ is any $\str{A}^*=(A,d^*)\in \mathcal G^{\delta-1}$ such that there exists $B\subseteq A$ and:
$$d^*(u,v) =
\begin{cases} 
      d(u,v)& \text{if $d(u,v)$ is defined and $u,v\in B$ or $u,v\notin B$} \\
      \delta-d(u,v) & \text{if $d(u,v)$ is defined and $u\in B,v\notin B$ or vice versa}.
\end{cases}
$$
\end{defn}

Observe that for every admissible antipodal parameters, the class $\mathcal A^\delta_{K_1,K_2,C_0,C_1, \mathcal S}$, and also the subclass of metric spaces of diameter at most $\delta-1$ form amalgamation classes.  This subclass is equal to $\mathcal A^{\delta-1}_{K_1,K_2,C_0,C_1, \mathcal S}$ and corresponds to either a primitive case (Case~\ref{IIa} for $K_1>2$ or \ref{III} for $1\leq K_1\leq 2$) or bipartite case (Case~\ref{I}) of Cherlin's catalogue. The following describes a reverse way to produce an antipodal space of diameter $\delta$ from a space of diameter $\delta-1$.

\begin{figure}[t]
\centering
\includegraphics{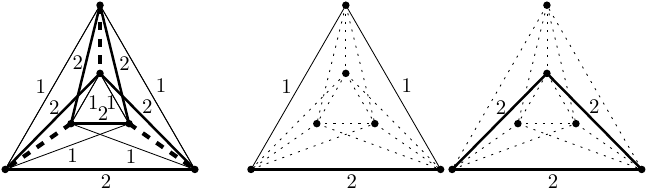}
\caption{Antipodal extension of the triangle 122 for $\delta=3$, the matching formed by edges of distance 3 is denoted by dashed lines. The second and third picture highlight the original triangle and one of its antipodal companions.}
\label{fig:antipodal122}
\end{figure}

\begin{defn}[Antipodal extensions]
\label{defn:antipodalext}
Given an edge-labelled graph $\str{M}=(M,d) \in \mathcal G^{\delta-1}$ its \emph{antipodal extension} is the edge-labelled graph $\str{M}^\ominus=(M\times \{0,1\},d') \in \mathcal G^{\delta-1}$ such
that $d'((u,i),(v,i))=d(u,v)$ and $d'((u,i),(v,1-i))=\delta-d(u,v)$ for $u,v\in M$ and $i\in \{0,1\}$.

For an ordered edge-labelled graph $\overrightarrow{\str{M}}=(M,d,\leq_\str{M})$, its \emph{ordered antipodal extension} $\overrightarrow{\str{M}}^\ominus=(M\times \{0,1\},d',\leq_{\str{M}^\ominus})$ 
where $(M,d)^\ominus=(M\times \{0,1\},d')$ and $(u,i)\leq_{\str{M}^\ominus} (v,j)$ if and only if $i<j$ or $i=j$ and $u\leq_\str{M} v$.

For an ordered bipartite edge-labelled graph $\overrightarrow{\str{M}}=(M,d,\leq_\str{M},B)$ where $B$ is unary predicate determining the bipartition we define $\overrightarrow{\str{M}}^\ominus=(M\times \{0,1\},d',\leq_{\str{M}^\ominus},B^\ominus)$ analogously and we put $(v,0)\in B^\ominus \iff (v)\in B$ and if $\delta$ is even then $(v,1)\in B^\ominus \iff (v)\in B$; if $\delta$ is odd then $(v,1)\in B^\ominus \iff (v)\notin B$.
\end{defn}

\subsection{Ramsey property}

The Ramsey property follows from the above correspondence in full generality. Recall that we use arrows to indicate that the structures in a class are ordered (a requirement for any Ramsey class).
\begin{thm}
\label{thm:antiporamsey}
Let $\overrightarrow{\K}$ be a Ramsey class of expansions of metric spaces. Then the class $\overrightarrow{\K}^\ominus$ of all antipodal expansions of $\overrightarrow{\K}$ is Ramsey.
\end{thm}
Here $\overrightarrow{\K}$ can be either a Ramsey class of ordered metric spaces (given by Theorems~\ref{thm:regularramsey} and~\ref{thm:ramseyHenson}) or ordered metric spaces with predicate denoting bipartition (given by Theorem~\ref{thm:biramseyregular}).  
\begin{proof}
Given $\str{A}^\ominus,\str{B}^\ominus\in \overrightarrow{\K}^\ominus$, apply the Ramsey property in $\overrightarrow{\K}$ to obtain $\str{C}\longrightarrow (\str{B})^\str{A}_2$. It is easy to check that $\str{C}^\ominus\longrightarrow(\str{B}^\ominus)^{\str{A}^\ominus}_2$.
\end{proof}
\begin{remark}\label{rem:antiporamsey}
Observe that $\overrightarrow{\K}^\ominus$ is not a hereditary class. It consists only of metric spaces where for every vertex there is precisely one vertex in a distance $\delta$.
Thus one can color only structures $\str{A}$ having this property.
The Ramsey property for antipodal expansions can be, equivalently, stated for the hereditary class of all subspaces of spaces in $\overrightarrow{\K}^\ominus$ with a unary predicate determining the podality.

This unary predicate becomes necessary because in $\overrightarrow{\K}^\ominus$ there is a definable equivalence relation on vertices which is
not in $\K^\ominus$, namely $u\sim v$ whenever both $\{w:w\leq u\}$ and $\{w:w\leq v\}$ span an edge of length $\delta$ or none does.  As shown in~\cite{Hubicka2016} an equivalence relation on vertices implies the necessity for unary predicates in the Ramsey lift. It is interesting to observe that in showing EPPA (see Section \ref{sec:antieppa}) there is sometimes no need for further expansion of the language. To maintain that, it is necessary to assume that $\K$ is closed under the operation of forming antipodal companions (Definition~\ref{defn:antipodal}).
\end{remark}

\subsection{Generalised completion algorithms for antipodal classes}

To show the extension property for partial automorphisms we need a way to complete the spaces symmetrically. 

Our completion algorithm for the antipodal classes $\K = \mathcal A^\delta_{K_1,K_2,C_0,C_1,\mathcal S}$ will be based on the completion algorithm for $\K^{\delta-1} = \mathcal A^{\delta-1}_{K_1,K_2,C_0,C_1, \mathcal S}$ with magic parameter $M = \lfloor \frac{\delta}{2} \rfloor$ for non-bipartite spaces and $\K^{\delta-1} = \mathcal A^{\delta-1}_{\infty,0,C_0,2\delta+2,2\delta-1, \mathcal S}$ for bipartite spaces. More precisely, we are not considering the completion to $\mathcal A^\delta_{K_1,K_2,C_0,C_1, \mathcal S}$, but to its expansion $\mathcal B^\delta_{K_1,K_2,C_0,C_1, \mathcal S}$ by an additional unary predicate $P$ determining the podality (so every edge of length $\delta$ has precisely one vertex $v\in P$). Roughly speaking our completion algorithm for an input structure $(G,d,P)$ will first complete the pode $(P,d)$ in $\mathcal A^{\delta-1}_{K_1,K_2,C_0,C_1, \mathcal S}$ and then forms its antipodal extension. 

\medskip

We are going to consider four separate cases:
\begin{enumerate}
  \item Antipodal classes in Case~\ref{IIa} with even $\delta$ (Corollary \ref{cor:antievencompletion}),
  \begin{example}
    An example of such a class is $\mathcal A^4_{1,3,10,9}$. The class of all metric spaces in $\mathcal A^4_{1,3,10,9}$ of diameter 3 (the underlying class of metric space appearing on each pode) is
    $\mathcal A^3_{1,3,10,9}$ which is the class of all finite metric spaces of diameter 3 omitting triangle $333$ (the primitive case covered by Section~\ref{sec:basic3}).
    All metric spaces in $\mathcal A^4_{1,3,10,9}$ are thus subspaces of antipodal extensions of spaces in $\mathcal A^3_{1,3,10,9}$.

    Another choice of such a class is  $\mathcal A^4_{2,2,10,9}$ with underlying class of metric spaces $\mathcal A^3_{2,2,10,9}$ which forbids
    triangles $333$ (by the $C$-bound), $111$ (by the $K_1$-bound) and $133$ (by the $K_2$-bound). It is useful to observe that from the conditions $C=2\delta+1$ and $K_1=\delta-K_2$
    it follows that the underlying metric spaces are always closed for antipodal companions:  the companion of the triangle $333$
    is $311$ which is non-metric and the companion of $111$ is $133$. This property is needed for the class of antipodal metric spaces to form an amalgamation class.
  \end{example}
  \item Antipodal bipartite classes in Case~\ref{I} with odd $\delta$ (Corollary \ref{cor:antibioddcompletion}),
  \begin{example}
    An example of such a class is $\mathcal A^3_{\infty,0,8,7}$.  This is 
    a special case, because the underlying metric space is of diameter 2 and thus not analyzed in this paper.
    The underlying class of metric spaces consists of the complete bipartite graphs where non-edge is represented by distance 2.
    All metric spaces in $\mathcal A^3_{\infty,0,8,7}$ are thus antipodal extensions of complete bipartite graphs.

    The first standard case is $\mathcal A^5_{\infty,0,12,11}$ with underlying class of metric spaces $\mathcal A^4_{\infty,0,12,9}$
    which is the class of all bipartite metric spaces omitting triangle $444$. This is a bipartite case covered in Section~\ref{sec:bipartite}.

    Observe that because $\delta$ is odd, all edges in distance $\delta$ cross the bipartition.
  \end{example}
  \item Antipodal classes in Case~\ref{IIa} with odd $\delta$ (Theorem \ref{thm:anticompletion}), and
  \begin{example}
    In this category lies the class $\mathcal A^3_{1,2,8,7}$ which is again a special case, because the underlying
    class has diameter 2.  It is the class of all metric spaces of diameter 2 (that is, a representation of 
    graphs where non-edge corresponds to distance 2). Structures in $\mathcal A^3_{1,2,8,7}$ can equivalently be seen
    as double-covers of complete graphs which are further discussed in Section~\ref{sec:conclussion}.

    An example of a standard case is the class $\mathcal A^5_{1,4,12,11}$ with underlying space $\mathcal A^4_{1,4,12,11}$
    which is the class of all metric spaces of diameter 4 with triangle $444$ forbidden. Again this is a primitive
    case covered by Section~\ref{sec:basic3}.

    The main difference with the first case is that the graph consisting of two
    disjoint edges of length $\delta$ (and no other edges or vertices) requires two different
    distances to be used in its completion, while in the first case all remaining edges
    can be completed by $\frac{\delta}{2}$.
  \end{example}
  \item Antipodal bipartite classes in Case~\ref{I} with even $\delta$ (Theorem \ref{thm:antibicompletion}).
  \begin{example}
    An example of such a class is $\mathcal A^4_{\infty,0,10,9}$ with underlying class $\mathcal A^3_{\infty,0,10,7}$.
    This is the class of all bipartite metric spaces of diameter 3 covered in Section~\ref{sec:bipartite}.

    The main difference with the second case is that the edges of distance $\delta$ connect pairs of vertices
    in the same bipartition and thus one can define equivalence both by use of the bipartition and by the pairing of $\delta$ edges.
  \end{example}
\end{enumerate}
In the first two cases we will show the perhaps surprising fact that the antipodal completion of $(G,d,P)$ described as above does not depend on the choice of $P$, hence we obtain a unique completion $(G, \bar d)$ to the original class of metric spaces $\mathcal A^\delta_{K_1,K_2,C_0,C_1, \mathcal S}$.
In the latter two cases this however does not hold and similarly to the bipartite case (cf. Remark \ref{rem:biunconnected}) we are confronted with different completions, depending on the choice of $P$. This ambiguity cannot be eliminated and is reflected in the fact that there is no canonical symmetric amalgamation on $\mathcal A^\delta_{K_1,K_2,C_0,C_1, \mathcal S}$ for those cases (Theorem \ref{thm:antiSIR}).

In order to simplify our proof we first show that it is enough to only consider edge-labelled graphs $\str{G}$ that are symmetric according to the following definition:

\begin{figure}[t]
\centering
\includegraphics{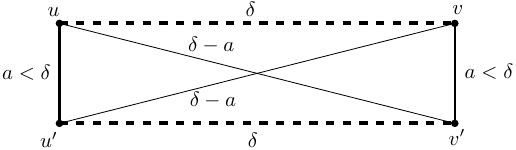}
\caption{Antipodal quadruple.}
\label{fig:antipodalquadruple}
\end{figure}
\begin{defn}[Antipodal quadruple]
A quadruple of distinct vertices $(u,v,u',v')$ in $\str{G}\in\mathcal G^\delta$ is {\em antipodal} if $d(u,v)=d(u',v')=\delta$ and $d(u,u')=d(v,v')$, $d(u,v')=d(v,u')=\delta-d(u,u')$. (In particular all those distances are defined.) See Figure~\ref{fig:antipodalquadruple}.

Let us call $\str{G}\in\mathcal G^\delta$ \emph{antipodally symmetric}, if for every vertex $x$ there is a unique $x^*$ with $d(x,x^*) = \delta$ and if every edge of length $< \delta$ is part of an antipodal quadruple.
\end{defn}

\begin{lem} \label{lem:antisymmetric}
Let $\str{G} = (G,d) \in\mathcal G^\delta$ and suppose that $\str{G}$ has a completion to $\mathcal A^\delta_{K_1,K_2,C_0,C_1, \mathcal S}$. Then there is an edge-labelled graph $(G^*,d^*)$ with $G^* \supseteq G$ and $d^* \supseteq d$ that is antipodally symmetric and every completion $(G,d') \in \mathcal A^\delta_{K_1,K_2,C_0,C_1, \mathcal S}$ of $\str{G}$ has a unique extension to a completion of $(G^*,d^*)$. Furthermore every automorphism of $(G,d)$ extends uniquely to an automorphism of $(G^*,d^*)$.
\end{lem}

\begin{proof}
First observe that if $\str{G}$ has a completion to $\mathcal A^\delta_{K_1,K_2,C_0,C_1, \mathcal S}$ then there are no $(\delta,\delta)$-forks in $\str{G}$. Moreover quadruples of vertices $(u,v,u',v')$ with $d(u,v)=d(u',v')=\delta$ can only be completed to antipodal quadruple. Instead of completing $(G,d)$ we can therefore consider the antipodally symmetric graph $(G^*,d^*)$ that is constructed by adding a unique vertex $x^*$ for every $x$ and with $d^*(x,x^*) = \delta$ (if $x^*$ is not already present in $G$) and completing all quadruples of vertices $(u,v,u',v')$ such that $d^*(u,v)=d^*(u',v')=\delta$ and $d(u,u')$ is defined to their unique completion as antipodal quadruple. It is not hard to see that a completion $(G,d') \in \mathcal A^\delta_{K_1,K_2,C_0,C_1, \mathcal S}$ of $(G,d)$ gives rise to a unique completion of $(G^*,d^*)$ into $\mathcal A^\delta_{K_1,K_2,C_0,C_1, \mathcal S}$ by the setting $d^*(x^*,y) = \delta - d'(x,y)$ and $d^*(x^*,y^*) = d'(x,y)$, for $d(x,x^*) = d(y,y^*) = \delta$. Also every automorphism $f \in \Aut(G,d)$ extends to a unique automorphism $f \in \Aut(G^*,d^*)$ defined by $f(x^*) = f(x)^*$.
\end{proof}

Analogously, a structure $(G,d,P)$ with $(G,d) \in\mathcal G^\delta$ and a unary predicate $P$ has a completion to $(G,d',P) \in \mathcal B^\delta_{K_1,K_2,C_0,C_1, \mathcal S}$ if and only if its extension to an antipodally symmetric space has a completion to $\mathcal B^\delta_{K_1,K_2,C_0,C_1, \mathcal S}$. Therefore, in the following we always assume $(G,d,P)$ to be antipodally symmetric.

Note that, for antipodally symmetric $(G,d,P)$, every completion of $(P,\left.d\right|_{P^2})$ to $\mathcal A^{\delta-1}_{K_1,K_2,C_0,C_1, \mathcal S}$ extends to a unique completion of $(G,d,P)$ to $\mathcal B^{\delta}_{K_1,K_2,C_0,C_1, \mathcal S}$, namely its antipodal expansion.

\begin{defn}[Non-bipartite antipodal completion algorithm]\label{defn:antiftmcompletion}
Let $\delta\geq 3$, $K_1 \leq \frac{\delta}{2}$, $M=\lfloor \frac{\delta}{2} \rfloor$, $C=2\delta+1$ and let $(G,d,P)$ be an antipodally symmetric edge-labelled graph with a predicate $P$ for a pode.

Then we define the \emph{antipodal completion $(G,\bar d,P)$ of $(G,d,P)$ with parameter $M$} as follows: restricted to the pode, $(P,\bar d)$ is the completion of $(P,d)$ with magic parameter $M$ and diameter $\delta-1$ (cf. Definition \ref{defn:magiccompletion}) and $(G,\bar d)$ is its antipodal expansion. Since $(G,d,P)$ is antipodally symmetric, $(G,\bar d,P)$ is in fact a completion of $(G,d,P)$, i.e. $d(x,y) = \bar d(x,y)$, whenever $d(x,y)$ is defined.
\end{defn}

Then the following holds:

\begin{thm} \label{thm:anticompletion}
Let $3\leq \delta<\infty$ and $K_1\leq \frac{\delta}{2}$ and $M=\lfloor \frac{\delta}{2} \rfloor$. Let $(G,d,P)$ be antipodally symmetric. Suppose that $(G,d,P)$ has a completion into $\mathcal B^\delta_{K_1,\delta-K_1,2\delta+2,2\delta+1,\mathcal S}$ and let $(G,\bar{d},P)$ be its antipodal completion with magic parameter $M$. Then $(G,\bar{d},P) \in\mathcal B^\delta_{K_1,\delta-K_1,2\delta+2,2\delta+1, \mathcal S}$ and it is optimal in the following sense:
Let $(G,d',P)\in\mathcal B^\delta_{K_1,\delta-K_1,2\delta+2,2\delta+1, \mathcal S}$ be an arbitrary completion of $(G,d,P)$, then, for all $u,v \in G$:
\begin{enumerate}
 \item $d'(u,v) \geq \bar{d}(u,v) \geq M$ or $d'(u,v) \leq \bar{d}(u,v) \leq M$ if both $u,v \in P$ or if both $u,v \in G \setminus P$
 \item $d'(u,v) \geq \bar{d}(u,v) \geq \delta - M$ or $d'(u,v) \leq \bar{d}(u,v) \leq \delta - M$ else.
\end{enumerate}
Furthermore $\Aut(G,d,P) = \Aut(G,\bar d,P)$.
\end{thm}

\begin{proof} 
An antipodally symmetric graph $(G,d,P)$ has a completion to $\mathcal B^\delta_{K_1,\delta-K_1,2\delta+2,2\delta+1, \mathcal S}$ if and only if its pode $(P,d)$ has a completion to $\mathcal A^{\delta-1}_{K_1,\delta-K_1,2\delta+2, 2\delta-1, \mathcal S}$. 

By Theorem \ref{thm:magiccompletion} the optimality statement $d'(u,v) \geq \bar{d}(u,v) \geq M$ or $d'(u,v) \leq \bar{d}(u,v) \leq M$ is true for all $u,v \in P$. In general, every pair of vertices $u,v$ in $G$ is part of an antipodal quadruple. This implies the optimality statement in its general form.

The fact that $\Aut(G,d,P) = \Aut(G,\bar d,P)$ follows directly from the automorphism preservation Lemma \ref{lem:aut} for the pode $(P,d)$ and the observation that every automorphism of $(G,d,P)$ is uniquely determined by its restriction to $P$.
\end{proof}

\begin{corollary} \label{cor:antievencompletion}
Let $3\leq \delta<\infty$ be even, $K_1\leq \frac{\delta}{2}$ and $M= \frac{\delta}{2}$. Let $(G,d)$ be antipodally symmetric.
Suppose that $\str{G} = (G,d)$ has a completion into $\mathcal A^\delta_{K_1,\delta-K_1,2\delta+2,2\delta+1,\mathcal S}$. Then there is a unique completion $(G,\bar d) \in \mathcal A^\delta_{K_1,\delta-K_1,2\delta+2,2\delta+1,\mathcal S}$ that is optimal in the following sense:
Let $\str{G}'=(G,d')\in\mathcal A^\delta_{K_1,\delta-K_1,2\delta+2,2\delta+1,\mathcal S}$ be an arbitrary completion of $\str{G}$, then, for all $u,v \in G$:
$$d'(u,v) \geq \bar{d}(u,v) \geq M \text{ or } d'(u,v) \leq \bar{d}(u,v) \leq M.$$
Furthermore $\Aut(G,d) = \Aut(G,\bar d)$.
\end{corollary}

\begin{proof}
We define $(G,\bar d)$ as the underlying metric space of the completion of $(G,d,P)$ for an arbitrary predicate $P$ for a pode, i.e. for some $P$ such that $|\{x,x^*\} \cap P | = 1$ for all $x,x^*$ with $d(x,x^*) = \delta$. By Theorem \ref{thm:anticompletion}, $(G,\bar d) \in \mathcal A^\delta_{K_1,\delta-K_1,2\delta+2,2\delta+1,\mathcal S}$, if and only if $(G,d)$ has a completion to $\mathcal A^\delta_{K_1,\delta-K_1,2\delta+2,2\delta+1,\mathcal S}$.

Since $\delta$ is even, we have that $M = \delta - M$. Hence, the optimality part of Theorem \ref{thm:anticompletion} states that actually for all $u,v \in G$ and every arbitrary completion $\str{G}'=(G,d')\in\mathcal A^\delta_{K_1,\delta-K_1,2\delta+2,2\delta+1,\mathcal S}$:
$$d'(u,v) \geq \bar{d}(u,v) \geq M \text{ or } d'(u,v) \leq \bar{d}(u,v) \leq M.$$
Hence $(G,\bar d)$ does not depend on the choice of the pode $P$. 

It remains to show that automorphisms of $(G,d)$ (that do not necessarily fix a pode $P$) are preserved by the completion. So let $f \in \Aut(G,d)$ and $P$ be some predicate for a pode. Then $(G,d,P)$ and $(G,d,f(P))$ are isomorphic under $f$. Hence their completions are isomorphic, so $f \in \Aut(G,\bar d)$.
\end{proof}

With the following example we would like to illustrate out that the assumption of $\delta$ to be even cannot be dropped in Corollary~\ref{cor:antievencompletion} and also the original completion algorithm fails for odd $\delta$. 

\begin{example}
Consider $\mathcal A^3_{1,2,8,7}$ which is an example of a class of antipodal metric spaces.  Completing a graph consisting of two edges of length $3$ and no other edges using Definition~\ref{defn:magiccompletion} will result in a non-antipodal metric space where every non-edge will be completed by $M$. Because triangles $322$ and $311$ are forbidden neither $M=1$ or $M=2$ will make the completion algorithm from Definition~\ref{defn:magiccompletion} give the correct answer. Using the completion in Definition~\ref{defn:antiftmcompletion} for some choice of $P$ we obtain a completion, which depends on the choice of the pode $P$ and has fewer automorphisms then the input graph. 
\end{example}

Next we are going to consider the bipartite cases. Again Lemma~\ref{lem:antisymmetric} allows us to only consider antipodally symmetric edge-labelled graphs. Then we define the bipartite antipodal completion of $(G,d,P)$ as follows:

\begin{defn}[Bipartite antipodal completion algorithm]\label{defn:antibipftmcompletion}
Given $\delta\geq 3$, $M=\lfloor \frac{\delta}{2} \rfloor$ and $C=2\delta+1$ and an antipodally symmetric, connected edge-labelled graph $(G,d,P)$ with a predicate $P$ for a pode.

Then we define the \emph{antipodal completion $(G,\bar d,P)$ of $(G,d,P)$ with parameter $M$} by setting $(P,\bar d)$ to be the bipartite completion of $(P,d)$ with magic parameter $M$ and diameter $\delta-1$ (according to Definition~\ref{defn:biftmcompletion}) and $(G,\bar d)$ to be its antipodal expansion. Since $(G,d,P)$ is antipodally symmetric, $(G,\bar d,P)$ is in fact a completion of $(G,d,P)$.
\end{defn}

\begin{thm} \label{thm:antibicompletion}
Let $3\leq \delta$ and $M=\lfloor \frac{\delta}{2} \rfloor$.
Let $(G,d,P)$ be such that $(G,d)$ is antipodally symmetric and connected and suppose that $(G,d,P)$ has a completion into $\mathcal B^\delta_{\infty,0,2\delta+2,2\delta+1,\mathcal S}$. Let $(G,\bar{d},P)$ be its bipartite completion with magic parameter $M$. Then $(G,\bar{d},P) \in\mathcal B^\delta_{\infty,0,2\delta+2,2\delta+1,\mathcal S}$.

The completion is optimal in the following sense:
Let $\str{G}'=(G,d',P)\in\mathcal B^\delta_{\infty,0,2\delta+2,2\delta+1,\mathcal S}$ be an arbitrary completion of $(G,d,P)$ in $\mathcal B^\delta_{\infty,0,2\delta+2,2\delta+1,\mathcal S}$, then 
for every pair of vertices $u,v\in G$, $d'(u,v)$ has the same parity as $\bar d(u,v)$. Furthermore the completion is optimal in the following sense:  For all $u,v\in P$ such that $\bar d(u,v) \neq M, M+1$ we have
$$d'(u,v) \geq \bar{d}(u,v) \geq M +1 \text{ or } d'(u,v) \leq \bar{d}(u,v) \leq M.$$
Furthermore every automorphism of $(G,d,P)$ is also an automorphism of $(G,\bar d,P)$.
\end{thm}

\begin{proof}
As in Theorem \ref{thm:anticompletion} this follows directly from the properties of the bipartite completion algorithm with parameter $M$ for the non-antipodal class $\mathcal A^{\delta-1}_{\infty,0,2\delta+2,2\delta-1}$.
\end{proof}

As a direct consequence we obtain the following for odd diameters $\delta$:

\begin{corollary} \label{cor:antibioddcompletion}
Let $3\leq \delta<\infty$ be odd and $M= \frac{\delta-1}{2}$. Let $(G,d,P)$ be antipodally symmetric and connected.
Suppose that $\str{G} = (G,d)$ has a completion into $\mathcal A^{\delta-1}_{\infty,0,2\delta+2,2\delta-1}$. Then there is a unique completion $(G,\bar d) \in \mathcal A^{\delta-1}_{\infty,0,2\delta+2,2\delta+1}$ that is optimal in the following sense:
Let $\str{G}'=(G,d')\in\mathcal A^{\delta-1}_{\infty,0,2\delta+2,2\delta-1}$ be an arbitrary completion of $\str{G}$, then, for all $u,v \in G$ one of the following holds:
\begin{enumerate}
\item $d'(u,v) \geq \bar{d}(u,v) \geq M+1$ or 
\item $d'(u,v) \leq \bar{d}(u,v) \leq M$ or 
\item $M \leq \bar{d}(u,v) \leq M + 1$
\end{enumerate}
Furthermore $\Aut(G,d) = \Aut(G,\bar d)$.
\end{corollary}

\begin{proof}
Since $\delta$ is odd, we have that $M + (M+1) = \delta$. Hence the optimality property for pairs $u,v \in P$, described in Theorem \ref{thm:antibicompletion} transfers to all pairs $u,v \in G$ (by considering antipodal quadruples). 

As in the proof of Corollary \ref{cor:antievencompletion}, this optimality properties imply that the completion does not depend of the choice of $P$. Furthermore analogously to Corollary \ref{cor:antievencompletion} one can show that automorphisms of $(G,d)$ are also automorphisms of $(G,\bar d)$.
\end{proof}

We remark that in the last remaining case --- bipartite antipodal spaces of even diameter --- again the standard algorithm for bipartite spaces fails, as shown in the following example:
\begin{example}
\begin{figure}[t]
\centering
\includegraphics{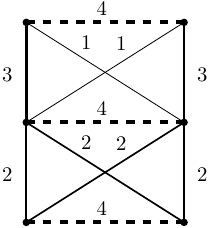}
\caption{An example of graph which fails to be completed to $\mathcal A^4_{\infty,0,10,9}$ using the algorithm given by Definition~\ref{defn:bimagiccompletion}.}
\label{fig:failedexample}
\end{figure}
Consider $\mathcal A^4_{\infty,0,10,9}$ which is an example of a class of bipartite antipodal metric spaces.  Completing the graph 
depicted in Figure~\ref{fig:failedexample} using Definition~\ref{defn:bimagiccompletion} and $M = 2$ will result in non-antipodal metric space where every non-edge will be completed by $3$.
\end{example}

\subsection{Antipodality-preserving EPPA} \label{sec:antieppa}
In the following we add extra layer to the construction of Herwig and Lascar
which makes sure that vertex closures are preserved.  This idea is the same as in~\cite{Evans2,Evans3}.

\begin{thm}
\label{thm:EPPAantipodal}
Let $(\delta,K_1,K_2,\allowbreak C_0,C_1,\mathcal S)$ be an antipodal sequence of admissible parameters.
\begin{enumerate}
\item If $K_1<\infty$ and $\delta$ is even then the class $\mathcal A^\delta_{K_1,K_2,C_0,C_1}\cap \mathcal A^\delta_\mathcal S$ has coherent EPPA.
\item If $K_1=\infty$ and $\delta$ is odd then the class $\mathcal A^\delta_{K_1,K_2,C_0,C_1}\cap \mathcal A^\delta_\mathcal S$ has coherent EPPA.
\item If $K_1<\infty$ and $\delta$ is odd then $S=\emptyset$ and the class $\mathcal B^\delta_{K_1,K_2,C_0,C_1}$ expanding $\mathcal A^\delta_{K_1,K_2,C_0,C_1}$
by an additional unary predicate $P$ determining the podality (so every edge of length $\delta$ has precisely one vertex $v\in P$)
has coherent EPPA.
\item If $K_1=\infty$ and $\delta$ is even then $S=\emptyset$ and the class $\mathcal B^\delta_{K_1,K_2,C_0,C_1}$ expanding $\mathcal A^\delta_{K_1,K_2,C_0,C_1}$
by an additional unary predicate $P$ determining the podality has coherent EPPA.
\end{enumerate}
\end{thm}
\begin{proof}
We first give a proof for the even diameter non-bipartite case with $\mathcal S=\emptyset$ and then discuss how to extend the construction
for remaining cases.

\paragraph{1}
Given $\str{A}=(A,d)\in \mathcal \mathcal A^{\delta}_{K_1,K_2,C_0,C_1}$ without loss of generality we can assume that for every vertex $u\in A$ there exists unique vertex $v\in A$, with $d(u,v)=\delta$. Denote by $\mathcal O$ the set of obstacles of $\mathcal A^{\delta-1}_{K_1,K_2,C_0,C_1}$ given by Lemma~\ref{lem:obstacles}.
Apply Theorem~\ref{thm:herwiglascar} to obtain $\str{B}\in \Forb(\mathcal O)$ which is a coherent EPPA-witness of $\str{A}$ and has no homomorphic image of any structure in $\mathcal O$.
Note that $\str{B}=(B,d)$ contains edges of type $\delta$ and thus it can not be completed to a metric space in $\mathcal A^{\delta-1}_{K_1,K_2,C_0,C_1}$.

Now we describe edge-labelled graph $\str{C}=(C,d')$.
$C$ is the set of all ordered pairs $(u,v)$ where $u\neq v\in B$, $d(u,v)=\delta$. We set
$$d'((u,v),(u',v')) =
\begin{cases}
 0 & \hbox{if } (u,v)=(u',v')\\
 \delta & \hbox{if } u=v', v=u'\\
 d(u,u') & \hbox{if $(u,v,u',v')$ is an antipodal quadruple.}\\
\end{cases}
$$

Clearly edges of distance $\delta$ form a (complete) matching in $\str{C}$.

\medskip

Given $u\in A$ put $\varphi(u) = (u,v)$ where $v$ is the unique vertex of $\str{A}$ such that $d(u,v)=\delta$.
It is easy to check that $\varphi$ is an embedding $\str{A}\to \str{C}$.
We show that $\str{C}$ is an coherent EPPA-witness of $\varphi(\str{A})$.

\medskip

Denote by $\pi\colon C\to B$ the {\em projection} assigning $(u,v)\mapsto u$.
Let $\psi$ be a partial isometry
of $\varphi(\str{A})$. Then $\psi\circ \pi$ is a partial isometry of $\str{A}$. Now extend this isometry to $\bar{\psi}\colon \str{B}\to \str{B}$. It follows that the mapping $\bar\psi'\colon C\to C$ defined by $\bar\psi'(u,v)=(\bar\psi(u),\bar\psi(v))$ is an automorphism of $\str{C}$ extending $\psi$.

\medskip

It remains to show that there is a completion of $\str{C}$ to a metric space $\str{C}'\in \mathcal A^{\delta}_{K_1,K_2,C_0,C_1}$.
For every pair of vertices $(u,v)$, such that $d'(u,v) = \delta$ choose arbitrarily one of the vertices and denote by
$\str{C}_{\delta-1}$ the subgraph induced by $C$ on all those vertices. 
As the edges of length $\delta$ form a complete matching, we know that $\str{C}_{\delta-1}$
is precisely half of the graph $\str{C}$ and whenever a distance is defined in
$C\setminus C_{\delta-1}$ it is part of a unique antipodal quadruple.

Note that $\pi(\str{C}_{\delta-1})$ is an isomorphic copy of $\str{C}_{\delta-1}$ to $\str{B}$ and
$\str{C}_{\delta-1}$ does not contain any edges of length $\delta$. Because $\K$ is
closed for antipodal companions, we then know that $\str{C}_{\delta-1}\in \Forb(\mathcal O)$.
Denote by $\str{C}'_{\delta-1}\in \K$ its completion with magic parameter $M = \frac{\delta}{2}$ and
by $\str{C}'\in \K^\ominus$ the completion of $\str{C}$ which extends $\str{C}'_{\delta-1}$
antipodally to the remaining vertices.

We proved that there is a completion of $\str{C}$ into $\mathcal A^\delta_{K_1,K_2,C_0,C_1}$, so now we can look at the completion with magic parameter $M$ of $\str{C}$ given by Corollary~\ref{cor:antievencompletion}.
This completion (which in fact is equivalent to one described above) preserves all automorphisms
of $\str{C}$, which is what we wanted.

Now consider the case when $\mathcal S\neq \emptyset$. Again we use the fact that the completion algorithm will never
introduce new Henson substructures and that the set of obstacles can be extended by $\mathcal S$.

\paragraph{2}
Now we adjust the construction for the bipartite case of odd diameter. Consider again
$\str{A}=(A,d)\in \mathcal \mathcal A^{\delta}_{K_1,K_2,C_0,C_1}$ such that for
every vertex $u\in A$ there exists unique vertex $v\in A$, $d(u,v)=\delta$.
Denote by $\mathcal O$ the set of obstacles given by Lemma~\ref{lem:biobstacles} for
$\mathcal A^{\delta-1}_{K_1,K_2,C_0,C_1}$ (which is a non-antipodal bipartite
amalgamation class). By the same construction as we used to build the EPPA-witness $\str{C}$ in the proof of Theorem~\ref{thm:biregulareppa}
we obtain a coherent EPPA-witness $\str{B}\in \Forb(\mathcal O)$ of $\str{A}$ without odd cycles.
Now we use the same construction as above to obtain a coherent
EPPA-witness $\str{C}\in \Forb(\mathcal O)$ where edges of length $\delta$ form a matching. Note that, since $\str{B}$ can be chosen to be connected, we can also assume that $\str{C}$ is connected. Finally we apply Corollary~\ref{cor:antibioddcompletion} to complete it to the bipartite
metric space in $\mathcal A^{\delta}_{K_1,K_2,C_0,C_1}$.

\paragraph{3}
In the third case of odd diameter non-bipartite spaces we proceed similarly to the first case, however this time with the explicit predicate $P$.  Consider a structure 
$\str{A}=(A,d,P) \in \mathcal B^{\delta}_{K_1,K_2,C_0,C_1}$, such that for every vertex $u\in A$ there exists a unique 
 vertex $v\in A$ with $d(u,v)=\delta$.  Denote by $\mathcal O$ the
set of obstacles given by Lemma~\ref{lem:obstacles} for $\mathcal
A^{\delta-1}_{K_1,K_2,C_0,C_1}$. We apply Theorem~\ref{thm:herwiglascar} to
obtain a structure $\str{B}\in \Forb(\mathcal O)$ which is a coherent EPPA-witness of
$\str{A}$.  Without loss of generality we can assume that the predicate $P$ in
$\str{B}$ picks precisely one vertex of every edge of length $\delta$ (all other
edges of length $\delta$ can be removed from $\str{B}$ because they are never
contained in a copy of $\str{A}$).
Then we construct $\str{C}$ in the same way as before, with $\pi$ being the projection of $\str{C}$ onto the $P$-part of $\str{B}$.

In the last step we complete $\str{C}$ according to the completion algorithm described in Definition \ref{defn:antiftmcompletion}. By Theorem \ref{thm:anticompletion} this completion preserves
automorphisms and gives the desired coherent EPPA-witness in $\mathcal B^{\delta}_{K_1,K_2,C_0,C_1}$.
\paragraph{4} The last case can be dealt with by the combination of case 2 and 3.
\end{proof}

\subsection{Stationary independence relation}

In this section we discuss the existence of a stationary independence relation on antipodal spaces. As one might suspect the answer is related to the question if there is a deterministic completion algorithm with magic parameter. Our results show that if $K_1<\infty$ and $\delta$ is odd, or if $K_1=\infty$ and $\delta$ is even, there is no local SIR. 

\begin{thm} 
\label{thm:antiSIR}
Let ($\delta, K_1, K_2, C_0, C_1, \mathcal S$) be a sequence of admissible antipodal parameters. Then the following holds in $\Gamma^\delta_{K_1,K_2,C_0,C_1,\mathcal S}$:
\begin{enumerate}
\item If $K_1<\infty$ and $\delta$ is even, there is a stationary independence relation.
\item If $K_1=\infty$ and $\delta$ is odd, there is no stationary independence relation, but a local one.
\item If $K_1<\infty$ and $\delta$ is odd, there is no local stationary independence relation.
\item If $K_1=\infty$ and $\delta$ is even, there is no local stationary independence relation.
\end{enumerate}
\end{thm}

\begin{proof} 
First observe that we may consider stationary independence relation over antipodally closed structures only because the antipodal completion
is unique.
\paragraph{1}
For $K_1<\infty$ and even diameter $\delta$, Corollary~\ref{cor:antievencompletion} gives us a completion algorithm. For structures $\str{A}, \str{B}, \str{C} \in \mathcal A ^\delta_{K_1,K_2,C_0,C_1,\mathcal S}$ we define $\str{A} \oplus_{\str{C}} \str{B}$ to be the space obtained by first forming the free amalgam of $\str{A}$ and $\str{B}$ over $\str{C}$ and then forming the completion with magic parameter $M = \frac{\delta}{2}$. This operator is a canonical symmetric amalgamation operator, hence there is a stationary independence relation on $\Gamma^\delta_{K_1,K_2,C_0,C_1,\mathcal S}$ (cf. Corollary~\ref{cor:SIR}).

\paragraph{2}
For $K_1 = \infty$ and odd diameter $\delta$, use Corollary~\ref{cor:antibioddcompletion} for connected edge-labelled graphs. Then, as in the proof of Corollary \ref{cor:biSIR} one can show, that there is a local stationary independence relation. In order to see that there is no stationary independence relation, we show that there is no canonical symmetric amalgamation in $\mathcal A ^\delta_{K_1,K_2,C_0,C_1,\mathcal S}$ over the empty structure. For that let $\str{A}$ consist of an antipodal pair $a_1,a_2$ and let $\str{B}$ be consist of a single vertex $b$. If there was a canonical symmetric amalgamation operator then in $\str{A} \oplus_{\empty} \str{B}$ we have that $d(b,a_1) = \delta - d(b,a_2)$, and since $\delta$ is odd $d(b,a_1) \neq d(b,a_2)$. Hence $\{a_1\} \oplus_{\empty} \str{B}$ and $\{a_2\} \oplus_{\empty} \str{B}$ are non-isomorphic, which contradicts the assumption that $\oplus$ is a canonical symmetric amalgamation operator.

\paragraph{3}
Next assume that $K_1 < \infty$, $\delta$ is odd, and let $M = \lfloor \frac{\delta}{2} \rfloor$. We are going to show that there is no local canonical symmetric amalgamation operation on $\mathcal A^{\delta}_{K_1,K_2,C_0,C_1, \mathcal S}$. Suppose for a contradiction that there is such an operator $\oplus$. Let $\str{C} = (\{c_0,c_1\},d)$ with $d(c_0,c_1) = M$, let $\str{A} =  (\{c_0,c_1, a_0, a_1\},d)$ with $d(a_0,a_1) = \delta$, $d(a_i,c_i) = M$ and $d(a_i,c_{1-i}) = M+1$, where $i\in\{0,1\}$. Note that the triangles $a_0,c_0,c_1$ and $a_1,c_1,c_0$ are isomorphic. Also let $\str{B} = (\{c_0,c_1, b\},d)$ with $d(b,c_0) = d(b,c_1) = M$. Let us consider the canonical amalgam $\str{A} \oplus_{\str{C}} \str{B}$. Note that in this amalgam $b$ cannot be identified with $a_0$ or $a_1$, so we can write $\str{A} \oplus_{\str{C}} \str{B}$ as $(\{a_0,a_1,b,c_0,c_1\},d)$ for some distance function $d$. By the antipodality we have that $d(a_0,b) = \delta - d(a_1,b)$, hence $d(a_0,b) \neq d(a_1,b)$. By monotonicity of $\oplus$, we have $(\{a_0,b,c_0,c_1\},d) = \{a_0,c_0,c_1\} \oplus_{\str{C}} \str{B}$ and  $(\{a_1,b,c_1,c_0\},d) = \{a_1,c_1,c_0\} \oplus_{\str{C}} \str{B}$. But $(\{a_0,b,c_0,c_1\},d) $ and $(\{a_1,b,c_0,c_1\},d) $ are not isomorphic, which contradicts the assumption that $\oplus$ is an amalgamation operation.

\paragraph{4}
Finally let us assume that $K_1 = \infty$ and $\delta$ is even. Then let us consider structures $\str{A}, \str{B}, \str{C} \in \mathcal A^{\delta}_{K_1,K_2,C_0,C_1, \mathcal S}$, such that $\str{C} = (\{c\},d)$, $\str{B} = (\{b_1,b_2,c\},d)$ with $d(b_1,b_2) = \delta$ and $d(b_1,c) = d(b_2,c) = \frac{\delta}{2}$ and $\str{A} = (\{a,c\},d)$ with $d(a,c) = 1$. Assume that there is a local canonical symmetric amalgamation operation $\oplus$, then let us consider the canonical symmetric amalgam $\str{A} \oplus_{\str{C}} \str{B}$. By monotonicity of $\oplus$ we must have $d(a,b_1) = d(a,b_2)$ in this amalgam. Since $(b_1,b_2)$ is an antipodal pair this implies that $d(a,b_1) = d(a,b_2) = \frac{\delta}{2}$. However then $\str{A} \oplus_{\str{C}} \str{B}$ cannot be a bipartite distance graph, since then $(a, c, b_1)$ is a triangle of odd perimeter $\delta +1$, a contradiction.
\end{proof}

\section{Classes with infinite diameter}
Classes of infinite diameter have two basic types---the 3-constrained cases and tree-like graphs.
\subsection{3-constrained spaces of infinite diameter}
\label{sec:infinite}
The admissible numerical parameters with $\delta=\infty$ are
\begin{enumerate}
\item \ref{I} $K_1,C_0,C_1=\infty$, $K_2=0$ (the generic bipartite metric space)
\item \ref{II} $1\leq K_1<\infty$, $C_0,C_1,K_2=\infty$ (the generic metric spaces without short odd cycles),
\end{enumerate}
where in the first case no Henson constraints are allowed, while in the second case it is possible to have $\mathcal S = \{\str K_n\}$, where $\str K_n$ is the clique on $n$ vertices for some $n\geq 4$ (see Theorem~\ref{thm:admissible_henson}).
\begin{thm}
The class $\overrightarrow{\mathcal A}^\infty_{\infty,0,\infty,\infty}$ of convex orderings of $\mathcal  A^\infty_{\infty,0,\infty,\infty}$ with an additional unary predicate determining the bipartition is Ramsey and has the expansion property.
\end{thm}
\begin{proof}
For every choice of finite $\str{A}$, $\str{B}$ denote by $\delta$ the maximal distance in $\str{B}$.
Then apply Theorem~\ref{thm:biramseyregular} for $\overrightarrow{\mathcal A}^\delta_{\infty,0,C_0,2\delta+1}$ where $C_0>3\delta+1$ even
to obtain $\str{C}\longrightarrow(\str{B})^\str{A}_2$. It is easy to see that $\str{C}\in \overrightarrow{\mathcal A}^\infty_{\infty,0,\infty,\infty}$.

The expansion property follows again by the standard argument.
\end{proof}
\begin{thm}
For every $1\leq K_1\leq \infty$ and every $\mathcal S \in \{\emptyset, \{\str K_4\}, \{\str K_5\},\ldots\}$ where $\str K_n$ is the clique on $n$ vertices, the class $\overrightarrow{\mathcal K}$ of free orderings of $\mathcal K =\mathcal  A^\infty_{K_1,\infty,\infty,\infty} \cap \mathcal A^\delta_{\mathcal S}$ is Ramsey and has the expansion property.
\end{thm}
\begin{proof}
Analogously to the previous proof, for every finite $\str{A}$, $\str{B}$ we choose $\delta$ to be the maximal distance in $\str{B}$ and reduce
to Theorem~\ref{thm:ramseyHenson}.
\end{proof}
\begin{thm}
The classes $\mathcal  A^\infty_{K_1,\infty,\infty,\infty}\cap \mathcal A^\delta_{\mathcal S}$, $1\leq K_1\leq \infty$, $\mathcal S \in \{\emptyset, \{\str K_4\}, \{\str K_5\},\ldots\}$ where $\str K_n$ is the clique on $n$ vertices,
and $\mathcal  A^\infty_{\infty,0,\infty,\infty}$ have coherent EPPA.
\end{thm}
\begin{proof}
Again, this follows by a reduction to Theorems~\ref{thm:EPPAHenson} and~\ref{thm:biregulareppa}.
\end{proof}
The stationary independence relation follows from the existence of completion algorithm
which adds the shortest path distances, as discussed in Section~\ref{sec:shortestpathsec}.
From this we immediately get:
\begin{thm}
The classes $\mathcal  A^\infty_{K_1,\infty,\infty,\infty}\cap \mathcal A^\delta_{\mathcal S}$, $1\leq K_1\leq \infty$, $\mathcal S \in \{\emptyset, \{\str K_4\}, \{\str K_5\},\ldots\}$ where $\str K_n$ is the clique on $n$ vertices,
and $\mathcal  A^\infty_{\infty,0,\infty,\infty}$ have no stationary independence relations
but local stationary independence relations that can be defined by means of the shortest path
completion algorithm.
\end{thm}

\subsection{Tree-like graphs}
\label{sec:tree-like}
Recall Definition~\ref{defn:tree-like} of tree-like graphs and denote by $\mathcal A_{T_{m,n}}$ the age of metric space associated with the tree-like
graph $T_{m,n}$.
\begin{thm}
For every $2\leq m,n\leq \infty$ the class $\mathcal A_{T_{m,n}}$ has no EPPA.
\end{thm}
\begin{proof}
For every $\str{F} \in \mathcal A_{T_{m,n}}$, let us call a vertex of $\str{F}$ a \emph{non-leaf} if it is contained in at least two blocks (maximal cliques) of $\str{F}$ and \emph{leaf} otherwise. Note that every $\str{F} \in \mathcal A_{T_{m,n}}$ contains a leaf.
Now suppose for a contradiction that $\mathcal A_{T_{m,n}}$ has EPPA, and consider any structure $\str{A}\in \mathcal A_{T_{m,n}}$ which contains a non-leaf. Given a finite EPPA-witness $\str{B}\in \mathcal A_{T_{m,n}}$ of $\str{A}$, we can assume without loss of generality that for every vertex $v$ of $\str{B}$ there exists $\phi\in \Aut(\str{B})$ such that $\phi(v)\in A$.
By the extension property it then follows that the graph $\str{B}$ is vertex transitive and thus either every vertex of $\str{B}$ is a leaf or every vertex of $\str B$ is a non-leaf, which contradicts to the assumption that $\str{A}$ contains both leaves and non-leaves as soon as $\str A$ is not a clique (e.g. when $\str A$ has more than $n$ vertices).
\end{proof}

\begin{thm}
For every $2\leq m,n\leq \infty$ the class $\mathcal A_{T_{m,n}}$ has no precompact Ramsey expansion.
\end{thm}
\begin{proof}
Assume for a contradiction that there is a precompact Ramsey expansion $\K$ of $\mathcal A_{T_{m,n}}$. Denote by $\str{A}^+_1,\str{A}^+_2,\ldots \str{A}^+_k$
all structures in $\K$ with one vertex. Let $\str{B}$ be a path of edges in distance $2$ with $k+1$ vertices completed to a metric space in $\Age(T_{m,n})$.  Denote by $\str{B}^+$ its
expansion in $\K$. There is some $\str{A}_i$, $1\leq i\leq k$, with at least two copies in $\str{B}^+$.  Denote by $\str{B}^+_2$ the substructure induced by $\str{B}^+$ on those two
copies and by $\ell$ their distance. 

By the Ramsey property, there is $\str{C}^+\in \K$ such that $\str{C}^+\longrightarrow (\str{B}^+_2)^{\str{A}^+_i}_{\ell+1}$ in $\K$. Without loss of generality we can assume that $\str{C^+}$ is an expansion of a connected fragment of $T_{m,n}$.
Fix a vertex $v\in \str{C^+}$ and colour the vertices $u\in \str{C}^+$ with $d(u,v) \mod {\ell+1}$.  Denote by $u_1$, $u_2$ the vertices of the monochromatic copy of $\str{A}^+_i$ in $\str{C}^+$, ensuring that $u_1$ is closer to $v$ than $u_2$.  Notice that between any pair of vertices at distance greater than 1 there is precisely one shortest path, so it follows that either $d(v,u_2)=d(v,u_1)+\ell$ or $d(v,u_1) + d(v, u_1) = \ell$.
By the choice of the colouring this pair is not monochromatic, contradicting the choice of $\str{C^+}$.
\end{proof}

\begin{lem} \label{lem:generictreelike}
For all $2 \leq m, n \leq \infty$ with  $m < \infty$ or $n \neq 2,3, \infty$ there is no local stationary independence relation on the metric space associated with the tree-like graph $T_{m,n}$
\end{lem}

\begin{proof}
For a contradiction, assume that there is a local stationary independence relation $\ind$. By Theorem \ref{thm:canonicalamalg} this is equivalent to the existence of a local canonical amalgamation $\oplus$ on the $\mathcal A_{T_{m,n}}$.

First let us look at the case where $m < \infty$. Then let $\str{C}$ consist of a single vertex $c$ and let $\str{A} = \{c,a_1,a_2,\ldots, a_{m}\}$ such that $d(c,a_i) = 1$ and $d(a_i,a_j) = 2$ for all $i \neq j$. And let $\str{B} = \{b,c\}$ with $d(b,c) = 1$. Then, by the definition of $(T_{m,n},d)$, in the amalgam $\str{A} \oplus_{\str{C}} \str{B}$ there has to be an index $i$, such that $d(b,a_i)$ is 1 or 0 and $d(b,a_j) = 2$ for all $j \neq i$. Thus $\{a_i,c\} \oplus_{\str{C}} \str{B}$ and $\{a_j,c\} \oplus_{\str{C}} \str{B}$ are non-isomorphic for $j \neq i$. But this contradicts the assumption that $\oplus$ is an amalgamation operator.

Next assume that $n \notin \{ 2,3, \infty \}$. Let $\str{C} = \{c_1,c_2\}$ be such that $d(c_1,c_2) = 1$, let $\str{B} = \{c_1,c_2,b\}$ be a clique of size $3$ and let $\str{A} = \{c_1,c_2, a_1,\ldots, a_{n-2} \}$ be a clique of size $n$. Then, by definition of $\mathcal A_{T_{m,n}}$, in the amalgam $\str{A} \oplus_{\str{C}} \str{B}$ the vertex $b$ has to be identified with one of the vertices $a_i$ in $\str{A}$. By monotonicity of $\oplus$ we have $\{a_i,c_1,c_2\} \oplus_{\str{C}} \str{B}$ is a clique of size 3, but for every $j \neq i$, $\{a_j,c_1,c_2\} \oplus_{\str{C}} \str{B}$ is a 4-clique. This contradicts the assumption that $\oplus$ is an amalgamation operator.
\end{proof}

Surprisingly, for the remaining cases there exists no SIR, but there is a local SIR. In order to show this, we first need a better understanding of how the metric space $(T_{n,m},d)$ relates to the underlying tree-like graph.

\begin{defn}
Let $(G,E)$ be a graph, inducing the graph metric $(G,d)$, and let $\str{A} = (A,d)$ be a subspace of $(G,d)$. Let $a,b \in A$, then we say that a path $a = y_0, y_1, \ldots, y_k = b$ in $(G,E)$ \emph{witnesses the distance $d(a,b)$}, if $d(a,b) = k$.
\end{defn}

\begin{lem} \label{lem:uniquetreelike}
Let $\str{A}$ be a subspace of $(T_{n,m},d)$ and let $(G_A,E)$ be a subgraph of $T_{n,m}$, consisting of paths witnessing the distances in $\str{A}$. Then, $(G_A,E)$ is uniquely determined by $\str{A}$. Also, if $\str{A} \cong \str{B}$, then $(G_A,E) \cong (G_B,E)$.
\end{lem}

\begin{proof}
It is enough to show that for every pair $u,v$ in $\str{A}$, there is exactly one path in $T_{n,m}$ witnessing its distance $d(u,v) =c$. Then clearly $(G_A,E)$ is uniquely determined by $\str{A}$; the second statement follows from the homogeneity of $(T_{n,m},d)$.

For a contradiction, assume that $c \geq 2$ and that there are two paths $u = x_0, x_1, x_2, \ldots, x_c = v$ and $u = y_0, y_1, y_2, \ldots, y_c = v$. Since the two paths are not equal, there have to be indices $i < k < c$ such that $x_i = y_i$, $x_{i+1} \neq y_{i+1}$, $\ldots$, $x_k \neq y_k$ and $x_{k+1} = y_{k+1}$. Then, by definition of $T_{n,m}$, the set $\{x_i, x_{i+1}, \ldots x_k\} \cup \{y_{i+1}, \ldots y_k\} $ has to be a clique. But this implies that $d(u,v) < c$, which contradicts our assumptions.
\end{proof}

\begin{thm}
For every $m,n$, the metric space associated with tree-like graph $T_{m,n}$ has no stationary independence relation. It has a local stationary independence relation if and only if $m = \infty$ and $n \in \{2,3,\infty\}$.
\end{thm}

\begin{proof}
By Lemma \ref{lem:generictreelike} the statement about local SIR holds for all cases where $m \neq \infty$ or $n \notin \{2,3, \infty \}$.

In order to prove that $(T_{m,n},d)$ has no SIR we show that there is no canonical symmetric amalgamation operator on $\mathcal A_{T_{m,n}}$ (cf. Theorem \ref{thm:canonicalamalg}). Assume there is such an operator $\oplus$. Let $\str{A} = \{a_1,a_2,a_3\}$ with $d(a_1,a_2)=d(a_2,a_3)=1$ and $d(a_1,a_3) = 2$ and let $\str{B}$ consist of a single point $b$. We claim that then, no matter how the amalgam $\str{A} \oplus \str{B}$ is formed, we have $d(b,a_1) \neq d(b,a_2)$ or $d(b,a_1) \neq d(b,a_3)$. In order to prove this claim, let us assume that $d(b,a_1) = d(b,a_3)$. By Lemma \ref{lem:uniquetreelike} the graph witnessing the distances in $\str{A} \oplus_{\empty} \str{B}$ has to consist of a path from $b$ to $a_2$ and the two edges $d(a_2,a_3) = d(a_2,a_1) = 1$. Hence $d(b,a_2) = d(b,a_1)-1$, which proves our claim. By the claim and monotonicity of $\oplus$, $\{a_1\} \oplus_{\empty} \str{B}$ is non-isomorphic to $\{a_i\} \oplus_{\empty} \str{B}$ for some $i \in \{2,3\}$. But this contradicts the assumption that $\oplus$ is an amalgamation operator.

It is only left to show that for $m = \infty$ and $n \in \{2,3,\infty\}$, there is a local stationary independence relation on  $(T_{m,n},d)$. We do so again by finding a local canonical symmetric amalgamation operator. So let $\str{A}, \str{B}, \str{C}$ be non-empty substructures of $(T_{m,n},d)$ such that $e_1\colon  \str{C} \to \str{A}$ and $e_2\colon  \str{C} \to \str{B}$ are embeddings. By Lemma \ref{lem:uniquetreelike} we can assume that every distance in the structures $\str{A}, \str{B}, \str{C}$ is witnessed by a path (so informally we can think about $\str{A}, \str{B}, \str{C}$ as the underlying tree-like graphs instead of the metric spaces).

If $m = \infty$ and $n = 2$ we define $\str{A} \oplus_{\str{C}} \str{B}$ to be the free amalgam of the trees $\str{A}$ and $\str{B}$ over $\str{C}$. This graph is again a tree and thus in $\mathcal A_{T_{\infty,2}}$. It is easy to see that this symmetric amalgamation operator is associative and monotone.

If $m = \infty$ and $n = 3$ we construct $\str{A} \oplus_{\str{C}} \str{B}$ again by first forming the free amalgam of the graphs $\str{A}$ and $\str{B}$ over $\str{C}$. This free amalgam might contain $c_1,c_2,a,b$, such that $a,c_1,c_2$ is a 3-clique and $b,c_1,c_2$ is a 3-clique, but there is no edge between $a$ and $b$. In this case we identify $a$ and $b$. Again this symmetric amalgamation operator is associative and monotone.

If $m = \infty$ and $n = \infty$ we construct $\str{A} \oplus_{\str{C}} \str{B}$ by first forming the free amalgam of the graphs $\str{{A}}$ and $ \str{{B}}$ over $\str{C}$. Then we add edges to complete all the subgraphs $A' \cup B'$ with $A' \subseteq A$ and $B' \subseteq B$, such that $A'$ and $B'$ are cliques of size $\geq 3$ and share at least one edge. Again this symmetric amalgamation operator is associative and monotone.
\end{proof}
\section{Corollaries}
While the main focus of our work is on the combinatorial properties of
metrically homogeneous graphs, let us briefly discuss the corollaries of our
results for the automorphism group of the \Fraisse{} limits.  This section is not meant to be
exhaustive and we refer the reader to a recent survey~\cite{NVT14} and thesis~\cite{Siniora2}
which discuss many of the topics in a greater detail.
\subsection{Ample generics}
\label{sec:ample}
The automorphism group of the \Fraisse{} limits of amalgamation classes in the
catalogue can be seen as Polish groups when endowed with the pointwise convergence
topology. A Polish group has {\em generic automorphisms} if it contains a comeagre conjugacy class.
A Polish group has {\em ample generics} if it has a comeagre diagonal conjugacy class in every dimension.

Known examples of structures with ample generics include $\omega$-stable $\omega$-categorical structures \cite{hodges1993b}, the random graph~\cite{hodges1993b,hrushovski1992}, the homogeneous $K_n$-free graph~\cite{herwig1998}, the rational Urysohn space~\cite{solecki2005,vershik2008}, free homogeneous structures over finite relational languages~\cite{Siniora} and Philip Hall's locally finite universal group~\cite{Siniora2}. 

Our study of coherent EPPA is directly motivated by the following result:
\begin{thm}[Hodges--Hodkinson--Lascar--Shelah~\cite{hodges1993b}, see also Siniora--Solecki~\cite{Siniora2}]
Suppose that $\str{M}$ is a homogeneous structure such that $\Age(\str{M})$ has both EPPA and APA. Then $\str{M}$ has ample generics.
\end{thm}
Where we say a class $\K$ of finite structures has the {\em amalgamation property
with automorphisms (APA)} if whenever $\str{A},\str{B}_1,\str B_2\in \K$ with embeddings $\alpha_1\colon \str{A}\to\str{B}_1$ and $\alpha_2\colon \str{A}\to\str{B}_2$ then there is a structure $\str{C}\in\K$ with embeddings $\beta_1\colon \str{B}_1\to\str{C}$ and $\beta_2\colon \str{B}_2\to\str{C}$ such that $\beta_1\circ\alpha_1=\beta_2\circ\alpha_2$ and whenever $f\in \Aut(\str{B}_1)$ and $g\in \Aut(\str{B}_2)$ such that $f\circ\alpha_1(\str{A})=\alpha_1(\str{A})$, $g\circ \alpha_2(\str{A})=\alpha_2(\str{A})$, and for every $a\in \str{A}$ we have $\alpha_1^{-1}\circ f\circ \alpha_1(a)=\alpha_2^{-1}\circ g\circ\alpha_2(a)$, then there exists $h\in \Aut(\str{C})$ which extends $\beta_1\circ f\circ\beta^{-1}_1 \cup \beta_2\circ g\circ \beta^{-1}_2$.
Observe that APA follows from our canonical amalgamation and the automorphism preservation lemma for
all classes where we shown the coherent EPPA. As a corollary of Theorem~\ref{thm:EPPAmain} we thus
extended the list of known structures with ample generics by many new examples.

Ample generics are a powerful tool to prove additional properties.
Among multiple consequences of the ample generic we list the following:
\begin{thm}[Kechris--Rosendal~\cite{Kechris2007}]
Suppose that $G$ is a Polish group with ample generics.  Then $G$ has the small index property.
\end{thm}
\begin{thm}[Kechris--Rosendal~\cite{Kechris2007}]
Suppose that $\str{M}$ is an $\omega$-categorical structure with ample generics.  Then $\Aut(\str{M})$ has uncountable cofinality and 21-Bergman property.
\end{thm}
\begin{thm}[Kechris--Rosendal~\cite{Kechris2007}]
Suppose that $\str{M}$ is an $\omega$-categorical structure with ample generics.  Then $\Aut(\str{M})$ has Serre's property (FA).
\end{thm}
Consequently we also gave a number of new examples of structures whose automorphism group has small index property, uncountable cofinality, 21-Bergman property and Serre's property (FA).

\subsection{Amenability and unique ergodicity}

In this section, let $G = \Aut(\text{Flim}(\mathcal{K}))$, be the automorphism group of the \Fraisse{} limit of the class $\mathcal{K}$ endowed with the topology of pointwise convergence. In the following we set up the machinery needed to discuss when $G$ is amenable (its universal minimal flow has a $G$-invariant measure) or uniquely ergodic (a unique such measure). See \cite{NVT14} and \cite{AKL14} for a more in depth discussion of the definitions and their history.

We now state precise definitions and lemmas which capture the idea that ``amenability and unique ergodicity can sometimes be shown by counting finite quantities related to the number of Ramsey expansions of some finite structures." We have in mind the example that arbitrary linearly ordered metric graphs (with a fixed set of parameters) is a precompact Ramsey expansion of the corresponding metric graphs (see Theorem~\ref{thm:ramseyall}).

Let $L \sse L^+$ be languages. Let $\K$ and $\K^+$ be classes of structures in $L$ and $L^+$ respectively such that $\K^+\vert L = \K$ (i.e. $\K^+$ is an expansion of $\K$). If $\str{B} \in \K^+$ then we write $\str{B}^+ = \expand{\str{A}}$ where $\str{A} = \str{B}^+ \vert L$ and $\str{A}^+ = \str{B}^+ \vert (L^+ \setminus L)$. Colloquially, ``$\str{A}$ is the old stuff, and $\str{A}^+$ is the new stuff'' when we use the representation $\expand{\str{A}}$. 

Let $\str{A} \leq \str{B}$ (where by $\leq$ we mean that $A\subseteq B$ and the relations induced by $\str{B}$ on $A$ give precisely $\str{A}$) be structures in $\K$ and let $\expand{\str{A}} \in \K^+$. Then we write:
\[
	\Nexp_{\K^+} (\str{A}^+, \str{B}) := \left\vert \{ \str{B}^+ : \expand{\str{A}} \leq \expand{\str{B}} \in \K^+ \}\right\vert,
\]
which is the number of expansions of $\str{B}$ in $\K^+$ that extend $\expand{\str{A}}$. If there is no confusion then we write $\Nexp(\str{A}^+,\str{B})$.

The following black-box lemma reduces checking amenability to a counting argument.

\begin{lem}[Angel--Kechris--Lyons~\cite{AKL14}, Pawliuk--Soki{\'c}~\cite{PawliukSokic16}]\label{lem:amenable_same_no_of_exp}
Let $\K^+$ be a precompact Ramsey expansion of $\mathcal{K}$. If for every $\str{A} \leq \str{B}$ in $\K$ and every $\expand[\prime]{\str{A}}, \expand[\prime\prime]{\str{A}} \in \K^+$ we have
\begin{align}
	\Nexp(\str{A}^\prime,\str{B}) = \Nexp(\str{A}^{\prime\prime},\str{B}) \label{ass:samenumberofexp}
\end{align}
then $\Aut(\Flim(\K))$ is amenable. 

In particular, this will happen if the values in equation (\ref{ass:samenumberofexp}) only depend on the isomorphism type of $\str{A}$ and $\str{B}$.
\end{lem}

For fixed structures $\str{A}$ and $\str{B}$ with expansions $\expand{\str{A}}$ and $\expand{\str{B}}$, define
\[
	\Nemb(\str{A},\str{B}) := \left\vert\{\str{A} : \str{A} \leq \str{B}\}\right\vert.
\]
Let $\str{H}\in\K$ with $\vert \str{H} \vert = k$ and $\mathbf G$ is a random element of $\K$ with $\vert \mathbf{G} \vert = n$ (typically $k<<n$). Let
\[
	f(\mathbf G) := \frac{\Nemb(\str{H}, \mathbf G)}{\mathbb {E}[\Nemb(\str{H}, \mathbf G)]}, \hspace{1cm}f^+(\mathbf G) := \frac{\Nexp(\str{H}^+,\mathbf{G})}{\mathbb {E}[\Nemb(\str{H}, \mathbf G)]},
\]
where $\mathbb{E}$ is the expected value. See section 7.2 of \cite{PawliukSokic16} for a more in depth discussion of these notions.

The following black-box lemma reduces checking unique ergodicity to counting three quantities.

\begin{lem}[Angel--Kechris--Lyons~\cite{AKL14}, Pawliuk--Soki{\'c}~\cite{PawliukSokic16}]\label{lem:QOP_strategy} Using the notation defined above, suppose that $\Aut(\Flim(\K))$ is amenable, that changing $\mathbf G$ by a single edge changes $f$ and $f^+$ by no more than $O(\frac{1}{n^2})$, and that the number of expansions of $\mathbf G \leq O((n!)^k)$. Then $\Aut(\Flim(\K))$ is uniquely ergodic.
\end{lem}

Note that this counting trick goes back to Nešetřil and Rödl~\cite{Nesetvril1978} where the \emph{strong ordering property} is established for classes of graphs without short cycles, a key lemma from that paper is then used in~\cite{AKL14}. See also~\cite{nevsetvril2017statistics}.

In general these two lemmas are proved by explicit computations. While the computations asked for by Lemma \ref{lem:QOP_strategy} are straightforward, it is often tedious to set up and the quantities are somewhat opaque. So instead of writing down these computations for metric graphs, we will leverage computations that have already been made for unlabelled directed graphs \cite{PawliukSokic16}. Despite the differences in the relations of the underlying structures (there there is a single asymmetric relation, here there are many symmetric relations) the precompact expansions are nearly identical. These expansions only care about the definable equivalence classes of the structures. For example, in \cite{PawliukSokic16} the class $\mathcal{D}_2$ of all complete bipartite directed graphs have convex (with respect to the bipartition) linear orders as expansions, and here we investigate various classes of metric graphs that form a bipartition and who also have convex (with respect to the bipartition) linear orders as expansions. In both cases, when showing amenability by counting, once the underlying structures $\str{A} \leq \str{B}$ are chosen, all counting about the expansions becomes identical. In this sense, it is immaterial for our purposes that the computations in \cite{PawliukSokic16} are about directed graphs.

The next lemma makes this precise.

\begin{lem}\label{lem:UE_counting} Let $\K$ and $\mathcal{F}$ be \Fraisse{} classes in signatures $L \subseteq F$ respectively. Let $\pi$ be the map that forgets the extra structure in $F \setminus L$. Suppose that $\K^+$ and $\mathcal{F}^+$ are precompact Ramsey expansions of $\K$ and $\mathcal{F}$ respectively, such that $\pi[\mathcal{F}^+] = \K^+$.
\begin{enumerate}
	\item If amenability of $\Aut(\Flim(\K))$ was shown by counting (i.e. Lemma~\ref{lem:amenable_same_no_of_exp}), then $\Aut(\Flim(\F))$ is amenable by counting.
	\item If unique ergodicity of $\Aut(\Flim(\K))$ was shown by counting (i.e. Lemma~\ref{lem:QOP_strategy}), then $\Aut(\Flim\allowbreak(\F))$ is uniquely ergodic by counting.
\end{enumerate}
\end{lem}

The relevant computations from \cite{PawliukSokic16} will all be applicable in our case with at most minor modifications.

\medskip

The case of tree-like graphs is exceptional because these classes do not have precompact expansions. Clearly, $\Aut(T_{2,2})$ is amenable because $T_{2,2}$ is just the infinite path with no endpoints. We conjecture that in all the other cases, $\Aut(T_{m,n})$ is not amenable, see Conjecture~\ref{conj:treelike} for more discussion.

The non-tree-like classes will be dealt with combinatorially:
\begin{thm} If $\mathcal{K}$ is a primitive 3-constrained class of metric graphs, then $G$ is amenable and uniquely ergodic.
\end{thm}

\begin{proof} Since we have shown that $\overrightarrow{\mathcal{K}}$ is a precompact Ramsey expansion of $\mathcal{K}$ (see Theorem~\ref{thm:regularramsey}), a direct, counting proof is given by Angel, Kechris and Lyons~\cite{AKL14}. Furthermore, this expansion guarantees that $G$ is uniquely ergodic.

Alternatively, amenability of $G$ will follow from the fact that $\mathcal{K}$ has EPPA, although this is not enough for unique ergodicity. See Kechris and Rosendal~\cite{Kechris2007}.
\end{proof}


\begin{thm} If $\mathcal{K}$ is a bipartite 3-constrained class of metric graphs, then $G$ is amenable and uniquely ergodic.
\end{thm}

\begin{proof} Amenability of $G$ follows from the fact that $\mathcal{K}$ has EPPA, although this is not enough for unique ergodicity. See Kechris and Rosendal~\cite{Kechris2007}.

Alternatively, we have shown that adding linear orders that are convex with respect to the equivalence classes of even distances and unary marks to determine the bipartition is a precompact Ramsey expansion of $\mathcal{K}$ (see Theorem~\ref{thm:biramseyregular}). A direct counting argument for Lemma~\ref{lem:amenable_same_no_of_exp}, is
\[
	\Nexp(\str{A}^\prime,\str{B}) = 2!\cdot \frac{\vert B_1 \vert! \vert B_2 \vert!}{\vert A_1 \vert! \vert A_2 \vert!}
\]
where $B_1$, $B_2$ are the equivalence classes of even distances of $\str{B}$, and $A_1$, $A_2$ are the equivalence classes of even distances of $\str{A}$, and $\str{A}$ contains vertices from both bipartitions (otherwise the coefficient $2!$ changes). This computation appears as Theorem~4.2 in~\cite{PawliukSokic16} (there it is discussed in the context of generic multipartite digraphs).

The computation for unique ergodicity and Lemma~\ref{lem:UE_counting} appears as Theorem~8.2 in~\cite{PawliukSokic16}.
\end{proof}


\begin{thm} If $\mathcal{K}$ is an antipodal, non-bipartite class of metric graphs, then $G$ is amenable and uniquely ergodic.
\end{thm}

\begin{proof} When $\delta$ is even, amenability of $G$ follows from the fact that $\mathcal{K}$ has EPPA (see Theorem~\ref{thm:EPPAantipodal}).

We have shown that adding linear orders that are convex with respect to the podes and a unary marks denoting the podes is a precompact Ramsey expansion of $\mathcal{K}$ (see Theorem~\ref{thm:ramseyall} and Remark~\ref{rem:antiporamsey}). Observe that this expansion is bi-definable with the expansion by an order and unary marks for the podes, where in the order each edge of length $\delta$ forms an interval and the smaller elements from each $\delta$ pair are in the same pode.

This alternative expansion corresponds to what is discussed in~\cite{PawliukSokic16} for double-covers of the generic tournaments. Suppose, for simplicity, that both structures $\str{A}$ and $\str{B}$ are antipodally closed. Then for amenability, the direct counting argument for Lemma~\ref{lem:amenable_same_no_of_exp}, is
\[
	\Nexp(\str{A}^\prime,\str{B}) = \frac{b!}{a!} \cdot 2^{b-a},	
\]
where $b$ is the number of $\delta$-pairs in $\str{B}$ and $a$ is the number of $\delta$-pairs $\str{A}$. This computation appears as Theorem~4.3 in~\cite{PawliukSokic16}.

The computation for unique ergodicity and Lemma~\ref{lem:UE_counting} appears as Theorem~8.3 in~\cite{PawliukSokic16}.
\end{proof}


\begin{thm} If $\mathcal{K}$ is an antipodal, bipartite class of metric graphs, then $G$ is amenable and uniquely ergodic.
\end{thm}

\begin{proof} For some values of $\delta$ amenability follows from the fact that $\mathcal{K}$ has EPPA (see Theorem~\ref{thm:EPPAantipodal}).

There are two cases to consider: $\delta$ even and $\delta$ odd. When $\delta$ is even, then each $\delta$-pair lives inside a bipartite equivalence class.

In both cases, for simplicity and to keep this section concise, we assume that $\str{A}$ and $\str{B}$ are antipodally closed and that $\str{A}$ contains elements from each part of the bipartition. If this is not true, the computations go in a similar manner with some further technicalities.

For $\delta$ even: Similarly as for antipodal non-bipartite, one can show that adding unary marks for the bipartitions and the podes and linear orders that are convex with respect to the bipartitions and where each edge of length $\delta$ forms an interval is a precompact Ramsey expansion of $\mathcal{K}$ bi-definable with the one presented in Theorem~\ref{thm:ramseyall} and Remark~\ref{rem:antiporamsey}. A direct counting argument for Lemma \ref{lem:amenable_same_no_of_exp}, is
\[
	\Nexp(\str{A}^\prime,\str{B}) = \frac{b_1!}{a_1!} \cdot \frac{b_2!}{a_2!} \cdot 2^{b-a},
\]
where $b$ is the number of $\delta$-pairs in $\str{B}$, $a$ is the number of $\delta$-pairs in $\str{A}$, $b_i$ (for $i=1,2$) is the number of $\delta$-edges in the $i$-th part of the bipartition (and $a_i$ is defined similarly).

For $\delta$ odd: We have shown that adding unary marks for the bipartitions and linear orders that are convex with respect to the bipartition, and when restricted to the bipartitions (or equivalently in this case the podes) are isomorphic, is a precompact Ramsey expansion of $\mathcal{K}$ (See Theorem~\ref{thm:ramseyall}). A direct counting argument for Lemma \ref{lem:amenable_same_no_of_exp}, is
\[
	\Nexp(\str{A}^\prime,\str{B}) = \frac{\frac{\vert B \vert }{2}!}{\frac{\vert A \vert }{2}!}.
\]

For unique ergodicity, it suffices to notice that defining the order on one node in a $\delta$-pair uniquely determines the order for the second node. In this way, unique ergodicity of these classes follows from unique ergodicity of the corresponding bipartite and non-bipartite 3-constrained classes.
\end{proof}

\section{Conclusion and open problems}
\label{sec:conclussion}
Between the submission of this paper (2017) and its publication, there have been many developments, and in particular some of the open problems have been solved (and others have arisen). We decided to update this section and mention whenever we are aware of some progress, while keeping the original numbering.

\paragraph{1 and 2} The EPPA results of this paper need an extra expansion for the odd-diameter antipodal classes and the even-diameter bipartite antipodal classes. In the original version we asked whether this is necessary and argued that this is closely connected to EPPA for two-graphs:

The class $\mathcal A^3_{0,10,7,2}$ of all antipodal metric spaces of diameter $3$, the metric spaces can be equivalently
interpreted as double-covers of complete graphs. By {\em double-cover} of
$\str{K}_n$ we mean a graph $\str{G}$ with a 2--to--1 covering map from the
vertices of $\str{G}$ to those of $\str{K}_n$, so that each edge of $\str{K}_n$
is covered by two edges of $\str{G}$.
There are two double-covers of $\str{K}_3$, a pair of triangles and a hexagon,
see Figure~\ref{fig:doublecover}.
\begin{figure}[t]
\centering
\includegraphics{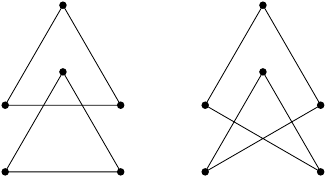}
\caption{Two double-covers of complete graphs with 3 vertices.}
\label{fig:doublecover}
\end{figure}

It is a well known fact that double-covers of complete graphs correspond to {\em two-graphs} (3-regular hypergraphs where
every subhypergraph with 4 vertices contains an even number of hyperedges):
given a double-cover $\str{G}$ of $\str{K}_n$ put vertices $a,b,c\in K_n$ to
a hyper-edge if and only if the corresponding double-cover is a hexagon.

Two-graphs have been extensively studied since 1960s~\cite{Seidel1973,Godsil2001} and many of those results are
relevant to our problem.
Because the correspondence between antipodal metric graphs and the underlying two-graphs is automorphism-preserving, by constructing
an EPPA-witness in $\str{A}\in \mathcal A^3_{0,10,7,2}$ we would also construct
an EPPA-witness of its associated two-graph. Extensions can be assumed to be
vertex transitive (because every vertex may be assumed to be contained in the copy)
and construction of vertex-transitive (or regular) two-graphs is related to
construction of strongly regular graphs and Taylor graphs. This is a well established
topic of algebraic graph theory and those graphs are rare.  Having a positive
answer to the problem of EPPA of the odd non-bipartite antipodal graphs would
thus require a construction of strongly regular graphs with even stronger symmetry
assumptions and thus seems difficult.

Removing the need for the unary expansion for the class $\str{A}\in \mathcal A^3_{0,10,7,2}$ would closely correspond to proving EPPA for two-graphs, a problem posed by Macpherson which also appears in~\cite{Siniora2}.

\medskip

There has been progress and the expansion was shown not to be necessary by Evans, Hubička, Konečný, and Nešetřil~\cite{eppatwographs} for the antipodal metric space of diameter $3$ (and for two-graphs) and by Konečný~\cite{Konecny2019a}. Curiously, while all the antipodal metric spaces do have coherent EPPA, coherence gets lost when applying the arguments to two-graphs. Hence the question whether the class of all finite two-graphs has EPPA still remains open:

\begin{question}[\cite{eppatwographs}]
Does the class of all finite two-graphs have coherent EPPA?
\end{question}

\paragraph{3} The classification of metrically homogeneous graphs is not a classification
of metric spaces with integer distances, because it includes the additional requirement of containing all \emph{geodesics}, i.e all triangles with edge lengths $a$, $b$ and $|a-b|$.
Our completion algorithm does not really rely on this requirement (though it is implicitly
used in the definition of the time function). It would be interesting to give a more general
characterisation of classes where such an approach works, possibly including non-binary
relations in the sense of homogenizations defined in~\cite{Hubicka2016}.

\paragraph{4} As already mentioned in the introduction, Sauer~\cite{Sauer2013} identified for which subsets $S\subseteq \mathbb R^+_0$ the class of all $S$-valued metric spaces is an amalgamation class. Hubi\v cka and Ne\v set\v ril~\cite{Hubicka2016} later proved the Ramsey property for all these classes. Conant~\cite{Conant2015} studied EPPA generalised metric spaces (that is metric spaces with values from some ordered \textit{distance} monoid). Finally, Braunfeld~\cite{Sam} proved that $\Lambda$-ultrametric spaces, where $\Lambda$ is a finite distributive lattice, also have the Ramsey property.

In~\cite{Hubicka2017sauer}, Hubi\v cka, Kone\v cn\'y and Ne\v set\v ril generalise all these results and discuss the Ramsey property and EPPA of a broad family of classes for which there is a commutative ordered semigroup $\mathfrak M$ satisfying some further axioms, such that the incomplete structures with completion to these classes can be completed by the shortest path algorithm, which uses the semigroup operation instead of $+$ and the semigroup order instead of the standard order of the reals. See also~\cite{Konecny2018b}.

It turns out that for a primitive 3-constrained class and a magic parameter $M$, one can define a commutative and associative operation $\oplus^M\colon  [\delta]^2 \rightarrow [\delta]$ as
$$x \oplus^M y =
  \begin{cases}
    |x-y| & \text{if } |x-y| > M \\
    \min\left(x+y, C-1-x-y\right) & \text{if } \min\left(\ldots\right) < M \\
    M & \text{otherwise}\end{cases}$$
and the partial order $\preceq^M$ as $x\preceq^M y$ if and only if $x=y$ or there is $z$ such that $x\oplus z = y$ (this is called the \textit{natural order of $\oplus^M$}).

One can observe that the order induced by $t_M(x)$ is an extension of $\preceq$. And in this setting, we basically proved in this paper the following proposition (cf. Theorem~\ref{thm:magiccompletion}):
\begin{prop}
Let $\str G = (G,d) \in \mathcal G^d$ such that it has a completion in $\mathcal A^\delta_{K_1,K_2,C_0,C_1}$. Let $M$ be a magic distance. Define $\bar d \colon  V^2 \rightarrow [\delta]$ as
$$\bar d(x,y) = 
  \begin{cases}
    d(x,y) & \text{if } d(x,y)\text{ defined} \\
    \min\limits_{P \text{ path from $x$ to $y$}} \bigoplus^M_i d(P_i, P_{i+1}) & \text{otherwise}
  \end{cases}$$
(minimum is taken in the $\preceq^M$ order).

Then $\overbar{\str G} = (V, \bar d)$ is in $\mathcal A^\delta_{K_1,K_2,C_0,C_1}$.
\end{prop}

In~\cite{Hubicka2017sauer} the interpretation of some metrically homogeneous graphs as generalised semigroup-valued metric spaces is discussed and their Ramsey property and EPPA then follow from a more general result (which, though, still depends on the fact that the magic completion works for these classes). See also~\cite{Konecny2018b}.

\paragraph{5} One of our original motivations for investigating the problem
was the problem stated by Lionel Nguyen Van Th{\'e} about the Ramsey expansion of the class of
affinely independent Euclidean metric spaces~\cite{The2010}. Our techniques
do not seem to generalise to this setting, however it seems more clear that
the lack of (local) canonical amalgamation is one of the main obstacles.  It would be interesting to identify
classes of a more combinatorial nature which also expose such problems as an
additional step in this direction.

\paragraph{6} Stationary independence relations have been an ingredient for showing simplicity
of automorphism groups of some homogeneous structures~\cite{Tent2013,Evans2016}. After the submission of this paper, Evans, Hubička, Konečný, Li, and Ziegler~\cite{Evanssimplicity} showed that the existence of a stationary independence relation satisfying some extra (which are indeed satisfied by the stationary independence relations defined in this paper for the finite-diameter primitive 3-constrained case with Henson constraints) implies simplicity of the automorphism group. However, for the other cases the following question remains open:

\begin{question}
What are the normal subgroups of the automorphism groups of the non-tree-like countably infinite metrically homogeneous graphs from Cherlin's catalogue?
\end{question}

In fact, we believe that a solution to this question is within reach by refining the analysis from~\cite{Evanssimplicity}.

\paragraph{7} Since the tree-like graphs do not have a precompact Ramsey expansion, the KPT correspondence cannot be used as is to argue about amenability of the automorphism groups of the tree-like graphs. $T_{2,2}$ is the infinite path with no endpoints and $\Aut(T_{2,2})$ is isomorphic to $(\mathbb Z,+)$ which is amenable. We conjecture that this is the only amenable case:
\begin{conjecture}\label{conj:treelike}
If $(m,n)\neq (2,2)$ then $\Aut(T_{m,n})$ is not amenable.
\end{conjecture}
Our intuition for conjecturing this is the following: If $m=2$ and $n<\infty$, one can consider a graph $T'$ whose vertices are the blocks of $T_{2,n}$ and two vertices are connected by an edge if and only if the blobs have a non-empty intersection. Then $T'$ is the regular infinite $n$-ary tree and the action of $\Aut(T_{2,n})$ on the blocks of $T_{2,n}$ should give a continuous surjection onto $\Aut(T')$. In general, if $m,n\leq \infty$, one can do a similar construction with the only complication being that $T'$ is additionally equipped with an equivalence relation on the neighbourhood of every vertex corresponding to several blocks intersecting in the same vertex. Regarding the infinite parameter cases, $T_{\infty,2}$ is just the infinitely branching tree and one should be able to use arguments from~\cite{Evans2} to prove non-amenability of its automorphism group. (Or maybe it is a known result?)

\paragraph{8} We have proved various results for the class of finite metric subspaces of the associated metric spaces of metrically homogeneous graphs. However, one might want to ask these questions in terms of finite graphs and their isometric embeddings. For example, for every finite metric space $\str A$ with integer distances there exists a finite graph $\str G$ whose associated metric space contains $\str A$ as a substructure. Moreover, this finite graph can be constructed canonically so that every automorphism of $\str A$ extends to an automorphism of the associated metric space of $\str G$ (one simply adds a disjoint path of length $k$ for every edge of distance $k$).

This becomes more interesting once one considers $\str A$ from some $\mathcal A^\delta_{K_1,K_2,C_0,C_1}$ and also requires that the associated metric space of $\str G$ belongs to $\mathcal A^\delta_{K_1,K_2,C_0,C_1}$:

\begin{question}\label{q:dist_finite}
Given a choice of admissible primitive parameters $\delta$, $K_1$, $K_2$, $C_0$ and $C_1$, and $\str A\in \mathcal A^\delta_{K_1,K_2,C_0,C_1}$, does there exist a finite graph $\str G$ with associated metric space $\str B\in \mathcal A^\delta_{K_1,K_2,C_0,C_1}$ such that $\str B$ contains $\str A$ as a substructure?
\end{question}
\begin{question}\label{q:dist_finte_eppa}
In the setting of Question~\ref{q:dist_finite}, can one require that every automorphism of $\str A$ extends to an automorphism of $\str B$?
\end{question}

A positive answer to Question~\ref{q:dist_finite} was sketched by Cherlin for the unconstrained cases for every $\delta>1$ (i.e., only non-metric triangles are forbidden) and also for all cases with $\delta=3$.\footnote{Gregory Cherlin, personal communication.}

This question can be traced back to Moss~\cite{Moss1992} who studied \emph{distanced graphs} (that is, graphs with isometric embeddings) and in particular asked which countable distanced graphs $\str G$ are \emph{distance finite} (that is, for every finite metric subspace $(X,d)$ of the associated metric space of $\str G$ there is a finite subgraph $\str Y$ of $\str G$ which embeds isometrically into $\str G$ such that $X\subseteq Y$). He proved that the metrically homogeneous graph corresponding to the class $\mathcal A^\infty_{1,\infty,\infty,\infty}$ (the class of all finite metric spaces with integer distances) is distance finite (the argument was sketched above) and that there exists a non-distance-finite countable distanced graph which, however, is not (metrically) homogeneous. On page 299 he remarks that distance finiteness is open even for \emph{distance homogeneous graphs} (metrically homogeneous graphs in our language) -- Question~\ref{q:dist_finite} is a special case of this remark. Moss' question can be phrased in today's terms as the first part of the following question:

\begin{question}[Moss~\cite{Moss1992}]\label{q:moss}
Given a countable metrically homogeneous graph $\Gamma$ and a finite subset $X\subseteq \Gamma$, does there exist a finite $Y\subset \Gamma$ such that $X\subseteq \Gamma$ and the graph induced by $\Gamma$ on $Y$ embeds isometrically to $\Gamma$? If this is the case, can one construct $Y$ canonically so that every isometry $X\to X$ extends to an isometry $Y\to Y$?
\end{question}

Let us remark that Question~\ref{q:dist_finite} is related to the \emph{finite model property}. More precisely, this is a special case of it for formulas describing that a given finite graph is isometrically realised and constraints of the ambient class are satisfied in the associated metric space. Even for triangle-free graphs (which would correspond to the class $\mathcal A^2_{2,4,8,7}$ if $\delta=2$ wasn't explicitly excluded in Definition~\ref{defn:acceptable} for practical reasons) the finite model property is wide open. See e.g.~\cite{Cherlin2011b} for an overview and~\cite{EvenZohar2015} for the currently strongest result in this direction.

\paragraph{9} Besides Ramsey classes, there is also a (currently very active) area studying \emph{big Ramsey degrees}, where one studies the following question: Given an (infinite) structure $\str B$ and its finite substructure $\str A$, is there a finite number $t$ such that for every finite colouring of embeddings from $\str A$ into $\str B$, one can find an embedding of $\str B$ into itself which only attains at most $t$ colours? The least such $t$ is called the \emph{big Ramsey degree of $\str A$ in $\str B$}, and we say that $\str B$ \emph{has finite big Ramsey degrees} if all its finite substructures have a finite big Ramsey degree in $\str B$. For an overview of the area, see e.g.~\cite{dobrinen2021ramsey,hubicka2024survey}.

In~\cite{balko2021big}, Balko, Chodounsk{\'y}, Hubi{\v{c}}ka, Kone{\v{c}}n{\'y}, Ne{\v{s}}et{\v{r}}il, and Vena give a condition which, together with another paper from this volume by Hubička, Kompatscher, and Ko\-neč\-ný~\cite{Hubickacycles2018}, implies that the \Fraisse{} limits of the finite-diameter non-antipodal 3-constrained classes from Cherlin's catalogue have finite big Ramsey degrees. It seems that a complete characterization of finiteness of big Ramsey degrees for \Fraisse{} limits of classes from Cherlin's catalogue is within reach of the current methods. Given the role that the results and methods of this paper played in the more recent developments of structural Ramsey theory, we believe that a big Ramsey counterpart of this paper is desirable.

\section{Acknowledgements}
All authors would like to thank Gregory Cherlin for the catalogue, for engaging discussions concerning
finite set of obstacles of metrically homogeneous graphs, and for many helpful comments regarding the paper, history of the area as well as future of the area.

A significant part of the research was done when the authors were participating in the Ramsey DocCourse programme, in Prague 2016--2017.
We are grateful to Lionel Nguyen Van Th{\'e} for the encouragement to include this problem 
for the afternoon problem-solving sessions of the DocCourse. We would also like to thank David M. Evans and Andy Zucker
for discussions regarding non-amenability of the tree-like case. The seventh author would like to thank Jacob Rus for his Python style guidance. We would also like to thank Colin Jahel for pointing out an omission in a statement of our results.

Shortly before this paper was submitted we were informed
that similar results on finite set of obstacles, Ramsey property and ample
generics were also obtained by Rebecca Coulson and Gregory Cherlin which later appeared in~\cite{Coulson}.
The Ramsey expansion of space $\mathcal A^3_{1,3,8,9}$
was also independently obtained by Miodrag Soki\'c~\cite{Sokic2017}.
\bibliography{ramsey.bib}

\newcommand{\etalchar}[1]{$^{#1}$}
\begin{thebibliography}{JLNVTW14}

\bibitem[ABWH{\etalchar{+}}21]{Aranda2017c}
Andres Aranda, David Bradley-Williams, Eng~Keat Hng, Jan Hubi{\v c}ka,
  Miltiadis Karamanlis, Michael Kompatscher, Mat{\v e}j Kone{\v c}n{\'y}, and
  Micheal Pawliuk.
\newblock Completing graphs to metric spaces.
\newblock {\em Contributions to Discrete Mathematics}, 16:71--89, 2021.

\bibitem[ACM21]{Amato2016}
Daniela~A. Amato, Gregory Cherlin, and Dugald Macpherson.
\newblock Metrically homogeneous graphs of diameter 3.
\newblock {\em Journal of Mathematical Logic}, 21(01):2050020, 2021.

\bibitem[AH78]{Abramson1978}
Fred~G. Abramson and Leo~A. Harrington.
\newblock Models without indiscernibles.
\newblock {\em Journal of Symbolic Logic}, 43:572--600, 1978.

\bibitem[AKL14]{AKL14}
Omer Angel, Alexander~S. Kechris, and Russell Lyons.
\newblock Random orderings and unique ergodicity of automorphism groups.
\newblock {\em Journal of the European Math. Society}, 16:2059--2095, 2014.

\bibitem[BCH{\etalchar{+}}21]{balko2021big}
Martin Balko, David Chodounsk{\'y}, Jan Hubi{\v{c}}ka, Mat{\v{e}}j
  Kone{\v{c}}n{\'y}, Jaroslav Ne{\v{s}}et{\v{r}}il, and Llu{\'\i}s Vena.
\newblock Big {R}amsey degrees and forbidden cycles.
\newblock In Jaroslav Ne{\v{s}}et{\v{r}}il, Guillem Perarnau, Juanjo Ru{\'e},
  and Oriol Serra, editors, {\em Extended Abstracts EuroComb 2021}, pages
  436--441. Springer International Publishing, 2021.

\bibitem[Bod15]{Bodirsky2015}
Manuel Bodirsky.
\newblock Ramsey classes: Examples and constructions.
\newblock {\em Surveys in Combinatorics 2015}, 424:1, 2015.

\bibitem[Bra26]{Sam}
Samuel Braunfeld.
\newblock Ramsey expansions of {$\Lambda$}-ultrametric spaces.
\newblock {\em European Journal of Combinatorics}, 132:in press,
  arXiv:1710.01193, 2026.

\bibitem[Cam80]{Cameron1980}
Peter~J. Cameron.
\newblock 6-transitive graphs.
\newblock {\em Journal of Combinatorial Theory, Series B}, 28(2):168--179,
  1980.

\bibitem[Che98]{Cherlin1998}
Gregory Cherlin.
\newblock {\em The Classification of Countable Homogeneous Directed Graphs and
  Countable Homogeneous $N$-tournaments}.
\newblock Number 621 in Memoirs of the American Mathematical Society. American
  Mathematical Society, 1998.

\bibitem[Che11]{Cherlin2011b}
Gregory Cherlin.
\newblock Two problems on homogeneous structures, revisited.
\newblock {\em Model theoretic methods in finite combinatorics}, 558:319--415,
  2011.

\bibitem[Che22]{Cherlin2013}
Gregory Cherlin.
\newblock {\em Homogeneous Ordered Graphs, Metrically Homogeneous Graphs, and
  Beyond}, volume~1 of {\em Lecture Notes in Logic}.
\newblock Cambridge University Press, 2022.

\bibitem[Con19]{Conant2015}
Gabriel Conant.
\newblock Extending partial isometries of generalized metric spaces.
\newblock {\em Fundamenta Mathematicae}, 244:1--16, 2019.

\bibitem[Cou19]{Coulson}
Rebecca Coulson.
\newblock {\em Metrically Homogeneous Graphs: Dynamical Properties of Their
  Automorphism Groups and the Classification of Twists}.
\newblock PhD thesis, Rutgers Mathematics Department, 2019.

\bibitem[Dob23]{dobrinen2021ramsey}
Natasha Dobrinen.
\newblock Ramsey theory of homogeneous structures: current trends and open
  problems.
\newblock In {\em International Congress of Mathematicians}, pages 1462--1486.
  EMS Press, 2023.

\bibitem[EGT16]{Evans2016}
David~M. Evans, Zaniar Ghadernezhad, and Katrin Tent.
\newblock Simplicity of the automorphism groups of some {H}rushovski
  constructions.
\newblock {\em Annals of Pure and Applied Logic}, 167(1):22--48, 2016.

\bibitem[EHK{\etalchar{+}}21]{Evanssimplicity}
David~M. Evans, Jan Hubi{\v{c}}ka, Mat{\v {e}}j Kone{\v {c}}n{\'{y}}, Yibei Li,
  and Martin Ziegler.
\newblock Simplicity of the automorphism groups of generalised metric spaces.
\newblock {\em Journal of Algebra}, 584:163--179, 2021.

\bibitem[EHKN20]{eppatwographs}
David~M. Evans, Jan Hubi{\v{c}}ka, Mat{\v {e}}j Kone{\v {c}}n{\'{y}}, and
  Jaroslav Ne\v{s}et\v{r}il.
\newblock E{P}{P}{A} for two-graphs and antipodal metric spaces.
\newblock {\em Proceedings of the American Mathematical Society},
  148:1901--1915, 2020.

\bibitem[EHN19]{Evans2}
David~M. Evans, Jan Hubi{\v c}ka, and Jaroslav Ne{\v{s}}et{\v{r}}il.
\newblock Automorphism groups and {R}amsey properties of sparse graphs.
\newblock {\em Proceedings of the London Mathematical Society},
  119(2):515--546, 2019.

\bibitem[EHN21]{Evans3}
David~M. Evans, Jan Hubi{\v{c}}ka, and Jaroslav Ne{\v{s}}et{\v{r}}il.
\newblock Ramsey properties and extending partial automorphisms for classes of
  finite structures.
\newblock {\em Fundamenta Mathematicae}, 253:121--153, 2021.

\bibitem[EZL15]{EvenZohar2015}
Chaim Even-Zohar and Nati Linial.
\newblock Triply existentially complete triangle-free graphs.
\newblock {\em Journal of Graph Theory}, 78(4):305--317, 2015.

\bibitem[GR01]{Godsil2001}
Chris {Godsil} and Gordon {Royle}.
\newblock {\em Algebraic Graph Theory}, volume 207 of {\em Graduate Texts in
  Mathematics.}
\newblock volume 207 of Graduate Texts in Mathematics. Springer, 2001.

\bibitem[Her95]{Herwig1995}
Bernhard Herwig.
\newblock Extending partial isomorphisms on finite structures.
\newblock {\em Combinatorica}, 15(3):365--371, 1995.

\bibitem[Her98]{herwig1998}
Bernhard Herwig.
\newblock Extending partial isomorphisms for the small index property of many
  $\omega$-categorical structures.
\newblock {\em Israel Journal of Mathematics}, 107(1):93--123, 1998.

\bibitem[HHLS93]{hodges1993b}
Wilfrid Hodges, Ian Hodkinson, Daniel Lascar, and Saharon Shelah.
\newblock The small index property for $\omega$-stable $\omega$-categorical
  structures and for the random graph.
\newblock {\em Journal of the London Mathematical Society}, 2(2):204--218,
  1993.

\bibitem[Hig52]{higman1952ordering}
Graham Higman.
\newblock Ordering by divisibility in abstract algebras.
\newblock {\em Proceedings of the London Mathematical Society}, 3(1):326--336,
  1952.

\bibitem[HK26]{hubicka2025twenty}
Jan Hubi{\v{c}}ka and Mat{\v{e}}j Kone{\v{c}}n{\'y}.
\newblock Twenty years of {N}e{\v s}et{\v r}il's classification programme of
  {R}amsey classes.
\newblock {\em Computer Science Review}, 59:100814, 2026.

\bibitem[HKK26]{Hubickacycles2018}
Jan Hubi{\v c}ka, Michael Kompatscher, and Mat{\v e}j Kone{\v c}n{\'y}.
\newblock Forbidden cycles in metrically homogeneous graphs.
\newblock {\em European Journal of Combinatorics}, 132:in press,
  arXiv:1808.05177, 2026.

\bibitem[HKN18]{Hubicka2017sauer}
Jan Hubi{\v{c}}ka, Mat{\v {e}}j Kone{\v {c}}n{\'{y}}, and Jaroslav
  Ne\v{s}et\v{r}il.
\newblock Semigroup-valued metric spaces: {R}amsey expansions and {E}{P}{P}{A}.
\newblock In preparation, 2018.

\bibitem[HKN19]{Hubicka2018metricEPPA}
Jan Hubi{\v{c}}ka, Mat{\v{e}}j Kone{\v{c}}n{\'y}, and Jaroslav
  Ne{\v{s}}et{\v{r}}il.
\newblock A combinatorial proof of the extension property for partial
  isometries.
\newblock {\em Commentationes Mathematicae Universitatis Carolinae},
  60(1):39--47, 2019.

\bibitem[HKN21]{Hubicka2017sauerconnant}
Jan Hubi{\v{c}}ka, Mat{\v{e}}j Kone{\v{c}}n{\'y}, and Jaroslav
  Ne{\v{s}}et{\v{r}}il.
\newblock Conant's generalised metric spaces are {R}amsey.
\newblock {\em Contributions to Discrete Mathematics}, 16:46--70, 2021.

\bibitem[HKN22]{Hubicka2018EPPA}
Jan Hubi{\v{c}}ka, Mat{\v{e}}j Kone{\v{c}}n{\'{y}}, and Jaroslav
  Ne{\v{s}}et{\v{r}}il.
\newblock All those {E}{P}{P}{A} classes (strengthenings of the
  {H}erwig--{L}ascar theorem).
\newblock {\em Transactions of the American Mathematical Society},
  375(11):7601--7667, 2022.

\bibitem[HL00]{herwig2000}
Bernhard Herwig and Daniel Lascar.
\newblock Extending partial automorphisms and the profinite topology on free
  groups.
\newblock {\em Transactions of the American Mathematical Society},
  352(5):1985--2021, 2000.

\bibitem[HN19]{Hubicka2016}
Jan Hubi{\v{c}}ka and Jaroslav Ne\v{s}et\v{r}il.
\newblock All those {R}amsey classes ({R}amsey classes with closures and
  forbidden homomorphisms).
\newblock {\em Advances in Mathematics}, 356C:106791, 2019.

\bibitem[HO03]{hodkinson2003}
Ian Hodkinson and Martin Otto.
\newblock Finite conformal hypergraph covers and {G}aifman cliques in finite
  structures.
\newblock {\em Bulletin of Symbolic Logic}, 9(03):387--405, 2003.

\bibitem[Hod93]{Hodges1993}
Wilfrid Hodges.
\newblock {\em Model theory}, volume~42.
\newblock Cambridge University Press, 1993.

\bibitem[Hru92]{hrushovski1992}
Ehud Hrushovski.
\newblock Extending partial isomorphisms of graphs.
\newblock {\em Combinatorica}, 12(4):411--416, 1992.

\bibitem[HZ25]{hubicka2024survey}
Jan Hubi{\v{c}}ka and Andy Zucker.
\newblock A survey on big {R}amsey structures.
\newblock {\em Zbornik Radova (Beograd)}, 22(30):317--355, 2025.
\newblock Selected topics in combinatorial analysis II.

\bibitem[JLNVTW14]{Jasinski2013}
Jakub Jasi{\'n}ski, Claude Laflamme, Lionel Nguyen Van~Th{\'e}, and Robert
  Woodrow.
\newblock Ramsey precompact expansions of homogeneous directed graphs.
\newblock {\em The Electronic Journal of Combinatorics}, 21(4):\#4.42, 2014.

\bibitem[Kon19]{Konecny2018b}
Mat{\v e}j Kone{\v c}n{\'y}.
\newblock Semigroup-valued metric spaces.
\newblock Master's thesis, Charles University, 2019.
\newblock arXiv:1810.08963.

\bibitem[Kon20]{Konecny2019a}
Mat{\v e}j Kone{\v c}n{\'y}.
\newblock Extending partial isometries of antipodal graphs.
\newblock {\em Discrete Mathematics}, 343(1):111633, 2020.

\bibitem[KPT05]{Kechris2005}
Alexander~S. Kechris, Vladimir~G. Pestov, and Stevo Todor{\v c}evi{\' c}.
\newblock Fra{\"\i}ss{\'e} limits, {R}amsey theory, and topological dynamics of
  automorphism groups.
\newblock {\em Geometric and Functional Analysis}, 15(1):106--189, 2005.

\bibitem[KR07]{Kechris2007}
Alexander~S. Kechris and Christian Rosendal.
\newblock Turbulence, amalgamation, and generic automorphisms of homogeneous
  structures.
\newblock {\em Proceedings of the London Mathematical Society}, 94(2):302--350,
  2007.

\bibitem[Lac84]{Lachlan1984}
Alistair~H. Lachlan.
\newblock On countable stable structures which are homogeneous for a finite
  relational language.
\newblock {\em Israel Journal of Mathematics}, 49(1-3):69--153, 1984.

\bibitem[LW80]{Lachlan1980}
Alistair~H. Lachlan and Robert~E. Woodrow.
\newblock Countable ultrahomogeneous undirected graphs.
\newblock {\em Transactions of the American Mathematical Society},
  262(1):51--94, 1980.

\bibitem[Ma{\v{s}}18]{masulovic2016pre}
Dragan Ma{\v{s}}ulovi{\'c}.
\newblock Pre-adjunctions and the {R}amsey property.
\newblock {\em European Journal of Combinatorics}, 70:268--283, 2018.

\bibitem[Mos92]{Moss1992}
Lawrence~S. Moss.
\newblock Distanced graphs.
\newblock {\em Discrete Mathematics}, 102(3):287--305, 1992.

\bibitem[M{\"u}l16]{Muller2016}
Isabel M{\"u}ller.
\newblock Fra{\"\i}ss{\'e} structures with universal automorphism groups.
\newblock {\em Journal of Algebra}, 463:134--151, 2016.

\bibitem[Ne{\v{s}}89]{Nevsetvril1989a}
Jaroslav Ne{\v{s}}et{\v{r}}il.
\newblock For graphs there are only four types of hereditary {R}amsey classes.
\newblock {\em Journal of Combinatorial Theory, Series B}, 46(2):127--132,
  1989.

\bibitem[Ne{\v{s}}95]{Nevsetvril1995}
Jaroslav Ne{\v{s}}et{\v{r}}il.
\newblock Ramsey theory.
\newblock In Ronald~L. Graham, Martin Gr\"{o}tschel, and L\'{a}szl\'{o}
  Lov\'{a}sz, editors, {\em Handbook of Combinatorics}, volume~2, pages
  1331--1403. MIT Press, Cambridge, MA, USA, 1995.

\bibitem[Ne{\v{s}}05]{Nevsetril2005}
Jaroslav Ne{\v{s}}etril.
\newblock Ramsey classes and homogeneous structures.
\newblock {\em Combinatorics, probability and computing}, 14(1-2):171--189,
  2005.

\bibitem[Ne{\v{s}}07]{Nevsetvril2007}
Jaroslav Ne{\v{s}}et{\v{r}}il.
\newblock Metric spaces are {R}amsey.
\newblock {\em European Journal of Combinatorics}, 28(1):457--468, 2007.

\bibitem[NR76]{Nevsetvril1976}
Jaroslav Ne{\v{s}}et{\v{r}}il and Vojt{\v{e}}ch R{\"o}dl.
\newblock The {R}amsey property for graphs with forbidden complete subgraphs.
\newblock {\em Journal of Combinatorial Theory, Series B}, 20(3):243--249,
  1976.

\bibitem[NR78]{Nesetvril1978}
Jaroslav Ne{\v{s}}et{\v{r}}il and Vojt{\v{e}}ch R{\"o}dl.
\newblock On a probabilistic graph-theoretical method.
\newblock {\em Proceedings of the American Mathematical Society},
  72(2):417--421, 1978.

\bibitem[NR17]{nevsetvril2017statistics}
Jaroslav Ne{\v{s}}et{\v{r}}il and Vojt{\v{e}}ch R{\"o}dl.
\newblock Statistics of orderings.
\newblock {\em Abhandlungen aus dem Mathematischen Seminar der Universit{\"a}t
  Hamburg}, 87(2):421--433, 2017.

\bibitem[NVT10]{The2010}
Lionel Nguyen Van~Th{\'e}.
\newblock {\em Structural {R}amsey Theory of Metric Spaces and Topological
  Dynamics of Isometry Groups}.
\newblock Memoirs of the American Mathematical Society. American Mathematical
  Society, 2010.

\bibitem[NVT15]{NVT14}
Lionel Nguyen Van~Th{\'e}.
\newblock {A survey on structural {R}amsey theory and topological dynamics with
  the {K}echris--{P}estov--{T}odorcevic correspondence in mind}.
\newblock {\em Zbornik Radova (Beograd)}, 17(25):189--207, 2015.
\newblock Selected topics in combinatorial analysis.

\bibitem[Paw17]{PawliukSage2}
Micheal Pawliuk.
\newblock Metric graph completer.
\newblock implementation of the completion algorithm as SAGE script:
  https://github.com/mpawliuk/Metric-graphs, 2017.

\bibitem[PS20]{PawliukSokic16}
Micheal Pawliuk and Miodrag Soki{\'c}.
\newblock Amenability and unique ergodicity of automorphism groups of countable
  homogeneous directed graphs.
\newblock {\em Ergodic Theory and Dynamical Systems}, 40(5):1351--1401, 2020.

\bibitem[Ros11a]{rosendal2011b}
Christian Rosendal.
\newblock Finitely approximable groups and actions part {I}{I}: {G}eneric
  representations.
\newblock {\em The Journal of Symbolic Logic}, 76(04):1307--1321, 2011.

\bibitem[Ros11b]{rosendal2011}
Christian Rosendal.
\newblock Finitely approximate groups and actions part {I}: The
  {R}ibes--{Z}alesski{\u\i} property.
\newblock {\em The Journal of Symbolic Logic}, 76(04):1297--1306, 2011.

\bibitem[Sau13]{Sauer2013}
Norbert~W. Sauer.
\newblock Distance sets of {U}rysohn metric spaces.
\newblock {\em Canadian Journal of Mathematics}, 65(1):222--240, 2013.

\bibitem[Sei73]{Seidel1973}
Johan~Jacob Seidel.
\newblock A survey of two-graphs.
\newblock {\em Colloquio Internazionale sulle Teorie Combinatorie (Rome,
  1973)}, 1:481--511, 1973.

\bibitem[Sin17]{Siniora2}
Daoud Siniora.
\newblock {\em Automorphism Groups of Homogeneous Structures}.
\newblock PhD thesis, University of Leeds, March 2017.

\bibitem[Sok20]{Sokic2017}
Miodrag Soki{\'c}.
\newblock Diameter 3.
\newblock {\em Contributions to Discrete Mathematics}, 15:27--41, 2020.

\bibitem[Sol05]{solecki2005}
S{\l}awomir Solecki.
\newblock Extending partial isometries.
\newblock {\em Israel Journal of Mathematics}, 150(1):315--331, 2005.

\bibitem[Sol09]{solecki2009}
S{\l}awomir Solecki.
\newblock Notes on a strengthening of the {H}erwig--{L}ascar extension theorem.
\newblock Unpublished note, 2009.

\bibitem[SS20]{Siniora}
Daoud Siniora and S{\l}awomir Solecki.
\newblock Coherent extension of partial automorphisms, free amalgamation, and
  automorphism groups.
\newblock {\em The Journal of Symbolic Logic}, 85(1):199--223, 2020.

\bibitem[TZ13]{Tent2013}
Katrin Tent and Martin Ziegler.
\newblock On the isometry group of the {U}rysohn space.
\newblock {\em Journal of the London Mathematical Society}, 87(1):289--303,
  2013.

\bibitem[Ver08]{vershik2008}
Anatoly~M. Vershik.
\newblock Globalization of the partial isometries of metric spaces and local
  approximation of the group of isometries of {U}rysohn space.
\newblock {\em Topology and its Applications}, 155(14):1618--1626, 2008.

\end{thebibliography}

\end{document}